\documentclass[a4paper,11pt]{amsart}
\pdfoutput=1
\usepackage[english]{babel}
\usepackage{amstext, amsfonts, amsthm, amssymb}
\usepackage{enumerate}
\usepackage{mathtools}
\usepackage{tikz-cd}
\usepackage[all]{xy}
\usepackage{graphicx, adjustbox}
\usepackage{extpfeil}
\usepackage{comment}
\usepackage{float}
\usepackage[labelformat=empty]{caption}
\usepackage[toc]{appendix}
\usepackage[left=3cm,top=3cm,right=3cm,bottom=2.5cm]{geometry}
\usepackage{hyperref}
\usepackage{resizegather}
\usepackage[activate={true, nocompatibility}, final, tracking=true, kerning=true, spacing=true, factor=1100, stretch=10, shrink=10]{microtype}
\usepackage{xcolor}
\usepackage{cleveref}

\theoremstyle{plain}
\newtheorem{thm}{Theorem}[section]
\newtheorem{coroll}[thm]{Corollary}
\newtheorem{defn}[thm]{Definition}
\newtheorem{lemma}[thm]{Lemma}

\newtheorem{example}[thm]{Example}
\newtheorem{prop}[thm]{Proposition}

\newtheorem{remark}[thm]{Remark}

\newtheorem{theorem}{Theorem}

\DeclareMathOperator{\Bun}{Bun}
\DeclareMathOperator{\Aff}{Aff}
\DeclareMathOperator{\Spec}{Spec}
\DeclareMathOperator{\Sym}{Sym}
\DeclareMathOperator{\Iso}{Iso}
\DeclareMathOperator{\Red}{Red}
\DeclareMathOperator{\rk}{rk}

\DeclareMathOperator{\res}{res}
\DeclareMathOperator{\Forget}{Forget}
\DeclareMathOperator{\Gr}{Gr}
\DeclareMathOperator{\Hom}{Hom}
\DeclareMathOperator{\Coh}{Coh}

\DeclareMathOperator{\Sch}{Sch}
\DeclareMathOperator{\Filt}{Filt}
\DeclareMathOperator{\Rees}{Rees}
\DeclareMathOperator{\Pic}{Pic}
\DeclareMathOperator{\Sect}{Sect}
\DeclareMathOperator{\Gal}{Gal}
\DeclareMathOperator{\wgt}{wt}
\DeclareMathOperator{\Frac}{Frac}
\DeclareMathOperator{\RelGies}{RelGies}
\DeclareMathOperator{\GL}{GL}
\DeclareMathOperator{\SL}{SL}

\DeclareMathOperator{\SO}{SO}
\DeclareMathOperator{\Int}{Int}

\newcommand{\bseries}[1]{ [\hspace{-0,5mm}[ {#1} ]\hspace{-0,5mm}] } 
\newcommand{\pseries}[1]{ (\hspace{-0,7mm}( {#1} )\hspace{-0,7mm}) }

\makeatletter
\newcommand{\colim@}[2]{%
  \vtop{\m@th\ialign{##\cr
    \hfil$#1\operator@font colim$\hfil\cr
    \noalign{\nointerlineskip\kern1.5\ex@}#2\cr
    \noalign{\nointerlineskip\kern-\ex@}\cr}}%
}
\newcommand{\colim}{%
  \mathop{\mathpalette\colim@{\rightarrowfill@\textstyle}}\nmlimits@
}
\makeatother

\makeatletter
\DeclareRobustCommand
  \myvdots{\vbox{\baselineskip4\p@ \lineskiplimit\z@
    \hbox{.}\hbox{.}\hbox{.}}}
\makeatother

\tikzset{
  symbol/.style={
    draw=none,
    every to/.append style={
      edge node={node [sloped, allow upside down, auto=false]{$#1$}}}
  }
}
\makeatletter
\tikzset{
  column sep/.code=\def\pgfmatrixcolumnsep{\pgf@matrix@xscale*(#1)},
  row sep/.code   =\def\pgfmatrixrowsep{\pgf@matrix@yscale*(#1)},
  matrix xscale/.code=%
    \pgfmathsetmacro\pgf@matrix@xscale{\pgf@matrix@xscale*(#1)},
  matrix yscale/.code=%
    \pgfmathsetmacro\pgf@matrix@yscale{\pgf@matrix@yscale*(#1)},
  matrix scale/.style={/tikz/matrix xscale={#1},/tikz/matrix yscale={#1}}}
\def\pgf@matrix@xscale{1}
\def\pgf@matrix@yscale{1}
\makeatother

\microtypecontext{spacing=nonfrench}

\usepackage[normalem]{ulem}
\usepackage{pbox}

\title{The moduli stack of principal $\rho$-sheaves and Gieseker-Harder-Narasimhan filtrations}

\author[T. L. G\'omez]{Tom\'as L. G\'omez}

\address{Instituto de Ciencias Matem\'aticas (CSIC-UAM-UC3M-UCM),
Nicol\'as Cabrera 15, Campus Cantoblanco UAM, 28049 Madrid, Spain}

\email{tomas.gomez@icmat.es}

\author[A. Fernandez Herrero]{Andres Fernandez Herrero}

\address{Mathematics Hall MC 4404, 2990 Broadway,
New York, N.Y. 10027, USA.}

\email{af3358@columbia.edu}

\author[A. Zamora]{Alfonso Zamora}

\address{Departamento de Matem\'atica Aplicada a las TIC, ETSI Inform\'aticos, Universidad Polit\'ecnica de Madrid, Campus de Montegancedo,
28660 Madrid, Spain}

\email{alfonso.zamora@upm.es}



\begin{document}

\maketitle
\pagestyle{plain}
\begin{abstract}
Let $X$ be a smooth projective variety and let $G$ be a connected reductive group, both defined over a field of characteristic $0$. Given a faithful representation $\rho$ of $G$ into a product of general linear groups, we define a moduli stack of principal $\rho$-sheaves that compactifies the stack of $G$-bundles on $X$. We apply the theory developed in \cite{halpernleistner2021structure, alper2019existence,torsion-freepaper} to construct a moduli space of Gieseker semistable principal $\rho$-sheaves. This provides an intrinsic stack-theoretic construction of the moduli space of semistable singular principal bundles \cite{schmitt.singular, glss.singular.char}. 

Our second main result is the definition of a schematic Gieseker-Harder-Narasimhan filtration for $\rho$-sheaves, which induces a stratification of the stack by locally closed substacks. This filtration for a general reductive group $G$ is a refinement of the canonical slope parabolic reduction previously considered at the level of points in \cite{anchouche-hassan-biswas} and as a stratification of the stack in \cite{nitsuregurjar2, gurjar2020hardernarasimhan}. In an appendix, we apply the same techniques to define Gieseker-Harder-Narasimhan filtrations in arbitrary characteristic and show that they induce a stratification of the stack by radicial morphisms.
\end{abstract}
\tableofcontents

\begin{section}{Introduction}
Fix a smooth projective variety $X$ and a connected reductive group $G$ defined over a field of characteristic $0$ (see Appendix \ref{appendix: positive char} for positive characteristic). When $\dim X=1$, Ramanathan \cite{ramanathan-stable} constructed a projective moduli space of semistable
principal $G$-bundles on $X$ using Geometric Invariant Theory (GIT). The stability condition in this case is analogous to the notion of slope stability for a vector bundle (introduced by Mumford), in the sense that it involves the degree and the rank of certain vector bundles.

If $\dim X>1$, it is necessary to consider also degenerations of principal $G$-bundles in order to construct a proper moduli space. These degenerations are called principal $\rho$-sheaves, where $\rho$ is a choice of representation of $G$. A principal $\rho$-sheaf is a triple $(P,\mathcal{E},\psi)$ consisting of a torsion-free sheaf $\mathcal{E}$ on $X$, a principal $G$-bundle $P$ defined on the open set $U$ where $\mathcal{E}$ is locally free, and an isomorphism $\psi$ between the restriction $\mathcal{E}|_{U}$ and the vector bundle associated to $P$ using the representation $\rho$ (see Definition \ref{defn: rho sheaf} and Proposition \ref{prop: relation scheme and G-reductions}). More generally, we allow $\rho$ to be a faithful homomorphism from $G$ into a product of general linear groups (see Definition \ref{defn:principalsheaf2}, Definition \ref{defn:principalsheaf3}, and Section \ref{section: generalization})\footnote{If we want our definition of stability to coincide with Ramanathan's Definition \ref{defn: ramanathan rational} in the case of curves, then we have to require that the representation $\rho$ is central (Definition \ref{defn: central representation}).}.
These $\rho$-sheaves have also been called in the literature ``singular principal bundles" by Schmitt \cite{schmitt.singular, glss.singular.char} (in the case when $\rho$ is valued in a special linear group), and ``principal $G$-sheaves" by G\'omez-Sols \cite{gomezsols.principalsheaves} (when $\rho$ is the adjoint representation, not injective in general).  We show that there is an algebraic stack $\text{Bun}_{\rho}(X)$ that parametrizes principal $\rho$-sheaves on $X$. 

When going from $\dim X=1$ to $\dim X>1$, an additional consideration needs to be introduced in order to construct the moduli space. Instead of using slope semistability, we have to use a notion where the degree is substituted by the Hilbert polynomial. This is Gieseker semistability for torsion-free sheaves. An analogous notion of semistability was defined in \cite{gomezsols.principalsheaves}  and \cite{glss.singular.char} for principal $\rho$-sheaves. Using the methods of GIT, G\'omez-Sols \cite{gomezsols.principalsheaves}, Schmitt \cite{schmitt.singular} and G\'omez-Langer-Schmitt-Sols \cite{glss.singular.char} constructed moduli spaces of semistable $\rho$-sheaves. We remark that the moduli space obtained will in general depend on the choice of representation $\rho$.
The notion of semistability for principal $\rho$-sheaves is given in Definition \ref{defn: Gieseker semistable}.

Alper, Halpern-Leistner and Heinloth \cite{alper2019existence} have recently introduced an alternative approach to GIT for the construction of moduli spaces. Their approach is intrinsic to the algebraic stack associated to the moduli problem of interest. Building on the criteria of \cite{alper2019existence}, Halpern-Leistner \cite[\S5]{halpernleistner2021structure} and Halpern-Leistner, Herrero and Jones \cite{torsion-freepaper} developed tools to construct moduli spaces of semistable loci inside of algebraic stacks. In this article we employ their techniques instead of the GIT methods that were previously used in the literature.

In this direction, our first main result in this article is an intrinsic-stack theoretic construction of the moduli space of semistable $\rho$-sheaves, where $\rho$ is any faithful representation of $G$ into a product of general linear groups. In order to obtain a quasicompact moduli space, it is helpful to fix a tuple $P^{\bullet}=(P^1, P^2, 
\ldots, P^b)$ of rational polynomials $P^i \in \mathbb{Q}[n]$ 
and restricts to the locus of $\rho$-sheaves where the $i^{th}$ underlying sheaf has Hilbert polynomial $P^i$. We state the result here in the case when the representation $\rho$ is central (Definition \ref{defn: central representation})\footnote{A more general modified statement is proven in the body of the document, see Theorem \ref{thm: main theorem products gen linear groups}.}:
\begin{theorem} \label{thm: main theorem 1 introduction}
Fix a central faithful representation $\rho: G \hookrightarrow \prod_{i} \GL_{n_i}$. Then
\begin{enumerate}[(a)]
    \item The moduli stack $\text{Bun}_{\rho}(X)$ of $\rho$-sheaves admits a $\Theta$-stratification induced by the polynomial numerical invariant $\nu$ defined in Subsection \ref{subsection: numerical invariant}.
    \item (Semistable reduction) The semistable locus $\text{Bun}_{\rho}(X)^{ss}$ satisfies the existence part of the valuative criterion for properness.
    \item For any choice of tuple of rational polynomials $P^{\bullet}$, the open substack $\text{Bun}_{\rho}(X)^{ss, P^{\bullet}}$ of $\rho$-sheaves with tuple $P^{\bullet}$ as Hilbert polynomials admits a proper good moduli space (in the sense of \cite{alper-good-moduli}).
\end{enumerate}
\end{theorem}
Theorem \ref{thm: main theorem 1 introduction} (a) is proven in subsection \ref{subsection: theta stratification} (Theorem \ref{thm: theta stratification}), (b) is proven in subsection \ref{section: properness} (Theorem \ref{thm: semistable reduction}) and (c) is proven in subsection \ref{subsection: boundedness of the semistable locus} (Theorem \ref{thm: moduli space single general linear}).

If the representation $\rho$ is valued in a single special linear group (resp. $\rho$ is the adjoint representation), then we recover the moduli space constructed in \cite{schmitt.singular, glss.singular.char} (resp. \cite{gomezsols.principalsheaves}).

Another important aspect of the theory of semistable vector bundles on a curve is the existence of a canonical filtration for any unstable vector bundle. This is called the Harder-Narasimhan filtration, and it is such that the associated graded object is a direct sum of semistable sheaves. This is helpful because unstable vector bundles are then constructed as extensions of semistable bundles, thus allowing us to do recursive arguments to study properties of the moduli space, as done by Harder-Narasimhan \cite{hn-cohomology-moduli} and Atiyah-Bott \cite{atiyah-bott}. More generally, for principal $G$-bundles on a curve Behrend \cite{behrend-canonical-bundles} and Biswas-Holla \cite{biswas-holla-hnreduction} constructed the Harder-Narasimhan reduction to a parabolic subgroup $P\subset G$. An important property is that the associated principal $L$-bundle is semistable, where $L=G/P$ is the Levi factor of $P$.

For torsion free sheaves over $X$ with $\dim X>1$, we can construct two different Harder-Narasimhan filtrations, depending on the stability condition we use, slope or Gieseker stability. 
They are respectively called slope-Harder-Narasimhan and Gieseker-Harder-Narasimhan.

For principal $\rho$-sheaves over $X$ with $\dim X>1$ we only find in the literature the slope version of the Harder-Narasimhan filtration, constructed at the level of points by Anchouche-Azad-Biswas \cite{anchouche-hassan-biswas} and in families by Gurjar-Nitsure \cite{nitsuregurjar2, gurjar2020hardernarasimhan}. Biswas and Zamora \cite{biswas-zamora-ghn-principal} showed that the notion of Gieseker filtration is not well-behaved with respect to change of group. The strategy of considering the canonical filtration of the adjoint vector bundle does not yield in general a well-defined notion of Gieseker-Harder-Narasimhan reduction in higher dimensions. Therefore a new approach needs to be introduced.

Our second main result (proven in section \ref{section: GHN filtrations} (Theorem \ref{thm: relative Gieseker filtrations})) is the definition of a Gieseker-Harder-Narasimhan parabolic reduction in the context of $\rho$-sheaves.
\begin{theorem} \label{thm: main theorem 2 introduction}
Every principal $\rho$-sheaf admits a uniquely defined multi-weighted filtration called the Gieseker-Harder-Narasimhan filtration (abbreviated GHN filtration), see Definition \ref{defn: GHN filtration}. 

This filtration is schematic, in the sense that the universal stacks for relative GHN filtrations induce a locally closed stratification of $\text{Bun}_{\rho}(X)$.

If the group $G$ is split, then the GHN filtration comes from a multi-weighted parabolic reduction defined over a big open subset (cf. Definition \ref{defn: big open subset}). The associated Levi $\rho_{\vec{\lambda}}$-sheaf is semistable.
\end{theorem}

If the representation $\rho$ is central, then our GHN filtration refines the canonical slope reduction constructed in \cite{anchouche-hassan-biswas}. In this case, the resulting stratification of the stack $\text{Bun}_{\rho}(X)$ refines the canonical slope stratification on the stack of $G$-bundles that was recently constructed in \cite{nitsuregurjar2}.

Our main tools originate in some recent work on the theory of stacks. Alper, Halpern-Leistner and Heinloth \cite{alper2019existence} have recently developed a theory which produces moduli spaces for Artin stacks, generalizing results of Keel-Mori on Deligne-Mumford stacks. This can be combined with the theory of $\Theta$-stability and $\Theta$-stratifications by Halpern-Leistner \cite{halpernleistner2021structure} to produce a powerful stack-theoretic approach to the construction of moduli spaces and canonical filtrations (see \cite[\S 5]{halpernleistner2021structure} and \cite[\S 2.5]{torsion-freepaper} for more details). Halpern-Leistner defines in \cite{halpernleistner2021structure} a notion of filtration of an object in an Artin stack. Given certain cohomology classes on the stack, one can define a numerical invariant on filtrations. An object in the stack is called semistable if the numerical function is non-positive for all filtrations. If the numerical invariant satisfies certain properties, then the open substack of semistable objects admits a \emph{good moduli space} (as defined in \cite{alper-good-moduli}). This is an intrinsic way of constructing the moduli space, in the sense that we do not need to choose a parameter space nor the action of a group.

Here are the main steps involved in this general approach:
\begin{enumerate}[(Step 1)]
    \item Define a (polynomial) numerical invariant $\nu$ on the stack (Subsection \ref{subsection: numerical invariant}).
    \item Prove that the numerical invariant $\nu$ satisfies some monotonicity properties (Subsection \ref{subsection: monotonicity properties of the numerical invariant}).
    \item Show that $\nu$ satisfies the HN boundedness condition (Subsection \ref{subsection: hn boundedness}).
    \item Prove boundedness of the semistable locus (Subsection \ref{subsection: boundedness of the semistable locus}).
\end{enumerate}

In this article we carry out this program for the moduli stack
of principal $\rho$-sheaves, thus obtaining the intrinsic stack-theoretic construction of the moduli space described in Theorem \ref{thm: main theorem 1 introduction}. Furthermore, this theory provides a $\Theta$-stratification of $\text{Bun}_{\rho}(X)$, which yields a canonical filtration (called the leading term HN filtration) for each unstable object. We apply this construction iteratively to the centers of the $\Theta$-stratification to define the GHN reduction from Theorem \ref{thm: main theorem 2 introduction}.

In Appendix \ref{appendix: positive char} we explain how to apply the same techniques to define the GHN filtration for $\rho$-sheaves in arbitrary characteristic (see \Cref{thm: weak theta stratification positive characteristic}).

\begin{subsection}{Stack-theoretic approach vs. GIT}
We end our introduction with a discussion comparing the stack theoretic approach in \cite{halpernleistner2021structure, alper2019existence} with the usual construction of moduli spaces using Geometric Invariant
Theory. We hope that this can provide some intuition for those readers that are more familiar with the GIT approach.

The construction of the moduli space of torsion free sheaves 
(and principal $\rho$-sheaves) 
follows the usual application of GIT. We first construct a scheme $R$ parametrizing all objects together
with an extra data (in this case an isomorphism between a fixed vector space $W$
and the space of sections $H^0(\mathcal{E}(n))$ for a chosen $n\gg 0$ and
the sheaf $\mathcal{E}$ associated to
each object). The group $\SL(W)$ acts on the scheme $R$. We
choose an ample line bundle on $R$ together with a linearization of the action of $\SL(W)$. Mumford's GIT shows that there is an open subset $R^{ss}\subset R$ which admits a \emph{good quotient} [Ses72], and this is the moduli space of semistable torsion-free sheaves we are looking for.

To identify the subset $R^{ss}$ of semistable points we use the
Hilbert-Mumford criterion. For each point $p \in R$ and for each
one-parameter subgroup $\lambda:\mathbb{G}_m \to \SL(W)$ 
we consider the weight $\mu(\lambda,p)$ of the
action of the one-parameter subgroup $
\lambda$ on the fiber of the limiting
point $\lim_{t\to 0} \lambda(t)\cdot p$. The point $p$ is semistable if and only if $\mu(\lambda,p)\leq 0$ for all one-parameter subgroups $\lambda$. We remark that, in the usual applications of GIT to the construction of moduli spaces involving sheaves, the one-parameter subgroup $\lambda$ produces a filtration of the sheaf, and the Hilbert-Mumford criterion becomes equivalent to checking certain
numerical condition over all filtrations.

In \cite{halpernleistner2021structure},
the one-parameter subgroups appearing in the Hilbert-Mumford criterion
of GIT are replaced by \emph{filtered objects}. For any algebraic stack $\mathcal{M}$, a filtration of a
point  $p:\Spec k \to \mathcal{M}$ is a morphism 
\begin{equation}
  \label{eq:filt}
f:\Theta:=[\mathbb{A}^1/\mathbb{G}_m]\to \mathcal{M}  
\end{equation}
and an identification
$f(1) \simeq p$. 
Given $f$, there is a group homomorphism from $\mathbb{G}_m$ 
to the automorphism group of the
object $f(0)$ in $\mathcal{M}$, which can be thought of as the graded object
associated to the filtration $f$. 

For understanding these notions it is useful to let $\mathcal{M}$ be the moduli
stack of torsion free sheaves on a projective scheme $X$. 
By the Rees
construction (cf. Proposition \ref{prop: filtrations of torsion free sheaves}), a decreasing $\mathbb{Z}$-indexed filtration 
\[0 \subset \cdots \subset \mathcal{E}_{i+1} \subset \mathcal{E}_i \subset \cdots \subset \mathcal{E}\]
of a torsion free sheaf $\mathcal{E}$ on $X$ determines an $\mathbb{A}^1$-family $\widetilde{\mathcal{E}}$ of torsion free sheaves on $X$. We have $\widetilde{\mathcal{E}}|_{X_t}\cong \mathcal{E}$ for all $t\neq 0$, and $\widetilde{\mathcal{E}}|_{X_0}\cong \oplus
\mathcal{E}_i/\mathcal{E}_{i+1}$. 
The group $\mathbb{G}_m$ has a natural action on the
graded object $\widetilde{\mathcal{E}}|_{X_0}$, with weight $i$ on each
piece $\mathcal{E}_i/\mathcal{E}_{i+1}$.
The family $\widetilde{\mathcal{E}}$ gives a morphism $\mathbb{A}^1\to \mathcal{M}$ and the standard action
of $\mathbb{G}_m$ on $\mathbb{A}^1$ lifts to $\widetilde{\mathcal{E}}$. Hence this morphism descends to
a map from the quotient stack $[\mathbb{A}^1/\mathbb{G}_m]$, 
providing a filtration $f$ as in \eqref{eq:filt}.
One can define a function $\nu$ which assigns a real number for each filtration $f$ as
in \eqref{eq:filt}. This is done in a similar way to GIT: we choose
a line bundle $L$  on $\mathcal{M}$, and  define $\nu(f)$ to be the weight
of the action of $\mathbb{G}_m$ on the fiber of $L|_{f(0)}$.

In practice this approach has an important advantage: in the GIT case
we are looking at the action of $\mathbb{G}_m$ on $\SL(W)$. This produces a
filtration on the vector space $W\cong H^0(\mathcal{E}(n))$, which in turn produces a
filtration on $\mathcal{E}$. 
To show that the numerical condition of the Hilbert-Mumford criterion given by $\mu(\lambda,p)$ matches the stability condition on Hilbert polynomials is an involved calculation. 
On the other hand, if we follow the intrinsic approach of \cite{halpernleistner2021structure}, we
are directly looking at a filtration on $\mathcal{E}$ and the action of $\mathbb{G}_m$
on the associated graded sheaf. This makes it easier to obtain the 
stability condition.

In the case of principal $G$-bundles the advantage of the new approach
is even more
clear. In the GIT setup we are looking at all filtrations of $W\cong
H^0(\mathcal{E}(n))$, and it is not easy to show that it is enough to look
at filtrations coming from 1-parameter subgroups of $G$. In contrast, in the
theory of \cite{halpernleistner2021structure, heinloth2018hilbertmumford} we only look at these subgroups from
the very beginning.

Both GIT and \cite{halpernleistner2021structure} can be used to produce filtrations
of unstable objects. We recall how this works in GIT. Let $p\in R$ be an unstable point.
Among all the 1-parameter subgroups
$\lambda$ of $\SL(W)$ with $\mu(\lambda,p)>0$, we can find one for which this weight is 
maximal (after some normalization). This defines the Kempf filtration (cf. \cite{kempf-filtration}) of $p$. This filtration depends on the choices that were made in the
GIT construction (the number $n$, the parameter space $R$ and
the linearization). For torsion free sheaves it can
be shown that it is independent of the choice of $n$ for $n\gg 0$, and
it coincides with the Gieseker-Harder-Narasimhan
filtration \cite{gomez-sols-zamora, zamora-tesis}. Unfortunately, the computations become too involved in the case of principal $G$-bundles, so we do not obtain a Gieseker-Harder-Narasimhan filtration following this approach \cite{zamora-rank2-tensors}.

In the theory in \cite{halpernleistner2021structure}, given an unstable object $p\in \mathcal{M}(k)$ we can also find a filtration $f$ with $f(1) \simeq p$ such that $\nu(f)$ is maximal. This program was carried out for $\Lambda$-modules and pairs in
\cite{torsion-freepaper}. It should be pointed out that the filtration
$f$ which maximizes $\nu$ does not have the property that the graded pieces of the filtration are semistable. But if we apply this procedure inductively we obtain the Gieseker-Harder-Narasimhan filtration.

In this article we show that a similar story applies in the case of principal $\rho$-sheaves. We use this to define what we
call the Gieseker-Harder-Narasimhan reduction of an unstable
$\rho$-sheaf.
\end{subsection}

\textbf{Acknowledgments:} This project was started during the workshop ``Moduli problems beyond geometric invariant theory" at the American Institute of Mathematics. We would like to thank the American Institute of Mathematics and the organizers that made the workshop possible. We would like to thank V. Balaji, Federico Fuentes, Daniel Halpern-Leistner and Jochen Heinloth for helpful remarks. We would also like to thank an anonymous referee for thoughtful comments on the manuscript.

This work is supported
by grants CEX2019-000904-S and PID2019-108936GB-C21 (funded by MCIN/AEI/ 10.13039/501100011033), and NSF grants DMS-1454893 and DMS-2001071.
\end{section}

\begin{section}{Preliminaries} \label{section: preliminaries}
\begin{subsection}{Notation} We will work over a field $k$ of characteristic $0$. We denote by $\Aff_k$ the category of affine schemes over $k$. More generally, for any $k$-scheme $T$ we let $\Aff_T$ be the category of (absolutely) affine schemes equipped with a morphism to $T$. We write $\Sch_k$ and $\Sch_T$ to denote the category of all schemes over $k$ and $T$ respectively. Unless otherwise stated, all schemes will be understood to belong to $\Sch_{k}$. An undecorated product of schemes (e.g. $X \times S$) is understood to be a fiber product over $k$. We might sometimes write $X_{S}$ instead of $X \times S$.

All of the sheaves that we consider in this paper are quasicoherent. Whenever we write ``$\mathcal{O}_Y$-module" we mean a quasicoherent $\mathcal{O}_Y$-module. In order to ease notation for pullbacks, we use the following convention. Whenever we have a map of schemes $f: Y \rightarrow T$ and a sheaf $\mathcal{G}$ on $T$, we will write $\mathcal{G}|_{Y}$ to denote the pullback $f^{*} \mathcal{G}$, if the morphism $f$ is clear from context.

We fix a smooth projective geometrically connected variety $X$ over $k$. We denote by $d$ the dimension of $X$. We also choose once and for all an ample line bundle $\mathcal{O}(1)$ on $X$. For any coherent sheaf $\mathcal{F}$ on $X$, we denote by $P_{\mathcal{F}}$ the Hilbert polynomial of $\mathcal{F}$ with respect to the line bundle $\mathcal{O}(1)$. If $\mathcal{F}$ is a torsion-free sheaf on $X$, then this is a polynomial of degree $d$ in the variable $n$ which can be written in the form $P_{\mathcal{F}}(n) = \sum_{i = 0}^{d} \frac{a_i}{i!} n^i$
for some sequence of rational numbers $a_i$ (see \cite[Lemma 1.2.1]{huybrechts.lehn}). We define the reduced Hilbert polynomial to be $\overline{p}_{\mathcal{F}} \vcentcolon = \frac{1}{a_{d}} P_{\mathcal{F}}$. For every $0 \leq i \leq d-1$, we set $\widehat{\mu}_i(\mathcal{F}) \vcentcolon = \frac{a_{i}}{a_d}$. We call this the $i^{th}$ slope of $\mathcal{F}$. The number $\widehat{\mu}_{d-1}(\mathcal{F}) = \frac{a_{d-1}}{a_d}$ is simply called the slope of $\mathcal{F}$.

We define an ordering for polynomials where, for any two $p_1, p_2 \in \mathbb{R}[n]$,  we write $p_1 \geq p_2$ if $p_1(n) \geq p_2(n)$ for $n\gg0$.

Fix a reductive linear algebraic group $G$ over $k$. We write $k[G]$ to denote the coordinate ring of $G$, which has the structure of a commutative Hopf algebra. We will denote by $\Bun_{G}(X)$ the moduli stack of principal $G$-bundles on $X$. This is an algebraic stack that has affine diagonal and is locally of finite type over $k$ (by \cite[Thm. 1.2]{hall-rydh-tannaka} taking $Z=X$ and $X= BG$).

Fix a $k$-vector space $V$ of dimension $r$. Let us denote by $\GL(V)$ the group of linear automorphisms of $V$. We choose once and for all a faithful representation $\rho: G \rightarrow \GL(V)$. In this paper we sometimes use the following notion which already appeared in the introduction.

\begin{defn}
  \label{defn: central representation}
We say that a homomorphism of algebraic groups $\rho:G_1\to G_2$ is central if the image of
the identity component of the center of $G_1$ is in the center of $G_2$.
\end{defn}

We denote by $\SL(V)$ the closed algebraic subgroup of $\GL(V)$ consisting of linear automorphisms $g$ of $V$ such that the $r^{th}$ exterior power $\wedge^rg$ induces the identity on $\bigwedge^{r}V$.

Set $\Theta$ to be the quotient stack $[\mathbb{A}^1_k/ \, \mathbb{G}_m]$. Here by convention we always equip $\mathbb{A}^1_k = \Spec(k[t])$ with the $\mathbb{G}_m$-action that gives the coordinate function $t$ weight $-1$.
\end{subsection}

\begin{subsection}{Big open sets and torsion-free sheaves} \label{subsection: big open and torsion-free}

\begin{defn} \label{defn: big open subset} Let $S$ be a $k$-scheme. An open subscheme $U \subset X_S$ is said to be big (relative to $S$) if for all points $s \in S$ the complement $X_s \setminus U_s$ of the fiber $U_s$ is a closed subset of codimension at least 2 in $X_{s}$.
\end{defn}

\begin{lemma} \label{lemma: lemma on very big open subsets}
Let $S$ be a $k$-scheme. Let $j: U \hookrightarrow X_S$ be a big open subset in $X_S$. Then the following are satisfied
\begin{enumerate}[(a)]
   \item For all points $x \in X_S \setminus U$, the local ring $\mathcal{O}_{X_S, x}$ has depth at least 2.
    \item $U$ is scheme theoretically dense in $X_S$.
    \item Let $p: Y \rightarrow X_S$ be an affine morphism of finite type. Any section of $p$ defined over $U$ extends uniquely to a section defined over the whole $X_S$.
\end{enumerate}
\end{lemma}
\begin{proof}
\begin{enumerate}[(a)]
    \item Let $x \in X_S \setminus U$ be a point with image $s$ in $S$. By assumption the point $x$ has codimension at least 2 in $X_s$. Since $X_s$ is smooth, this means that $\mathcal{O}_{X_s, x}$ has depth at least $2$. So there exists a regular sequence $\overline{f}_1, \overline{f}_2$ of length $2$ in $\mathcal{O}_{X_s, x}$. Lift this sequence to a pair of elements $f_1, f_2 \in \mathcal{O}_{X_S, x}$. Since $\mathcal{O}_{X_S, x}$ is flat over $\mathcal{O}_{S,s}$, we can apply \cite[\href{https://stacks.math.columbia.edu/tag/00ME}{Tag 00ME}]{stacks-project} twice to conclude that $f_1, f_2$ is a regular sequence. First we apply \cite[\href{https://stacks.math.columbia.edu/tag/00ME}{Tag 00ME}]{stacks-project} with $N = M = \mathcal{O}_{X_S,x}$ and $u = f_1$. Then we apply this result a second time with $N = M = \mathcal{O}_{X_S,x}/(f_1)$ and $u = f_2$.
    \item By definition, the scheme theoretic closure of $U$ is cut out by the ideal sheaf of $\mathcal{O}_{X_S}$ that locally consists of sections whose restriction to $U$ vanishes. In other words, it is cut out by the kernel of the unit $\mathcal{O}_{X_S} \rightarrow j_* j^*\mathcal{O}_{X_S}$. By part $(a)$ every point outside $U$ has depth at least $2$. Therefore the unit map $\mathcal{O}_{X_S} \rightarrow j_* j^*\mathcal{O}_{X_S}$ is an isomorphism by \cite[\href{https://stacks.math.columbia.edu/tag/0E9I}{Tag 0E9I}]{stacks-project}. We conclude that the scheme theoretic closure is cut out by the $0$ ideal, i.e. it is the whole $X_S$.
    \item Uniqueness of the extension follows because $U$ is scheme theoretically dense in $X_S$ and the morphism $Y \rightarrow X_S$ is separated. It remains to show existence. Since extensions are unique, any set of local extensions will glue. Therefore we can work Zariski locally. Let $V$ be an affine open subset of $X_S$. Since $Y$ is of finite type over $X_S$, the scheme $Y_V$ can be viewed as a closed subscheme of $\mathbb{A}^n_{V}$ for some $n$. Since $U \cap V$ is scheme theoretically dense inside $V$, any extension of the section to $\mathbb{A}^n_{V}$ will automatically factor through the closed subscheme $Y_{V}$. It therefore suffices to extend the composition $p: U \cap V \rightarrow Y_{U \cap V} \hookrightarrow \mathbb{A}^n_{U \cap V}$ to a map $\widetilde{p}: V \rightarrow \mathbb{A}^n_{V}$. This amounts to taking a section of $\mathcal{O}_{U\cap V}^{\oplus n}$ and extending to a section of $\mathcal{O}_{V}^{\oplus n}$. This in turn follows from \cite[\href{https://stacks.math.columbia.edu/tag/0E9I}{Tag 0E9I}]{stacks-project} applied to the sheaf $\mathcal{O}_{V}^{\oplus n}$.
\end{enumerate}
\end{proof}

For the following definition, recall that the rank of a torsion-free sheaf on an integral scheme is the dimension of the generic fiber.
\begin{defn} Let $S$ be a $k$-scheme. Let $\mathcal{F}$ be a $\mathcal{O}_{X_S}$-module. We say that $\mathcal{F}$ is a relative torsion-free sheaf on $X_{S}$ if it is $S$-flat, finitely presented, and for all points $s \in S$ the fiber $\mathcal{F}|_{X_s}$ is torsion-free. We say that $\mathcal{F}$ has rank $h$ if all of the fibers $\mathcal{F}|_{X_s}$ have rank $h$.
\end{defn}

From now on we will drop the adjective ``relative" whenever it is clear from context.

\begin{lemma}
Let $\mathcal{F}$ be a torsion-free sheaf on $X_S$. Then, there exists a big open subset $U \subset X_S$ such that $\mathcal{F}|_{U}$ is locally-free.
\end{lemma}
\begin{proof}
Let $U$ denote the (maximal) open subset of $X_S$ where $\mathcal{F}$ is locally-free. Choose $s \in S$. Since $\mathcal{F}$ is $S$-flat and finitely presented, we can apply \cite[\href{https://stacks.math.columbia.edu/tag/0CZR}{Tag 0CZR}]{stacks-project} to conclude that $U \cap X_s$ coincides with the set of points of $X_s$ where the fiber $\mathcal{F}|_{X_s}$ is locally-free. Since $X_s$ is smooth, all points of codimension $1$ are discrete valuation rings. Therefore the torsion-free sheaf $\mathcal{F}|_{X_s}$ is automatically free at all codimension $\leq 1$ points. We conclude that the closed complement of $U \cap X_s$ has codimension at least $2$, as desired.
\end{proof}

Next we recall the notion of determinant for torsion-free sheaves, as in \cite[5.6]{kobayashi-vectorbundles}. Let $Y$ be a $k$-scheme and $\mathcal{G}$ be a finitely presented $\mathcal{O}_Y$-module. Let $h$ be a nonnegative integer. There is a natural action of the symmetric group $S_h$ of $h$-letters on the $h$-fold tensor $\mathcal{G}^{\otimes h}$, given by permuting the indexes. Let $\text{sign}: S_h \rightarrow \{1, -1\}$ denote the sign character.

\begin{defn} \label{defn: wedge of sheaves}
Let $Y$ and $\mathcal{G}$ be as in the paragraph above. We denote by $\mathcal{G}_{Alt, h}$ the $\mathcal{O}_Y$-submodule of $\mathcal{G}^{\otimes h}$ that is locally generated over each affine open by sections of the form $\sigma \cdot (x_1 \otimes x_2 \otimes \cdots \otimes x_h) - \text{sign}(\sigma) \, x_1 \otimes x_2 \otimes \cdots \otimes x_h$ with $\sigma \in S_h$. We denote by $\bigwedge^{h} \mathcal{G}$ the quotient sheaf $\mathcal{G}^{\otimes h} / \, \mathcal{G}_{\text{Alt, h}}$.
\end{defn}

Note that $\bigwedge^h \mathcal{G}$ is a finitely presented $\mathcal{O}_Y$-sheaf by definition. The formation of $\bigwedge^h \mathcal{G}$ is compatible with arbitrary base-change, because taking pullbacks is right exact and commutes with tensor products.

\begin{defn} \label{defn: determinant}
Let $S$ be a $k$-scheme and let $\mathcal{F}$ be a torsion-free sheaf of rank $h$ on $X_S$. We define $\det(\mathcal{F}) \vcentcolon = \left( \bigwedge^h\mathcal{F}\right)^{\vee \vee}$.
\end{defn}

For all morphisms of $k$-schemes $T \rightarrow S$ there is a canonical morphism $\det(\mathcal{F})|_{X_T} \rightarrow \det\left(\mathcal{F}|_{X_T}\right)$. It turns out that this is an isomorphism, as shown in the following lemma.
\begin{lemma} \label{lemma: properties of determinant}
Let $S$ be a $k$-scheme and let $\mathcal{F}$ be a torsion-free sheaf on $X_S$ such that the rank of the fibers is constant. Then:
\begin{enumerate}[(a)]
    \item Let $j: U \hookrightarrow X_S$ be a big open subscheme such that $\mathcal{F}$ is locally-free. Then $\det(\mathcal{F}) = j_* \left(\det(\mathcal{F}|_{U}) \right)$, where $\det(\mathcal{F}|_{U})$ on the right hand side denotes the usual determinant of a locally-free sheaf.
    \item $\det(\mathcal{F})$ is a line bundle on $X_S$.
    \item For all morphisms of $k$-schemes $T \rightarrow S$, the canonical morphism $\det(\mathcal{F})|_{X_T} \rightarrow \det\left(\mathcal{F}|_{X_T}\right)$ is an isomorphism.

\end{enumerate}
\end{lemma}
\begin{proof}
\begin{enumerate}[(a)]
    \item By construction, $\det(\mathcal{F})|_{U}$ coincides with the usual determinant of $\mathcal{F}|_{U}$ as a locally-free sheaf. Therefore it suffices to show that the unit $\det(\mathcal{F}) \rightarrow j_* \left(\det(\mathcal{F})|_{U} \right)$ is an isomorphism. By \cite[\href{https://stacks.math.columbia.edu/tag/0E9I}{Tag 0E9I}]{stacks-project} it suffices to check that the sheaf $\det(\mathcal{F}) = \Hom\left(\left( \bigwedge^h\mathcal{F}\right)^{\vee}, \mathcal{O}_{X_S}\right)$ has depth at least $2$ at all points in the complement of $U$. Notice that this follows from $(2)$ in \cite[\href{https://stacks.math.columbia.edu/tag/0AV5}{Tag 0AV5}]{stacks-project}, since $\mathcal{O}_{X_S}$ has depth at least $2$ at all points in the complement of $U$ by Lemma \ref{lemma: lemma on very big open subsets}.
    \item One can define a notion of determinant $\text{d}(\mathcal{F})$ using a finite resolution by locally-free sheaves, as in \cite[pg. 36,37]{huybrechts.lehn}. By definition, $\text{d}(\mathcal{F})$ is a line bundle that agrees with $\det(\mathcal{F})$ on the open subset $U$. By the same proof as in part $(a)$ applied to the line bundle $\text{d}(\mathcal{F})$, it follows that $\text{d}(\mathcal{F}) \cong j_* (\det(\mathcal{F}|_{U})$. We conclude that $\det(\mathcal{F}) = j_* \left(\det(\mathcal{F}|_{U}) \right) = \text{d}(\mathcal{F})$ is a line bundle.
    \item This follows because the notion of determinant $\text{d}(\mathcal{F})$ defined in \cite[pg. 36,37]{huybrechts.lehn} using a fixed choice of a resolution can be seen to be compatible with base-change, by pulling back the resolution.

\end{enumerate}
\end{proof}

Recall that $r$ is the dimension of the vector space $V$ that we use to define the faithful representation of the group $G$. The following is an open substack of the algebraic stack of coherent sheaves on $X$ as in \cite[4.6.2]{lmb-champsalgebriques}.
\begin{defn} \label{defn: stack of torsion free sheaves with trivialization}
We define $\Coh_r^{tf}(X)$ to be the pseudofunctor from $\left(\Aff_{k}\right)^{op}$ into groupoids defined as follows. For any affine scheme $T \in \Aff_{k}$, we set
\begin{gather*} \Coh_r^{tf}(X)\, (T) \; = \; \left[ \begin{matrix} \; \; \text{groupoid of $T$-flat relative torsion-free sheaves $\mathcal{F}$ of rank $r$ on $X_{T}$}  \; \; \end{matrix} \right]\end{gather*}
\end{defn}

\begin{prop} \label{prop: algebraicity of SLtf}
$\Coh_r^{tf}(X)$ is an algebraic stack with affine diagonal and locally of finite type over $k$.
\end{prop}
\begin{proof}
Since the rank of a family of torsion-free sheaves is a locally constant function on the base, $\Coh^{tf}_{r}(X)$ is an open substack of the stack $\Coh^{tf}(X)$ that parametrizes all torsion-free sheaves. The proposition follows, because $\Coh^{tf}(X)$ is an algebraic stack with affine diagonal and locally of finite type over $k$ \cite[Prop. 2.5]{torsion-freepaper}.
\end{proof}

\begin{remark}
The stack $\Coh_r^{tf}(X)$ contains the stack $\Bun_{r}(X)$ of rank $r$ vector bundles on $X$ as an open substack.
\end{remark}
\end{subsection}

\begin{subsection}{The scheme of $G$-reductions} \label{subsection: scheme of G-reductions}
Let $S$ be a $k$-scheme and let $\mathcal{F} \in \Coh_r^{tf}(X)(S)$ be a rank $r$ torsion-free sheaf on $X_S$. In this subsection we modify the construction in \cite{schmitt.singular} in order to define a scheme of $G$-reductions of $\mathcal{F}$. In our convention the representation $\rho$ is the dual of the representation in \cite{schmitt.singular}. 

Let $Y$ be a $k$-scheme. Let $\mathcal{G}$ be a finitely presented $\mathcal{O}_Y$-module. We can form the universal graded associative algebra $T(\mathcal{G}) := \bigoplus_{j=0}^{\infty} \mathcal{G}^{\otimes j}$. There is a graded ideal that on each affine open is generated by sections of the form $x \otimes y - y \otimes x$. The quotient of $T(\mathcal{G})$ by this ideal is the symmetric algebra $\Sym^{\bullet}(\mathcal{G})$. By definition $\Sym^{\bullet}(\mathcal{G})$ is a graded commutative $\mathcal{O}_{Y}$-algebra that is locally finitely presented as an algebra.

Suppose that the sheaf $\mathcal{G}$ admits a $k[G]$-comodule structure. This induces a comodule structure on $T(\mathcal{G})$, where we equip each graded component $\mathcal{G}^{\otimes j}$ with the tensor comodule structure. The ideal used to form $\Sym^{\bullet}(\mathcal{G})$ is a $k[G]$-subcomodule. It follows that the quotient $\Sym^{\bullet}(\mathcal{G})$ acquires a $k[G]$-comodule structure that is compatible with the multiplication and the grading.

\begin{defn}
Let $S$ be a $k$-scheme and let $\mathcal{F}$ be a torsion-free sheaf of rank $r$ on $X_S$. We set
\[ E_{V}(\mathcal{F}) \vcentcolon = \Sym^{\bullet}(\mathcal{F} \otimes_k V^{\vee}) \otimes_{\mathcal{O}_{X_{S}}} \Sym^{\bullet}(\det(\mathcal{F})^{\vee} \otimes_k \det(V)) \]
We also define $H(\mathcal{F}, V) \vcentcolon = \underline{\Spec}_{X_S}(E_{V}(\mathcal{F}))$. Note that the structure morphism $H(\mathcal{F}, V) \rightarrow X_S$ is affine and of finite presentation.
\end{defn}

The vector space $V$ acquires a left $k[G]$-comodule structure $V \rightarrow V \otimes k[G]$ via the representation $\rho: G \rightarrow \GL(V)$. We equip $V^{\vee}$ with the dual right comodule structure $c : V^{\vee} \rightarrow V^{\vee} \otimes k[G]$. This induces a comodule structure $\mathcal{F} \otimes_{k} V^{\vee} \rightarrow (\mathcal{F} \otimes_{k} V^{\vee}) \otimes k[G]$ given by $\text{id}_{\mathcal{F}} \otimes c$. By the discussion above, the graded algebra $\Sym^{\bullet}(\mathcal{F} \otimes_k V^{\vee})$ inherits a right $k[G]$-comodule structure.

On the other hand, $\det(V)$ acquires a left $k[G]$-comodule structure via the determinant representation of $V$. By post-composing with the antipode of the Hopf algebra $k[G]$ (i.e. pre-composing the representation with inversion on $G$), we get a right $k[G]$-comodule structure on $\det(V)$. In a similar way as above, this induces a right $k[G]$-comodule structure on the algebra  $\Sym^{\bullet}(\det(\mathcal{F})^{\vee} \otimes_k \det(V))$. 

By taking the tensor comodule structure for the $k[G]$-comodules defined above, we can equip $E_{V}(\mathcal{F})$ with a right $k[G]$-comodule structure
\[c_{E_{V}(\mathcal{F})}:E_{V}(\mathcal{F}) \rightarrow E_{V}(\mathcal{F}) \otimes_k k[G]\]
Since this comodule structure is compatible with the algebra multiplication, it translates into a left action of the algebraic group $G$ on the scheme $H(\mathcal{F}, V)$.

Consider the subalgebra $E_{V}(\mathcal{F})^G$ of $G$-invariants of $E_{V}(\mathcal{F})$. Recall that this is defined over each affine open as the $\mathcal{O}_{X_S}$-subalgebra of $E_{V}(\mathcal{F})$ consisting of sections $s$ such that $c_{E_{V}(\mathcal{F})}(s) = s \otimes 1_{k[G]}$. Since the characteristic of $k$ is $0$ and $G$ is reductive, the algebra of invariants $E_{V}(\mathcal{F})^G$ is a locally finitely presented $\mathcal{O}_{X_S}$-algebra and its formation is compatible with arbitrary base-change. This can be checked affine locally, and over each affine it follows from the existence of the Reynolds operator \cite[Chpt. 1, \S1,2]{mumford-git}. By definition, the scheme $\underline{\Spec}_{X_S}\left((E_{V}(\mathcal{F}))^G\right)$ is the affine GIT quotient $H(\mathcal{F}, V) /\!\!/ \,G$ relative to $X_S$.

 Let $U$ be a big open subscheme of $X_S$ where $\mathcal{F}$ is locally-free. The restriction $H(\mathcal{F}, V)|_{U}$ is, by definition, the scheme 
 \[H \vcentcolon = \Hom\left( \mathcal{F}|_{U}, V \otimes_{k} \mathcal{O}_{U}\right) \times \Hom\left( \det(\mathcal{F}|_{U})^{\vee}, \, \det(V^{\vee}) \otimes_k \mathcal{O}_{U}\right)\]
 parametrizing pairs $(f_1, f_2)$ of homomorphisms of vector bundles
 \[ f_1: \mathcal{F}|_{U} \rightarrow V \otimes_{k} \mathcal{O}_{U} \;  \; \; \; \text{and} \;  \; \; \; f_2: \det(\mathcal{F}|_{U})^{\vee} \rightarrow \det(V^{\vee}) \otimes_{k} \mathcal{O}_{U}\]
 This means that there is a universal pair $(u_1, u_2)$ of morphisms
 \[u_1: \mathcal{F}|_{H} \rightarrow V \otimes_k \mathcal{O}_{H} \; \; \; \; \text{and} \; \; \; \; u_2: \det(\mathcal{F}|_{H})^{\vee} \rightarrow \det(V^{\vee}) \otimes \mathcal{O}_{H}\]
 We take the $r^{th}$ wedge of $u_1$ to get a universal determinant morphism.
 \[\wedge^{r} u_1: \det(\mathcal{F}|_{H}) \rightarrow \det(V)\otimes \mathcal{O}_{H} \]
 We set $\delta_{U} \vcentcolon = (\wedge^{r} u_1) \otimes u_2$. Using the canonical trivializations $\det (V) \otimes \det(V^{\vee}) \xrightarrow{\sim} k$ and $\det(\mathcal{F}|_{H}) \otimes \det(\mathcal{F}|_{H})^{\vee} \xrightarrow{\sim} \mathcal{O}_H$, we can view this as a map $\delta_U: \mathcal{O}_H \rightarrow \mathcal{O}_H$. We can then form the map $\delta_U -1: \mathcal{O}_{H} \rightarrow \mathcal{O}_H$ obtained by subtracting the identity of $\mathcal{O}_H$. Notice that $\delta_{U} -1$ cuts out the divisor on $H$ given by the locus where $u_1$ is an isomorphism and $u_2$ is the determinant of the inverse dual of $u_1$. In particular the vanishing locus of $\delta_{U}-1$ is canonically isomorphic to the scheme parametrizing linear isomorphisms of vector bundles $\Iso(\mathcal{F}|_{U}, V \otimes_{k} \mathcal{O}_{U})$. This identification is given by forgetting $u_2$.
 
By construction, the composition $\mathcal{O}_{H/\!\!/\,G} \hookrightarrow \mathcal{O}_H \xrightarrow{\delta_{U}-1} \mathcal{O}_{H}$ lands in the $\mathcal{O}_{U}$-subalgebra of invariants $(E_{V}(\mathcal{F})|_{U})^G$. Therefore this composition induces a well-defined map $\alpha_{U}: \mathcal{O}_{H/\!\!/\,G} \rightarrow \mathcal{O}_{H/\!\!/\,G}$ that cuts out an effective divisor on $H/\!\!/\,G$. We denote this closed subscheme by $(H/\!\!/\,G)_{\alpha_{U} = 0} \hookrightarrow H/\!\!/\,G$.

\begin{defn} \label{defn: scheme of G reductions}
$\Red_{G}(\mathcal{F})$ is defined to be the scheme theoretic closure of $(H/\!\!/\,G)_{\alpha_{U} = 0}$ inside $H(\mathcal{F}, V) /\!\!/\,G$. In other words, $\Red_{G}(\mathcal{F})$ is the scheme theoretic image of the composition
\[ (H/\!\!/\,G)_{\alpha_{U} = 0} \hookrightarrow H/\!\!/\,G \hookrightarrow H(\mathcal{F}, V)/\!\!/\,G\]
where the second map $H/\!\!/\,G \hookrightarrow H(\mathcal{F}, V)/\!\!/\,G$ is the open immersion that comes from identifying 
$H/\!\!/\,G \xrightarrow{\sim} \left(H(\mathcal{F}, V)/\!\!/\,G\right)|_{U}$
via the compatibility of taking $G$-invariants with flat base-change.
\end{defn}

\begin{remark}
It should be noted that the construction of $\Red_{G}(\mathcal{F})$ does not depend on the choice of $U$. Indeed, for any two big open subsets $U_1$ and $U_2$ of $X_S$, the intersection $U_1 \cap U_1$ is also big. Since the element $\alpha_{U}$ is canonically defined above, the restrictions of $\alpha_{U_1}$ and $\alpha_{U_2}$ to the subset $U_1 \cap U_2$ agree with $\alpha_{U_1 \cap U_2}$. Notice that $H$ is always flat over $U$, because it is locally given by a polynomial algebra. It follows that $H /\!/\! \, G$ is also flat over $U$, because taking $G$ invariants is exact by the existence of the Reynolds operator. Since $U$ is scheme theoretically dense in $X_S$ by Lemma \ref{lemma: lemma on very big open subsets}, it follows that  the open $(H/\!/\! \,G)|_{U_1 \cap U_2}$ is scheme theoretically dense in $(H/\!/\! \,G)|_{U_1}$ and $(H/\!/\! \,G)|_{U_2}$. In particular the closures of $(H/\!/\! \,G)_{\alpha_{U_1} = 0}$ and $(H/\!/\! \,G)_{\alpha_{U_2} = 0}$ both agree with the closure of $(H/\!/\! \,G)_{\alpha_{U_1\cap U_2} = 0}$.
\end{remark}

The following lemma is an immediate consequence of the constructions above.
\begin{lemma} \label{lemma: reduction scheme and base-change}
The scheme $\Red_{G}(\mathcal{F})$ is affine and of finite presentation over $X_S$. The formation of $\Red_{G}(\mathcal{F})$ is compatible with base-change. In other words, for any morphism of $k$-schemes $T \rightarrow S$ there is a canonical isomorphism $\Red_{G}(\mathcal{F}) \times_{X_S} X_T \xrightarrow{\sim} \Red_{G}\left(\mathcal{F}|_{X_T}\right)$. \qed
\end{lemma}

Let $U$ be a big open subset of $X_S$ where $\mathcal{F}$ is locally-free. Since $\mathcal{F}|_{U}$ is a vector bundle of rank $r$ on $U$, we can naturally associate to it a $\GL(V)$-bundle $\mathcal{P}$ on $U$.
\begin{prop} \label{prop: scheme of reductions as fiber bundle}
Let $\mathcal{F}$ and $U$ be as above. Let $\mathcal{P}$ denote the corresponding $\GL(V)$-bundle on $U$. Then there is a canonical identification of $U$-schemes $\Red_{G}(\mathcal{F})|_{U} \xrightarrow{\sim} \mathcal{P}/G$. Here the quotient on the right is formed via the inclusion $\rho:G \hookrightarrow \GL(V)$.
\end{prop}
\begin{proof}
Let $H_{\delta_{U} =1}$ denote the closed subscheme of $H$ cut out by the equation $\delta_{U} -1$, as described in the construction of $\Red_{G}(\mathcal{F})$ above. By definition, $\Red_{G}(\mathcal{F})|_{U}$ is the affine GIT quotient $(H_{\delta_{U} =1})/\!/ \! \, G$. We have seen that $H_{\delta_{U} =1}$ is canonically isomorphic to the $U$-scheme $\Iso(\mathcal{F}|_{U}, V\otimes_{k} \mathcal{O}_{U})$ parametrizing isomorphisms of vector bundles. This identification respects the natural left $G$-actions defined on both schemes via the inclusions $\rho: G \rightarrow \GL(V)$.

The map that takes an isomorphism $\varphi: \mathcal{F}|_{U} \oplus \det(\mathcal{F})^{\vee}|_{U} \xrightarrow{\sim} (V \oplus \det(V^{\vee}) )  \otimes_{k} \mathcal{O}_{U}$ to its inverse in $\Hom\left( (V \oplus \det(V^{\vee}) ) \otimes_{k} \mathcal{O}_{U}, \, \mathcal{F}|_{U} \oplus \det(\mathcal{F})^{\vee}|_{U} \right)$ induces an isomorphism of schemes
\[ (-)^{-1} : \; \Iso(\mathcal{F}|_{U}, V \otimes_{k} \mathcal{O}_{U}) \xrightarrow{\sim} \Iso\left( V \otimes_{k} \mathcal{O}_{U}, \, \mathcal{F}|_{U} \right)\]
The right-hand side admits a right action of $\GL(V)$ given by precomposition. By definition $\Iso(V \otimes_{k} \mathcal{O}_{U}, \, \mathcal{F}|_{U})$ is the total space of the $\GL(V)$-bundle $\mathcal{P}$ equipped with its right $\GL(V)$-action. The inclusion $\rho : G \hookrightarrow \GL(V)$ induces a right $G$-action on $\mathcal{P}$. The inversion map $(-)^{-1}$ above intertwines the left $G$-action on the left hand side and the right $G$-action on the right hand side. 

 Since the  $\rho$ is faithful, it follows that $G$ acts properly and freely on $\mathcal{P}$. Therefore the GIT quotient $\mathcal{P} /\!/ \! \, G$ and the quotient as \'etale sheaves $\mathcal{P}/G$ coincide. This way we get the sought after canonical identification of $\Red_{G}(\mathcal{F})|_{U} \xrightarrow{\sim} \mathcal{P}/G$.
\end{proof}

The following proposition explains the connection between the scheme $\Red_{G}(\mathcal{F})$ and $G$-reductions of structure group.
\begin{prop} \label{prop: relation scheme and G-reductions}
For any big open subset $U \subset X_S$ where $\mathcal{F}$ is locally-free, there is a bijection between the following two sets $(A), (B)$:
\begin{enumerate}[(A)]
    \item The set of sections of the structure morphism $\Red_{G}(\mathcal{F}) \rightarrow X_S$.
    \item The set of $G$-reductions of structure group for the $\GL(V)$-bundle $\mathcal{P}$ corresponding to $\mathcal{F}|_{U}$.
\end{enumerate}
\end{prop}
\begin{proof}
$\Red_{G}(\mathcal{F}) \rightarrow X_S$ is an affine morphism of finite type. Lemma \ref{lemma: lemma on very big open subsets} shows that the restriction of maps induces a bijection between sections $\sigma : X_S \rightarrow \Red_{G}(\mathcal{F})$ defined over the whole $X_S$ and sections $\sigma_{U} : U \rightarrow \Red_{G}(\mathcal{F})$ defined over $U$. By Proposition \ref{prop: scheme of reductions as fiber bundle}, sections of $\Red_{G}(\mathcal{F})|_{U} \rightarrow U$ are in canonical correspondence with sections of $\mathcal{P}/G \rightarrow U$ via the inversion map. A section of $\mathcal{P}/G \rightarrow U$ is the same as a $G$-reduction of structure group for $\mathcal{P}$. This establishes the canonical bijection between $(A)$ and $(B)$.
\end{proof}

\end{subsection}
\end{section}

\begin{section}{Principal $\rho$-sheaves} \label{section: rho-sheaves}
In this section we define the stack of $\rho$-sheaves and give a description of its filtrations.
\begin{subsection}{Principal $\rho$-sheaves and the moduli stack $\text{Bun}_{\rho}(X)$}
\begin{defn} \label{defn: rho sheaf}
Let $S$ be a $k$-scheme. A principal $\rho$-sheaf on $X_S$ consists of a pair $(\mathcal{F}, \sigma)$, where
\begin{enumerate}[(1)]
    \item $\mathcal{F}$ is a family of torsion-free sheaves of rank $r$ on $X_S$.
    \item $\sigma$ is a section of the structure morphism $\Red_{G}(\mathcal{F}) \rightarrow X_S$.
\end{enumerate}
\end{defn}

\begin{remark}
We will often omit the word ``principal" and refer to these objects simply as $\rho$-sheaves.
\end{remark}

By Proposition \ref{prop: relation scheme and G-reductions}, one can think of a $\rho$-sheaf over $X_S$ as the data of
\begin{enumerate}[(1)]
    \item A torsion-free sheaf $\mathcal{F}$ of rank $r$ on $X_S$.
    \item A reduction of structure group to $G$ for the $\GL(V)$ bundle $\mathcal{F}|_{U}$, where $U$ is (any) big open subscheme of $X_S$ such that $\mathcal{F}|_{U}$ is locally-free.
\end{enumerate}

\begin{remark} \label{remark: comparison classical defn}
Assume that the image of the representation $\rho$ is contained in the subgroup $\SL(V) \subset \GL(V)$. Then our $\rho$-sheaves can be viewed as singular principal $G$-bundles \cite{glss.singular.char} for the dual representation $\rho^{\vee}$.
\end{remark}

\begin{defn}
Let $(\mathcal{F}_1, \sigma_1)$ and $(\mathcal{F}_2,\sigma_2)$ be  two $\rho$-sheaves on $X_S$. By construction, an isomorphism $\varphi: \mathcal{F}_{1} \xrightarrow{\sim} \mathcal{F}_{2}$ induces an isomorphism of $G$-reduction schemes $\Red_{G}(\mathcal{F}_1) \xrightarrow{\sim} \Red_{G}(\mathcal{F}_2)$. We say that $\varphi$ is an isomorphism of the $\rho$-sheaves if the sections $\sigma_1$ and $\sigma_2$ are identified under this induced isomorphism. 
\end{defn}

Using this notion of isomorphism, we endow the class of $\rho$-sheaves on $X_S$ with the structure of a groupoid. By Lemma \ref{lemma: reduction scheme and base-change}, the formation of $\Red_{G}(\mathcal{F})$ is compatible with base-change. In particular, it makes sense to pull back principal $\rho$-sheaves under a morphism of $k$-schemes $T \rightarrow S$.
\begin{defn} \label{defn: stack of singular bundles}
We define $\Bun_{\rho}(X)$ to be the pseudofunctor from $\left(\Aff_{k}\right)^{op}$ into groupoids defined as follows. For any affine scheme $T \in \Aff_{k}$, we set
\[ \Bun_{\rho}(X)\, (T) \; = \; \left[ \; \; \text{groupoid of $\rho$-sheaves $(\mathcal{F}, \sigma)$ on $X_T$} \; \; \; \right]\]
\end{defn}

There is a natural forgetful morphism $\Bun_{\rho}(X) \rightarrow \Coh_{r}^{tf}(X)$ given by $(\mathcal{F}, \sigma) \mapsto \mathcal{F}$.

\begin{prop} \label{prop: forgetful morphism is affine}
The forgetful morphism $\Bun_{\rho}(X) \rightarrow \Coh_{r}^{tf}(X)$ is schematic, affine and of finite type. In particular $\Bun_{\rho}(X)$ is an algebraic stack with affine diagonal and locally of finite type over $k$.
\end{prop}
\begin{proof}
Let $T \in \Aff_k$. Choose a morphism $T \rightarrow \Coh_{r}^{tf}(X)$ represented by a relative torsion-free sheaf $\mathcal{F}$ on $X_{T}$. The fiber product $\Bun_{\rho}(X) \times_{\Coh_{r}^{tf}(X)} T$ is the functor that parametrizes sections of the structure morphism $\Red_{G}(\mathcal{F}) \rightarrow X_T$. Since we know that $\Red_{G}(\mathcal{F}) \rightarrow X_T$ is affine and of finite presentation, we can apply \cite[Prop. 4.4.1]{behrend-thesis} (or \cite[Lemma 2.3 (i)]{hall-rydh-hilbert-quot} + \cite[Thm. 1.3 (ii)]{hall-rydh-tannaka}) to conclude that $\Bun_{\rho}(X) \times_{\Coh_{r}^{tf}(X)} T$ is represented by a scheme that is relatively affine and of finite type over $T$.
\end{proof}

Recall that we have an open immersion $\Bun_{r}(X) \hookrightarrow \Coh_{r}^{tf}(X)$ from the stack $\Bun_{r}(X)$ of rank $r$ vector bundles on $X$.
\begin{prop}
There is a Cartesian square
\[
\begin{tikzcd}
  \Bun_{G}(X) \ar[r, symbol= \xhookrightarrow{\;\;}] \ar[d, "\rho_*"] &  \Bun_{\rho}(X) \ar[d]\\   \Bun_{r}(X) \ar[r, symbol= \xhookrightarrow{\;\;}] & \Coh_{r}^{tf}(X)
\end{tikzcd}
\]
where $\Bun_{G}(X)$ is the stack of principal $G$-bundles on $X$ and $\rho_*$ is the map on stacks induced by extension of the structure group via the representation $\rho$.
\end{prop}
\begin{proof}
Let $S \in \Aff_k$. The groupoid $\Bun_{r}(X) \, (S)$ consists of points $\mathcal{F}\in \Coh_{r}^{tf}(X) \, (S)$ where $\mathcal{F}$ is locally-free on the entirety of $X_S$. Applying Proposition \ref{prop: relation scheme and G-reductions} with $U= X_S$, we see that the fiber product $\Bun_{r}(X) \times_{\Coh_{r}^{tf}(X)} \Bun_{\rho}(X) \,(S)$ parametrizes pairs $(\mathcal{P}, \theta)$, where $\mathcal{P}$ is a $\GL(V)$-bundle on $X_S$ and $\theta$ is a reduction of structure group to $G$. This groupoid is naturally isomorphic to the groupoid of $G$-bundles on $X_S$.
\end{proof}
\end{subsection}

\begin{subsection}{Filtrations of $\rho$-sheaves} 
We explain the notion of $\Theta$-filtration \cite{halpernleistner2021structure, heinloth2018hilbertmumford, alper2019existence} in the context of $\rho$-sheaves. For this subsection we will work over a field extension $K \supset k$. 

Let us recall the definition of a $\Theta$-filtration.  We denote by $\Theta_{K}$ the base-change $\left[ \, \mathbb{A}^1_{K} / \mathbb{G}_m \right]$. The stack $\Theta_{K}$ has a unique closed point $0$ corresponding to the origin of $\mathbb{A}^1_K$. It also has an open point $1$, which is the image of the $K$-point $1$ in $\mathbb{A}^1_{K}$. Let $\mathcal{M}$ be a stack over $k$ and let $p \in \mathcal{M}(K)$ be a $K$-point. A $\Theta$-filtration of the point $p$ is a morphism of stacks $f: \Theta_{K} \rightarrow \mathcal{M}$ along with an isomorphism $f(1) \xrightarrow{\sim} p$. The graded point $f|_{[0/\mathbb{G}_m]}$ is called the associated graded of the filtration. We say that the filtration is nondegenerate if the action of $\mathbb{G}_m$ on the object $f(0)$ is nontrivial.

We start by giving a concrete description of $\Theta$-filtrations in the stack $\Coh_{r}^{tf}(X)$.
\begin{prop} \label{prop: filtrations of torsion free sheaves}
Let $\mathcal{F}$ be a $K$-point of $\Coh^{tf}_r(X)$. A $\Theta$-filtration $f : \Theta_{K} \rightarrow \Coh_{r}^{tf}(X)$ of $\mathcal{F}$ corresponds to a $\mathbb{Z}$-filtration $(\mathcal{F}_m)_{m \in \mathbb{Z}}$ by $\mathcal{O}_{X_K}$-subsheaves of $\mathcal{F}$ satisfying the following properties
\begin{enumerate}[(1)]
    \item $\mathcal{F}_{m+1} \subset \mathcal{F}_{m}$.
    \item $\mathcal{F}_m = 0$ for $m \gg 0$ and $\mathcal{F}_m = \mathcal{F}$ for $m \ll 0$.
    \item $\mathcal{F}_{m}/\mathcal{F}_{m+1}$ is torsion-free.
\end{enumerate}
An isomorphism between two such filtrations is an automorphism of $\mathcal{F}$ that identifies the filtrations. The graded object at the origin $f|_{[0/\mathbb{G}_m]} : [0/\mathbb{G}_m] \rightarrow \Coh_{r}^{tf}(X)$ is the associated graded sheaf $\bigoplus_{m \in \mathbb{Z}} \mathcal{F}_{m} / \mathcal{F}_{m+1}$.
\end{prop}
\begin{proof}
For any $\mathbb{Z}$-filtration $(\mathcal{F}_{m})_{m \in \mathbb{Z}}$ satisfying $(1)$ and $(2)$ above, we can associate a $\mathbb{G}_m$-equivariant coherent sheaf on $X \times \mathbb{A}^{1}_{K}$ via the Rees construction \cite[Prop. 1.0.1]{halpernleistner2021structure}. It follows from the construction that the corresponding morphism $\Theta_K \to \Coh(X)$ lands in the substack of torsion-free sheaves $\Coh_{r}^{tf}(X)$ if and only if the condition $(3)$ above is satisfied.
\end{proof}

For any filtration $\Theta_{K} \rightarrow \Bun_{\rho}(X)$ we can post-compose with the forgetful morphism $\Bun_{\rho}(X) \rightarrow \Coh_{r}^{tf}(X)$ to get a filtration $\Theta_{K} \rightarrow \Coh_{r}^{tf}(X)$. This defines a functor $F_{K}$ between the groupoid of $\Theta$-filtrations of $K$-points of $\Bun_{\rho}(X)$ and the groupoid of $\Theta$-filtrations of $K$-points of $\Coh_{r}^{tf}(X)$.
\begin{prop} \label{prop: filtrations of singular G bundles vs filtrations of torsion-free sheaves}
The functor $F_{K}$ defined above is faithful.
\end{prop}
\begin{proof}
This follows from \cite[Lem. 3.7 (3)]{alper2019existence}, which applies here because the forgetful morphism is affine by Proposition \ref{prop: forgetful morphism is affine}.
\end{proof}

This tells us that, if $(\mathcal{F}, \sigma)$ is a $\rho$-sheaf, a filtration of $(\mathcal{F}, \sigma)$ is the same thing as a $\mathbb{Z}$-filtration $(\mathcal{F}_m)_{m \in \mathbb{Z}}$ of $\mathcal{F}$ (as in Proposition \ref{prop: filtrations of torsion free sheaves}) that we call the underlying associated filtration, which satisfies some ``compatibility conditions" that will be explained in the remainder of this section. First, we need to explain what is a filtration of $\GL(V)$-bundles on any scheme $U$.

\begin{prop}  \label{prop: filtrations of SL(V) bundles}
Let $U$ be a scheme over $K$. Let $\mathcal{F}$ be a $\GL(V)$-bundle on $U$. Then the following two groupoids are equivalent.
\begin{enumerate}[(1)]
    \item The groupoid of $\GL(V)$-bundles $\mathcal{P}$ on $U \times_{K} \Theta_{K}$ along with an isomorphism of the restriction $\mathcal{P}|_{U \times_{K} 1}$ with $\mathcal{F}$.
    \item The groupoid of decreasing $\mathbb{Z}$-weighted filtrations $(\mathcal{F}_m)_{m \in \mathbb{Z}}$ of $\mathcal{F}$ by vector subbundles $\mathcal{F}_m$. An isomorphism between two filtrations is just an automorphism of $\mathcal{F}$ that identifies the filtrations.
\end{enumerate}
\end{prop}
\begin{proof}
The equivalence is given by the Rees construction from \cite[Prop. 1.0.1]{halpernleistner2021structure}, as in Proposition \ref{prop: filtrations of torsion free sheaves}, we omit the details.
\end{proof}

Now we are ready to give a first description of $\Theta$-filtrations for a $\rho$-sheaf. Let $(\mathcal{F}, \sigma)$ be a $\rho$-sheaf on $X_K$. A filtration of $(\mathcal{F}, \sigma)$ will consist of a pair $((\mathcal{F}_{m})_{m \in \mathbb{Z}}, \, \mathcal{P})$. Here $(\mathcal{F}_{m})_{m \in \mathbb{Z}}$ is a filtration of $\mathcal{F}$, as described in Proposition \ref{prop: filtrations of torsion free sheaves}. Let $U \hookrightarrow X_{K}$ be the maximal big open subset of $X_{K}$ where the associated graded $\bigoplus_{m \in \mathbb{Z}} \mathcal{F}_{m} / \mathcal{F}_{m+1}$ is locally-free. The second element of the pair $\mathcal{P}$ is a $G$-bundle on $U \times_{K} \Theta_{K}$ such that the associated $\GL(V)$-bundle $\rho_*(\mathcal{P})$ on $U \times_{K} \Theta_{K}$ corresponds to the restricted filtration $(\mathcal{F}_{m}|_{U})_{m \in \mathbb{Z}}$ under the equivalence described in Proposition \ref{prop: filtrations of SL(V) bundles}. 

The filtration $(\mathcal{F}_{m}|_{U})_{m \in \mathbb{Z}}$ on $U$ completely determines the underlying associated filtration $(\mathcal{F}_{m})_{m \in \mathbb{Z}}$ on $X_{K}$. This is because each $\mathcal{F}_{m}$ is saturated on $\mathcal{F}$, so it is determined in codimension 1 \cite[pg. 90-92]{langton}. We can describe $\mathcal{F}_m$ more explicitly in terms of reflexive hulls. Let $\widetilde{\mathcal{F}}_{m}$ denote the pushforward of the vector bundle $\mathcal{F}_{m}|_{U}$ under the open immersion $U \hookrightarrow X_{K}$. We have that $\widetilde{\mathcal{F}}_{m}$ is a subsheaf of the reflexive hull $\mathcal{F}^{\vee\vee}$ of $\mathcal{F}$. Then we can recover $\mathcal{F}_{m}$ as the intersection $\widetilde{\mathcal{F}}_{m} \cap \mathcal{F}$ taken inside $\mathcal{F}^{\vee \vee}$. In particular the filtration is completely determined by the $G$-bundle $\mathcal{P}$ on $U \times_{K} \Theta_{K}$. We will, however, keep the underlying associated filtration as part of the notation.

A morphism between two such filtrations $((\mathcal{F}_{m})_{m \in \mathbb{Z}}, \mathcal{P})$ and $((\mathcal{F}'_{m})_{m \in \mathbb{Z}}, \mathcal{P}')$ consists of an isomorphism of $G$-bundles $\varphi: \mathcal{P} \xrightarrow{\sim} \mathcal{P}'$ defined over some big open subset $U \subset X_{K}$ where both $\mathcal{P}$ and $\mathcal{P}'$ are defined. We require that the associated isomorphism of $\GL(V)$-bundles on $U \times_{K} \Theta_{K}$ given by $\rho_*(\varphi): (\mathcal{F}_m|_{U})_{m\in \mathbb{Z}} \xrightarrow{\sim} (\mathcal{F}'_{m}|_{U})_{m \in \mathbb{Z}}$ admits a (unique) extension to an isomorphism of filtrations $\widetilde{\varphi}:(\mathcal{F}_m)_{m\in \mathbb{Z}} \xrightarrow{\sim} (\mathcal{F}'_{m})_{m \in \mathbb{Z}}$ defined over all of $X_{K}$. By the description of the filtrations as intersections of reflexive hulls, it follows that we only need to check that the isomorphism $\rho_*(\varphi)|_{U \times_{K} 1}$ extends to an automorphism $\mathcal{F} \xrightarrow{\sim} \mathcal{F}$. This is automatically satisfied if $\mathcal{F}$ is reflexive, but not otherwise.
\begin{defn}
Let $(\mathcal{F}, \sigma)$ be a $\rho$-sheaf on $X_K$. We will denote by $\Filt(\mathcal{F}, \sigma)$ the groupoid of pairs $((\mathcal{F}_{m})_{m \in \mathbb{Z}}, \mathcal{P})$, with isomorphisms as described above.
\end{defn}

\begin{prop} \label{prop: nonexplicit filtrations singular G bundles}
Let $(\mathcal{F}, \sigma)$ be a $\rho$-sheaf on $X_K$. There is a natural equivalence between the following two groupoids.
\begin{enumerate}[(1)]
    \item The groupoid of $\Theta$-filtrations $f: \Theta_{K} \rightarrow \Bun_{\rho}(X)$ of the $\rho$-sheaf $(\mathcal{F}, \sigma)$.
    \item The groupoid $\Filt(\mathcal{F}, \sigma)$.
\end{enumerate}
\end{prop}
\begin{proof}
The groupoid of $\Theta$-filtrations classifies $\mathbb{G}_m$-equivariant $\rho$-sheaves $(\widetilde{\mathcal{F}}, \widetilde{\sigma})$ on $X \times \mathbb{A}^{1}_{K}$ along with an isomorphism of the restriction $(\widetilde{\mathcal{F}}, \widetilde{\sigma})|_{X \times 1}$ with the original $(\mathcal{F}, \sigma)$. So $\widetilde{\mathcal{F}}$ is a $\Theta$-filtration of $\mathcal{F}$ and $\widetilde{\sigma}$ is a $\mathbb{G}_m$-equivariant section $\widetilde{\sigma}: X \times \mathbb{A}^1_{K} \rightarrow \Red_{G}(\widetilde{\mathcal{F}})$ extending $\sigma: X \times 1 \rightarrow \Red_{G}(\mathcal{F})$. By Proposition \ref{prop: filtrations of torsion free sheaves}, the bundle $\widetilde{\mathcal{F}}$ is obtained by applying the Rees construction to some $\mathbb{Z}$-filtration $(\mathcal{F}_{m})_{m \in \mathbb{Z}}$. Let $U$ be the maximal open in $X_{K}$ where $\bigoplus_{m \in \mathbb{Z}} \mathcal{F}_{m} / \mathcal{F}_{m+1}$ is locally-free. Then $U \times_{K}\mathbb{A}^{1}_{K}$ is a big open in $X \times \mathbb{A}^{1}_{K}$ and $\widetilde{\mathcal{F}}|_{U \times_{K}\mathbb{A}^1_{K}}$ is locally-free. By Proposition \ref{prop: relation scheme and G-reductions}, the $\mathbb{G}_m$-equivariant section $\widetilde{\sigma}$ is equivalent to the data of a $\mathbb{G}_m$-equivariant $G$-reduction of the $\GL(V)$-bundle $\widetilde{\mathcal{F}}|_{U \times_{K} \mathbb{A}^1_{K}}$. This is by definition a $\mathbb{G}_m$-equivariant $G$-bundle $\mathcal{P}$ on $ U \times_{K} \mathbb{A}^1_{K}$ such that the associated $\GL(V)$-bundle $\rho_*(\mathcal{P})$ on $U \times_{K} \Theta_{K}$ recovers the restriction of the filtration $(\mathcal{F}_{m}|_{U})_{m \in \mathbb{Z}}$ under the equivalence in Proposition \ref{prop: filtrations of SL(V) bundles}. This shows how to define the functor on objects, and it also shows that it is essentially surjective. The definition of the functor on morphisms and the fact that it is fully faithful follows from the definition of  the isomorphisms in the groupoid $\Filt(\mathcal{F}, \sigma)$.
\end{proof}
\end{subsection}
\begin{subsection}{Filtrations of $\rho$-sheaves in the split case} \label{subsection: filtrations}
We will describe the elements of $\Filt(\mathcal{F}, \sigma)$ more explicitly in the case when $G_{K}$ is split. For that purpose, we need a better description of $G$-bundles on $U \times_{K} \Theta_{K}$. This is done in \cite[1.F]{heinloth2018hilbertmumford}, which gives a group theoretic description in terms of parabolic reductions and a generalized Rees construction\footnote{We note that some of the discussion in \cite[1.F]{heinloth2018hilbertmumford} applies without the hypothesis that $G_{K}$ is split. However, the results are more concrete under this assumption.}.

We start by recalling some facts about parabolic subgroups of reductive groups. Let $H$ be a reductive group over $K$. Let $\lambda$ be a cocharacter $\lambda: \mathbb{G}_m \rightarrow H$. We can associate a corresponding parabolic subgroup $P_{\lambda}$. For any $K$-scheme $T$, its $T$-points are
\[P_{\lambda}(T) = \left\{ \, p \in H(T) \, \mid \, \lim \limits_{x \to 0} \, \lambda(x) \,p\, \lambda(x)^{-1} \;\text{exists} \,\right\}  \]
Furthermore, the cocharater $\lambda$ induces a Levi decomposition $P_{\lambda} = L_{\lambda} \ltimes U_{\lambda}$, where
\[ U_{\lambda}(T) = \left\{ \, u \in H(T) \, \mid \, \lim \limits_{x \to 0} \, \lambda(x) \,u\, \lambda(x)^{-1} = 1_{H} \,\right\} \]
\[L_{\lambda}(T) = \left\{ \, l \in H(T) \, \mid \, \lambda(x) \,l\, \lambda(x)^{-1} = l \,\right\}\]
\begin{defn}
Let $H$ be a reductive linear algebraic group over $K$. Let $Y$ be a $K$-scheme and let $\mathcal{H}$ be a $H$-bundle on $Y$. A weighted parabolic reduction is a pair $(\lambda, \mathcal{H}_{\lambda})$, where $\lambda: \mathbb{G}_m \rightarrow H$ is a cocharacter and $\mathcal{H}_{\lambda}$ is a reduction of structure group of $\mathcal{H}$ to the parabolic subgroup $P_{\lambda}$.
\end{defn}

Let $\theta: H \rightarrow F$ be a homomorphism of reductive groups over $K$. For every cocharacter $\lambda: \mathbb{G}_m \rightarrow H$, we can compose with $\theta$ to obtain a cocharacter $\theta \circ \lambda$ of $F$. Notice that $\theta$ restricts to a well defined homomorphism of parabolic subgroups $\theta_{P}: P_{\lambda} \rightarrow P_{\theta \circ \lambda}$. Suppose that we have an $H$-bundle $\mathcal{H}$ defined on a $K$-scheme $Y$. Then we can define an $F$-bundle $\theta_{*} \mathcal{H}$ via extension of structure group. If we have a parabolic reduction $\mathcal{H}_{\lambda}$ of $\mathcal{H}$ to the subgroup $P_{\lambda}$, then by extension of structure group using $\theta_{P}$ we get a parabolic reduction of $\theta_* \mathcal{H}$ to the subgroup $P_{\theta \circ \lambda }$. This way we can define a map $\theta_*$ between weighted parabolic reductions of $\mathcal{H}$ and weighted parabolic reductions of $\theta_* \mathcal{H}$.

Having set up the notation for weighted parabolic reductions, we proceed to explain how this notion relates to $G$-bundles on $U \times_{K} \Theta_{K}$ following \cite{heinloth2018hilbertmumford}. Let $U$ be a $K$-scheme. Suppose that $\mathcal{G}$ is a $G$-bundle on $U$. We can equivalently view this as a $G_{K}$-bundle on $U$. Let $(\lambda, \mathcal{G}_{\lambda})$ be a weighted parabolic reduction of $\mathcal{G}$. The morphism $\lambda: \mathbb{G}_{m} \rightarrow G_{K}$ can be used to define a homomorphism $\text{conj}_{\lambda}: P_{\lambda} \times \mathbb{G}_m \rightarrow P_{\lambda} \times \mathbb{G}_m$ of group schemes over $\mathbb{G}_m$ defined by $(p, x) \mapsto (\lambda(x) p \lambda(x)^{-1}, \, x)$. This extends to a homomorphism  $\text{gr}_{\lambda}: P_{\lambda} \times_{K} \mathbb{A}^1_{K} \rightarrow  P_{\lambda} \times_{K} \mathbb{A}^{1}_{K}$. Let $t\in \mathbb{A}^{1}_{K}$ be a closed point. 
The restriction of $\text{gr}_{\lambda}$ to a slice
$P_{\lambda} \times \{t\}$ is the identity when $t=1$, it is an isomorphism when $t\neq 0$, 
and it is the projection of the parabolic 
subgroup $P_{\lambda}$ to its Levi $L_{\lambda}$ when $t=0$.
Now we can define a $P_{\lambda}$-bundle on $U \times_{K} \Theta_{K}$ as follows. Note that 
$\mathcal{G}_{\lambda} \times_{K} \mathbb{A}^1_{K}$ is a $P_{\lambda}\times_{K} \mathbb{A}^1_{K}$-torsor on $U\times \mathbb{A}^1_{K}$. Using the homomorphism $\text{gr}_{\lambda}$ we obtain, by extension of structure group, a new $P_{\lambda}\times_{K} \mathbb{A}^1_{K}$-torsor which is denoted $( \mathcal{G}_{\lambda} \times_{K} \mathbb{A}^1_{K}) \times^{\text{gr}_{\lambda}}_{\mathbb{A}^1_{K}} (P_{\lambda} \times_{K} \mathbb{A}^{1}_{K})$. This bundle is $\mathbb{G}_m$-equivariant because the homomorphism $\text{gr}_{\lambda}$ is $\mathbb{G}_m$-equivariant (with the adjoint action on $P_\lambda$ and multiplication on $\mathbb{A}^1_{K}$), and therefore we can take the quotient by $\mathbb{G}_m$
\[\Rees(\mathcal{G}_{\lambda}, \lambda) \vcentcolon = \left[ \,\left( ( \mathcal{G}_{\lambda} \times_{K} \mathbb{A}^1_{K}) \times^{\text{gr}_{\lambda}}_{\mathbb{A}^1_{K}} (P_{\lambda} \times_{K} \mathbb{A}^{1}_{K}) \right)\, / \, \mathbb{G}_m \,  \right]  \]
By construction, the restriction $\Rees(\mathcal{G}_{\lambda}, \lambda)|_{U\times 1}$ is canonically isomorphic to the $P_{\lambda}$-bundle $\mathcal{G}_{\lambda}$, and the restriction to $U\times 0$ is a $P_\lambda$-bundle which admits a reduction of structure group to $L_\lambda$.

We can extend the structure group from $P_{\lambda}$ to $G_{K}$ in order to obtain a $G_K$-bundle $\Rees(\mathcal{G}_{\lambda}, \, \lambda) \times^{P_{\lambda}} G_{K}$. By construction, the restriction $\Rees(\mathcal{G}_{\lambda}, \, \lambda) \times^{P_{\lambda}} G_{K}|_{U \times_{K}1}$ is canonically isomorphic to the original $G$-bundle $\mathcal{G}$.

\begin{prop} \label{prop: filtrations of G bundles over U}
Suppose that $G_{K}$ is split. Let $U$ be a connected scheme of finite type over $K$. Let $\mathcal{G}$ be a $G$-bundle on $U$. Then every $G$-bundle $\mathcal{E}$ on $U \times_{K} \Theta_{K}$ such that $\mathcal{E}|_{U \times_{K}1} \cong \mathcal{G}$ is of the form $\Rees(\mathcal{G}_{\lambda}, \, \lambda) \times^{P_{\lambda}} G_{K}$, where
\begin{enumerate}[(a)]
    \item $\lambda: \mathbb{G}_m \rightarrow G_{K}$ is a cocharacter that is uniquely determined up to conjugation.
    \item $\mathcal{G}_{\lambda}$ is a $P_{\lambda}$-reduction of $\mathcal{G}$.
\end{enumerate}
In other words, all $\Theta$-filtrations of $\mathcal{G}$ arise from weighted parabolic reductions of $\mathcal{G}$.
\end{prop}
\begin{proof}
This is \cite[Lem. 1.13]{heinloth2018hilbertmumford}. The lemma found there is stated for a connected smooth projective curve over $K$, but the same argument applies to any connected scheme of finite type over $K$.
\end{proof}

Suppose that $G_K$ is split. Let $(\mathcal{F}, \sigma)$ be a $\rho$-sheaf over $X_K$. Let $U$ be a big open subset of $X_{K}$ where $\mathcal{F}$ is locally-free. Let $\mathcal{G}$ denote the $G_{K}$-bundle on $U$ corresponding to the $\rho$-sheaf $(\mathcal{F}, \sigma)$. By Proposition \ref{prop: filtrations of G bundles over U}, all $\Theta$-filtrations of $\mathcal{G}$ arise from the Rees construction applied to a weighted parabolic reduction $(\lambda, \mathcal{G}_{\lambda})$ of $\mathcal{G}$. We have an associated weighted parabolic reduction $(\rho_*(\lambda), \rho_*(\mathcal{G}_{\lambda}))$ of the $\GL(V)$-bundle $\mathcal{F}|_{U}$. The Rees construction applied to this weighted parabolic reduction yields a filtration $(\mathcal{F}'_{m})_{m \in \mathbb{Z}}$ of $\mathcal{F}|_{U}$. This induces a filtration $(\mathcal{F}_m)_{m \in \mathbb{Z}}$ of $\mathcal{F}$ as follows. Let $\widetilde{\mathcal{F}'_{m}}$ denote the pushforward of $\mathcal{F}'_{m}$ under the open inclusion $U \hookrightarrow X_{K}$. Note that $\widetilde{\mathcal{F}'_{m}}$ is a subsheaf of the reflexive hull $\mathcal{F}^{\vee \vee}$. Then, we set $\mathcal{F}_m \vcentcolon = \widetilde{\mathcal{F}'_{m}} \cap \mathcal{F}^{\vee \vee}$. It follows by construction that $(\mathcal{F}_m)_{m \in \mathbb{Z}}$ satisfies the properties in Proposition \ref{prop: filtrations of torsion free sheaves}, so it is indeed a filtration of $\mathcal{F}$. 

The pair $((\mathcal{F}_{m})_{m \in \mathbb{Z}}, \Rees(\mathcal{G}_{\lambda}, \lambda) \times^{P_{\lambda}} G)$ is an element in $\Filt(\mathcal{F}, \sigma)$. Conversely, it follows from our discussion above that every element of $\Filt(\mathcal{F}, \sigma)$ arises this way. 
\begin{defn} \label{defn: weighted parabolic reduction singular G bundles}
Let $(\mathcal{F}, \sigma)$ be a $\rho$-sheaf on $X_{K}$. A weighted parabolic reduction of $(\mathcal{F}, \sigma)$ consists of a pair $(\lambda, \mathcal{G}_{\lambda})$, where
\begin{enumerate}[(a)]
    \item  Implicit in the notation we have the choice of (any) big open subset $U \subset X_{K}$ such that $\mathcal{F}|_{U}$ is locally-free. We denote by $\mathcal{G}$ the $G_{K}$-bundle on $U$ corresponding to $(\mathcal{F}, \sigma)$.
    \item $(\lambda, \mathcal{G}_{\lambda})$ is a weighted parabolic reduction of $\mathcal{G}$.
\end{enumerate}
The filtration $\mathcal{F}_{m} \vcentcolon = \widetilde{\mathcal{F}'_{m}} \cap \mathcal{F}^{\vee \vee}$ of $\mathcal{F}$ described in the discussion above is called the underlying filtration associated to the parabolic reduction $(\lambda, \mathcal{G}_{\lambda})$. We say that $(\lambda, \mathcal{G}_{\lambda})$ is nondegenerate if $\lambda$ is not the trivial cocharacter. 
\end{defn}
We record the description of filtrations of $\rho$-sheaves obtained from the discussion above.
\begin{prop} \label{prop: explicit description of filtrations singular G bundles}
Suppose that $G_{K}$ is split. Let $(\mathcal{F}, \sigma)$ be a $\rho$-sheaf on $X_{K}$. Then every filtration in $\Filt(\mathcal{F}, \sigma)$ is of the form $((\mathcal{F}_m)_{m \in \mathbb{Z}}, \Rees(\mathcal{G}_{\lambda}, \lambda) \times^{P_{\lambda}} G)$, where $(\lambda, \mathcal{G}_{\lambda})$ is a weighted parabolic reduction of $(\mathcal{F}, \sigma)$ with underlying associated filtration $(\mathcal{F}_m)_{m \in \mathbb{Z}}$. 

Moreover, the filtration $((\mathcal{F}_{m})_{m \in \mathbb{Z}}, \Rees(\mathcal{G}_{\lambda}, \lambda) \times^{P_{\lambda}} G)$ is nondegenerate if and only if $(\lambda, \mathcal{G}_{\lambda})$ is nondegenerate.
\qed
\end{prop}

We can also give a concrete description of graded points of the stack $\Bun_{\rho}(X)$. A graded $\rho$-sheaf over $X_K$ is defined to be a morphism $\left[ \, \Spec(K) / \mathbb{G}_m \, \right] \rightarrow \Bun_{\rho}(X)$. By unraveling the definitions, this notion admits the following concrete description. A graded $\rho$-sheaf on $X_K$ consists of a $\rho$-sheaf $(\mathcal{F}, \sigma)$ on $X_K$, plus a $\mathbb{Z}$-grading $\mathcal{F} = \bigoplus_{m \in \mathbb{Z}} \overline{\mathcal{F}}_m$ of the torsion-free sheaf $\mathcal{F}$ such that each summand $\overline{\mathcal{F}}_m$ is also torsion-free. We get a well-defined $\mathbb{G}_m$-action on the scheme of $G$-reductions $\Red_{G}(\mathcal{F})$ and we impose the condition that the section $\sigma$ is $\mathbb{G}_m$-invariant. Hence we can think of morphisms $g: \left[ \, \Spec(K) / \mathbb{G}_m \, \right] \rightarrow \Bun_{\rho}(X)$ concretely as a graded pair $\left(\bigoplus_{m \in \mathbb{Z}} \overline{\mathcal{F}}_m, \sigma\right)$.

We end this section by explaining a concrete description of the associated graded of a filtration. Suppose that $G_{K}$ is split. Let $(\mathcal{F}, \sigma)$ be a $\rho$-sheaf on $X_K$. Let $f: \Theta_K \rightarrow  \Bun_{\rho}(X)$ be a filtration of $(\mathcal{F}, \sigma)$. By Proposition \ref{prop: explicit description of filtrations singular G bundles} $f$ comes from a weighted parabolic reduction $(\lambda,  \mathcal{G}_{\lambda})$ with underlying associated filtration $(\mathcal{F}_{m})_{m \in \mathbb{Z}}$. We can restrict $f$ to the point $0 \in \Theta_K$ in order to obtain a graded $\rho$-sheaf $f|_{[0/\, \mathbb{G}_m]}: \left[ \, \Spec(K) / \mathbb{G}_m \, \right] \rightarrow \Bun_{\rho}(X)$. The underlying graded torsion-free sheaf of $f|_{[0/\, \mathbb{G}_m]}$ is the associated graded $\bigoplus_{m \in \mathbb{Z}} \mathcal{F}_{m} / \mathcal{F}_{m+1}$. We can describe the $G$-reduction over $U$ of the $\rho$-sheaf $f|_{[0/\, \mathbb{G}_m]}$ as follows. Let $P_{\lambda} = L_{\lambda} \ltimes U_{\lambda}$ be the Levi decomposition defined by $\lambda$. Then over the big open subset $U$, the corresponding $G$-bundle is the one obtained by applying extension of structure group under the composition $P_{\lambda} = L_\lambda \ltimes U_{\lambda} \twoheadrightarrow L_{\lambda} \hookrightarrow G_{K}$ to the $P_{\lambda}$-bundle $\mathcal{G}_{\lambda}$.
\end{subsection}
\end{section}

\begin{section}{$\Theta$-stratification on $\Bun_{\rho}(X)$} \label{section: numerical invariant}
In this section we define a polynomial numerical invariant $\nu$ on the stack $\Bun_{\rho}(X)$. We prove that $\nu$ induces a $\Theta$-stratification, as in \cite[Defn. 2.1.2]{halpernleistner2021structure} \cite[Defn. 2.23]{torsion-freepaper}.
\begin{subsection}{The numerical invariant} \label{subsection: numerical invariant}
We define a family of line bundles $L_n$ on $\Bun_{\rho}(X)$ following \cite[\S 2]{torsion-freepaper}.
\begin{defn} \label{defn: line bundles}
Let $n \in \mathbb{Z}$. We define $M_n$ to be the line bundle on $\Bun_{\rho}(X)$ determined as follows. Let $T \in \Aff_k$ and let $f:T \rightarrow \Bun_{\rho}(X)$ be a morphism, represented by a $\rho$-sheaf $(\mathcal{F}, \sigma)$ on $X_T$. Let $\pi_{T} : X_{T} \to T$ denote the structure morphism. Then, we set
\[ f^{*}\, M_n \vcentcolon = \det \, \left(R\pi_{T \, *}\left(\mathcal{F}(n) \right) \right) \]
Moreover, we define $b_d \vcentcolon = \bigotimes_{i = 0}^d M_{i}^{\otimes (-1)^{d+i} \binom{d}{i}}$. We set $L_n$ to be the rational line bundle given by
\[ f^{*}\, L_n \vcentcolon = f^* \, M_n \otimes (f^*\, b_d)^{-\otimes \, \overline{p}_{\mathcal{F}}(n)} \]
\end{defn}
\begin{remark}
$L_n$ is an element of the rational Picard group $\Pic(\Bun_{\rho}(X))\otimes \mathbb{Q}$ of the stack. This is what we mean by rational line bundle. We will henceforth omit the adjective ``rational" and refer to $L_n$ as a line bundle. All the arguments will apply to the rational Picard group in the same way as for the usual Picard group.
\end{remark}
The line bundle $L_n$ is the restriction under the forgetful map $\Bun_{\rho}(X) \rightarrow \Coh^{tf}(X)$ of the line bundle $L_n$ on $\Coh^{tf}(X)$ defined in \cite[\S2]{torsion-freepaper}. 

Next, we define a rational quadratic norm on graded $\rho$-sheaves (see \cite[\S2.5, \S4.1]{torsion-freepaper} for more details on rational quadratic norms).
\begin{defn}
We define the rational quadratic norm $b$ on graded points of $\Bun_{\rho}(X)$ as follows. Let $g: \left[\Spec(K)/ \, \mathbb{G}_m \right] \rightarrow \Bun_{\rho}(X)$ be represented by a graded $\rho$-sheaf $\left(\bigoplus_{m \in \mathbb{Z}} \overline{\mathcal{F}}_m, \sigma \right)$. Then we set
\[ b(g) := \sum_{m \in \mathbb{Z}} m^2 \cdot \rk(\overline{\mathcal{F}}_m)\]
\end{defn}

We define the polynomial numerical invariant $\nu$ to be the one determined by the family of line bundles $L_n$ and the rational quadratic norm $b$, as it was done for the case of pure sheaves in \cite[\S 4.1]{torsion-freepaper}. For our purposes, it is sufficient to know the value of the numerical invariant associated to a nondegenerate $\Theta$-filtration. Let us describe this next.

Let $K \supset k$ be a field extension, and let $(\mathcal{F}, \sigma)$ be a $\rho$-sheaf on $X_K$. Let $f: \Theta_K \rightarrow \Bun_{\rho}(X)$ be a nondegenerate filtration of $(\mathcal{F}, \sigma)$. By Proposition \ref{prop: filtrations of singular G bundles vs filtrations of torsion-free sheaves} this amounts to a filtration $(\mathcal{F}_m)_{m \in \mathbb{Z}}$ of the torsion-free sheaf $\mathcal{F}$ satisfying some extra compatibility conditions. The restriction  $f|_{\left[0 /\mathbb{G}_m\right]}: {\left[0 /\mathbb{G}_m\right]}\rightarrow \Bun_{\rho}(X)$ yields a graded $\rho$-sheaf on $X_K$, whose underlying torsion-free sheaf is given by the associated graded $\bigoplus_{m \in \mathbb{Z}} \mathcal{F}_m / \mathcal{F}_{m+1}$. We define $\nu(f)$ by
 \begin{equation}\label{eqn_nu}
 \nu(f) \vcentcolon = \frac{\wgt\left(L_n|_{0}\right)}{\sqrt{b(f|_{[0 / \mathbb{G}_m ]})}}
 \end{equation}
 
 The same computation as in \cite[\S4.1]{torsion-freepaper} shows that $\nu$ is a polynomial numerical invariant (valued in $\mathbb{R}[n]$) which admits a description in terms of reduced Hilbert polynomials. Let $\frac{A_d}{d!}$ be the leading coefficient of the Hilbert polynomial $P_{\mathcal{O}_X}$ of the structure sheaf with respect to $\mathcal{O}(1)$.
 \begin{prop} \label{prop: value numerical invariant}
 Let $K \supset k$ be a field extension. Let $(\mathcal{F}, \sigma)$ be a $\rho$-sheaf on $X_K$. Choose a nondegenerate filtration $f: \Theta_{K} \rightarrow \Bun_{\rho}(X)$ of $(\mathcal{F}, \sigma)$. Suppose that $f$ has underlying associated filtration $(\mathcal{F}_m)_{m\in \mathbb{Z}}$ of $\mathcal{F}$. Then, we have
 \begin{equation} \label{eqn_1}
 \nu(f) = \sqrt{A_d} \, \cdot \, \frac{\sum_{m \in \mathbb{Z}} m \cdot  (\overline{p}_{\mathcal{F}_m/ \mathcal{F}_{m+1}} - \overline{p}_{\mathcal{F}}) \cdot \rk(\mathcal{F}_m/\mathcal{F}_{m+1})}{\sqrt{\sum_{m \in \mathbb{Z}} m^2 \cdot \rk(\mathcal{F}_m / \mathcal{F}_{m+1})}}
 \end{equation}
 \end{prop}
 \begin{proof}
 The value of the denominator $\sqrt{b(f|_{[0 / \mathbb{G}_m ]})}$ follows by definition. The numerator has been computed in \cite[\S 4.1]{torsion-freepaper}. The extra factor $\sqrt{A_{d}}$ appears because we define the rank to be the dimension of the fiber of $\mathcal{F}$ at the generic point, as opposed to the rank defined in \cite[\S2.3]{torsion-freepaper} using the leading term of the Hilbert polynomial.
 \end{proof}
Given a weighted parabolic reduction $f= (\lambda, \mathcal{G}_{\lambda})$ of a $\rho$-sheaf $(\mathcal{F}, \sigma)$, we denote by $\nu(f)$ the value of the numerical invariant for the $\Theta$-filtration obtained by applying the Rees construction to $f$. If $(\mathcal{F}_{m})_{m \in \mathbb{Z}}$ is the underlying associated filtration of $(\lambda, \mathcal{G}_{\lambda})$, then the same formula (\ref{eqn_1}) applies.
\end{subsection}

\begin{subsection}{Affine Grassmannians for $\rho$-sheaves} \label{subsection: affine grassmannians}
The setup for this subsection is as follows. Let $S$ be a quasicompact $k$-scheme. Let $\mathcal{F}$ be a family of torsion-free sheaves of rank $r$ on $X_S$. Suppose that we are given an effective Cartier divisor $D \hookrightarrow X_S$ that is flat over $S$. Set $Q = X_S \setminus D$ (note that this is not a big open subset). The paper \cite{torsion-freepaper} defines a notion of affine Grassmannian $\Gr_{X_S, D, \mathcal{F}}$. We recall the definition and some properties proven in \cite[\S3]{torsion-freepaper}.
\begin{defn} \label{defn: step 1 grassmannian}
$\Gr_{X_S, D, \mathcal{F}}$ is the functor from $\Aff_S^{\, \, op}$ to sets given as follows. For every affine scheme $T$ in $\Aff_S$, the set $\Gr_{X_S, D, \mathcal{F}} (T)$ consists of equivalence classes of pairs $(\mathcal{E}, \theta)$, where
\begin{enumerate}[(1)]
    \item $\mathcal{E}$ is a family of torsion-free sheaves on $X_T$.
    \item $\theta$ is a morphism $\theta:  \mathcal{F}|_{X_{T}} \rightarrow  \mathcal{E}$ such that the restriction $\theta|_{Q_{T}}$ is an isomorphism.
\end{enumerate}
Two such pairs $(\mathcal{E}_1, \theta_1)$ and $(\mathcal{E}_2, \theta_2)$ are considered equivalent if there is an isomorphism $\mathcal{E}_1 \xrightarrow{\sim} \mathcal{E}_2$ that identifies the morphisms $\theta_1$ and $\theta_2$.
\end{defn}

Let $j$ be the open immersion $j: Q \hookrightarrow X_{S}$. Choose $T \in \Aff_S$ and $(\mathcal{E}, \theta) \in \Gr_{X_S, D, \mathcal{F}}(T)$. The unit $\mathcal{E} \rightarrow j_{T \, *} \, j_T^* \, \mathcal{E}$ is a monomorphism. Hence we can use the isomorphism $j_{T *}(j_{T})^*\theta$ in order to view $\mathcal{E}$ as a subsheaf of $j_{T \, *} \, j_T^* \, \mathcal{F}|_{X_T}$. In particular, we see that $\theta$ is automatically injective, and we have $\mathcal{F}|_{X_{T}} \subset \mathcal{E} \subset j_{T \, *} \, j_T^* \, \mathcal{F}|_{X_T}$. There is an increasing sequence of sheaves 
\[\mathcal{F}|_{X_T} \subset \mathcal{F}|_{X_T}(D_{T}) \subset \mathcal{F}|_{X_T}(2D_{T}) \subset \cdots \subset \mathcal{F}|_{X_T} (nD_{T}) \subset \cdots\] 
 We use this to define a truncated version of $\Gr_{X, D, \mathcal{E}}$.
\begin{defn}
Let $N$ be a positive integer. $\Gr^{\leq N}_{X_S, D, \mathcal{F}}$ is the subfunctor of $\Gr_{X_S, D, \mathcal{F}}$ defined as follows. For all $T$ in $\Aff_T$, we set
\[ \Gr^{\leq N}_{X_S, D, \mathcal{F}} (T) \; := \; \left\{ \begin{matrix} \; \; \; \text{pairs $(\mathcal{E}, \theta)$ in  $\Gr_{X_S, D, \mathcal{F}}(T)$ such that} \; \; \; \\ \text{$\mathcal{E} \subset \mathcal{F}|_{X_T}(ND_T)$} \end{matrix} \right\}\]
For any rational polynomial $P \in \mathbb{Q}[n]$, we denote by $\Gr_{X_S, D, \mathcal{F}}^{\leq N, P}$ the open and closed subfunctor of $\Gr_{X_S, D, \mathcal{F}}^{\leq N}$ consisting of points $(\mathcal{E}, \theta)$ where $\mathcal{E}$ has Hilbert polynomial $P$ on every fiber.
\end{defn}

\begin{prop}[{\cite[Prop. 3.13]{torsion-freepaper}}]\label{prop: ind-representability of step 1 grassmannian}
We have $\Gr_{X_{S}, D, \mathcal{F}} = \underset{N>0}{\colim} \; \Gr^{\leq{N}}_{X_{S}, D, \mathcal{F}}$ as presheaves on $\Aff_S$. For each $N \geq 0$, we have $\Gr_{X_S, D, \mathcal{F}}^{\leq N} = \bigsqcup_{P \in \mathbb{Q}[n]} \Gr_{X_S, D, \mathcal{F}}^{\leq N, P}$.
Each $\Gr_{X_S, D, \mathcal{F}}^{\leq N, P} $ is represented by a projective scheme of finite presentation over $S$. This induces a presentation of $\Gr_{X_{S}, D, \mathcal{F}}$ as an ind-projective ind-scheme over $S$.
\end{prop}

Suppose now that we have a section $\sigma : Q \rightarrow \Red_{G}(\mathcal{F})$ defined over $Q$. We use this datum to define a version of the affine Grassmannian for $\rho$-sheaves.
\begin{defn} \label{defn: affine grassmannian}
The affine Grassmannian $\Gr_{X_S, D, \mathcal{F}, \sigma}$ is the functor from $\left(\Aff_S\right)^{op}$ into sets defined as follows. For each affine scheme $T$ in $\Aff_S$, $\Gr_{X_S, D, \mathcal{F}, \sigma}(T)$ is the set of equivalence classes of the following data
\begin{enumerate}[(1)]
    \item A $\rho$-sheaf $(\mathcal{E}, \zeta)$ on $X_T$.
    \item A morphism $\theta: \mathcal{F}|_{X_T} \rightarrow \mathcal{E}$ such that the restriction $\theta|_{Q_{T}}$ is an isomorphism. Moreover, we require that the following diagram induced by the isomorphism $\theta|_{Q_T}$ commutes.
\[
\begin{tikzcd}
  Q_T \ar[r, "\zeta|_{Q_T}"] \ar[rd, "\sigma|_{Q_T}", labels=below left] & \Red_{G}(\mathcal{E}|_{Q_T})  \\   & \Red_{G}(\mathcal{F}|_{Q_T}) \ar[u, symbol = \xrightarrow{\; \; \sim \; \;} ]
\end{tikzcd}
\]
\end{enumerate}
Two elements above are considered equivalent if there is an isomorphism of $\rho$-sheaves that is compatible with the morphism $\theta$.
\end{defn}

There is a forgetful morphism $\Gr_{{X_S, D, \mathcal{F}, \sigma}} \rightarrow \Gr_{X_S, D, \mathcal{F}}$ defined by $(\mathcal{E}, \zeta, \theta) \mapsto (\mathcal{E}, \theta)$.
 \begin{prop} \label{prop: the affine grassmannian is indproper}
 The forgetful morphism witnesses $\Gr_{{X_S, D, \mathcal{F}, \sigma}}$ as a closed sub ind-scheme of $\Gr_{X_S, D, \mathcal{F}}$.
 \end{prop}
 
Before proving Proposition \ref{prop: the affine grassmannian is indproper}, we first state a couple of lemmas.
 \begin{lemma} \label{lemma:vanishing section closed}
 Let $Y \longrightarrow S$ be a proper morphism of finite presentation. Let $\mathcal{G}$ and $\mathcal{F}$ be $\mathcal{O}_{Y}$-modules. Suppose that $\mathcal{F}$ is finitely presented and $S$-flat. Fix a morphism $\sigma : \mathcal{G} \rightarrow \mathcal{F}$. Let $Z_{\sigma}$ be the functor from $k$-schemes to sets that sends a $k$-scheme $T$ to
 \[  Z_{\sigma}(T) \vcentcolon = \; \left\{ \begin{matrix} \text{morphisms $g: T \rightarrow S$ such that}\\ \text{$ (\text{id}_Y \times_S g)^{*} \sigma = 0$} \end{matrix} \right\} \]
Then $Z_{\sigma}$ is represented by a closed subscheme of $S$.
 \end{lemma}
 \begin{proof}
 By \cite[\href{https://stacks.math.columbia.edu/tag/08K6}{Tag 08K6}]{stacks-project}, there is a separated $S$-scheme $\Hom(\mathcal{G}, \mathcal{F})$ parametrizing morphisms $\mathcal{G} \to \mathcal{F}$. The two morphisms $\sigma$ and $0$ yield two sections $f_{\sigma}, f_{0}: S \rightarrow \Hom(\mathcal{G}, \mathcal{F})$. The functor $Z_{\sigma}$ is represented by the locus where these two morphisms $f_{\sigma}, f_{0}$ agree. This is a closed subscheme of $S$, because $\Hom(\mathcal{G}, \mathcal{F})\rightarrow S$ is separated.
 \end{proof}

Let $S$ be a $k$-scheme. Suppose that $p: Y \rightarrow X_S$ is an affine morphism of finite presentation. Let $s: Q \rightarrow Y$ be a section of $p$  defined over the open complement $Q=X_{S}\backslash D$.
\begin{defn} \label{defn: scheme of extension of section outside of a divisor}
We define $\Sect^{Y}_{s}$ to be the functor from $(\Aff_S)^{op}$ to sets given as follows. For $T \in \Aff_S$, let $s_T: Q_T \rightarrow Y_T$ and $p_T: Y_T \rightarrow X_T$ denote the base-changes. We define
\[ \Sect^{Y}_{s}(T) \; \vcentcolon = \; \left\{ \begin{matrix} \; \; \; \text{sections $f:X_T \rightarrow Y_T$ of $p_T$} \; \; \;\\
\; \; \text{such that $f|_{Q_T} = s_T$} \; \; \; \end{matrix} \right\}\]
\end{defn}

\begin{lemma} \label{lemma: representability of scheme of section outside of a divisor}
The functor $\Sect^{Y}_{s}$ is represented by a closed subscheme of $S$.
\end{lemma}
\begin{proof}
After passing to an affine open cover, we can assume that $S = \Spec(R)$. Since everything is of finite presentation, we can use standard approximation results to reduce to the case when $R$ is Noetherian. So $X_{S}$ is Noetherian, and $Y = \underline{\Spec}_{X_{S}}(\mathcal{A})$ for some finitely generated $\mathcal{O}_{X_{S}}$-algebra $\mathcal{A}$. We can choose a finite open cover $\{V_i\}$ of $X_{S}$ and a finite set of sections on each $V_i$ that generate $\mathcal{A}|_{V_i}$ as a $\mathcal{O}_{V_i}$-algebra. By \cite[\href{https://stacks.math.columbia.edu/tag/01PF}{Tag 01PF}]{stacks-project} there exists a coherent subsheaf $\mathcal{F} \subset \mathcal{A}$ that contains all of these sections. Serre vanishing implies that exists some positive integers $N, m$ and a surjective map of sheaves $\mathcal{O}_{X_{S}}^{\oplus m} \twoheadrightarrow \mathcal{F}(N)$. Twisting by $\mathcal{O}_{X}(-N)$, we get a surjection $\mathcal{O}_{X_{S}}(-N)^{\oplus m} \twoheadrightarrow \mathcal{F}$. This induces a surjective map of $\mathcal{O}_{X_{S}}$-algebras $\Sym^{\bullet}\left(\mathcal{O}_{X_{S}}(-N)^{\oplus m}\right) \twoheadrightarrow \mathcal{A}$. If we set $\mathcal{E} \vcentcolon = \mathcal{O}_{X_{S}}(-N)^{\oplus m}$, then we have a closed immersion $ \iota : Y \hookrightarrow \underline{\Spec}_{X_{S}}\left(\Sym^{\bullet} \mathcal{E} \right)$. Let us denote by $s': Q \rightarrow \underline{\Spec}_{X_{S}}\left(\Sym^{\bullet} \mathcal{E} \right)$ the composition of $s:Q \rightarrow Y$ with the closed immersion $\iota$. There is a natural map of functors $\Sect^{Y}_{s} \rightarrow \Sect^{\underline{\Spec}_{X_{S}}\left(\Sym^{\bullet} \mathcal{E} \right)}_{s'}$ given by composition with $\iota$.

The lemma is implied by the following two claims:
\begin{enumerate}[(i)]
    \item The functor $\Sect^{\underline{\Spec}_{X_{S}}\left(\Sym^{\bullet} \mathcal{E} \right)}_{s'}$ is represented by a closed subscheme of $S$.
    \item The map $\Sect^{Y}_{s} \rightarrow \Sect^{\underline{\Spec}_{X_{S}}\left(\Sym^{\bullet} \mathcal{E} \right)}_{s'}$ is a closed immersion.
\end{enumerate}

For claim (i), note that $s'$ is just an element of $\Gamma(Q, \mathcal{E})$. Since $j_*\, j^* \, \mathcal{E} = \underset{m \in \mathbb{Z}}{\colim} \; \mathcal{E}(mD)$, there exists a positive integer $M$ such that $s'$ extends to a section $\sigma \in \Gamma(X_{S}, \mathcal{E}(MD))$. Consider the composition $\overline{\sigma}: \mathcal{O}_{X_S} \xrightarrow{\sigma} \mathcal{E}(MD) \twoheadrightarrow \mathcal{E}(MD) \, / \, \mathcal{E}$. The functor $\Sect^{\underline{\Spec}_{X_{S}}\left(\Sym^{\bullet} \mathcal{E} \right)}_{s'}$ can be seen to be isomorphic to the functor $Z_{\overline{\sigma}}$, as in Lemma \ref{lemma:vanishing section closed}. The coherent sheaf $\mathcal{E}(MD) \, / \, \mathcal{E}$ is $S$-flat by the slicing criterion for flatness \cite[\href{https://stacks.math.columbia.edu/tag/046Y}{Tag 046Y}]{stacks-project}. Hence, Lemma \ref{lemma:vanishing section closed} implies that $Z_{\overline{\sigma}}$ is represented by a closed subscheme of $S$.

We now proceed to prove claim (ii). Let $T$ be a $S$-scheme. For our purposes, we can assume that $T$ is affine and Noetherian. Choose $T \rightarrow \Sect^{\underline{\Spec}_{X_{S}}\left(\Sym^{\bullet} \mathcal{E} \right)}_{s'}$ represented by a section $f$ of the structure morphism $\underline{\Spec}_{X_T}\left(\Sym^{\bullet} \mathcal{E}_T\right) \rightarrow X_T$. We want to show that the fiber product $\Sect^{\underline{\Spec}_{X_{S}}\left(\Sym^{\bullet} \mathcal{E} \right)}_{s'} \times_{\Sect^{Y}_{s}} T$ is represented by a closed subscheme of $T$. For any scheme $W$, we have
\begin{gather*}  \Sect^{\underline{\Spec}_{X_{S}}\left(\Sym^{\bullet} \mathcal{E} \right)}_{s'} \times_{\Sect^{Y}_{s}} T \, (W) = \left\{ \begin{matrix}\text{morphisms $h: W \rightarrow T$ such that}\\ \; \; \text{$ (\text{id}_{X_T} \times h)^* f : X_W \rightarrow \underline{\Spec}_{X_W}\left(\Sym^{\bullet} \mathcal{E}_W \right)$} \; \; \\ \; \; \text{factors through $Y_W$} \; \; \end{matrix} \right\}\end{gather*}
Define $B$ to be the fiber product
\[
\begin{tikzcd}
     B \ar[d] \ar[r, symbol = \hookrightarrow]  &  X_T \ar[d, "f"] \\
     Y_T \ar[r, "\iota \times \text{id}_T"] & \underline{\Spec}_{X_T}\left(\Sym^{\bullet} \mathcal{E}_T\right)
\end{tikzcd}
\]
The morphism $m : B \hookrightarrow X_T$ is a closed immersion. For any scheme $W$, the $W$-points of the fiber product can be alternatively described by
\begin{gather*}
    \Sect^{\underline{\Spec}_{X_{S}}\left(\Sym^{\bullet} \mathcal{E} \right)}_{s'} \times_{\Sect^{Y}_{s}} T \, (W) = \left\{ \begin{matrix}\text{morphisms $h: W \rightarrow T$ such that}\\ \; \; \text{$(\text{id}_{X_T} \times h)^{*} m : B_{W} \hookrightarrow X_W$} \; \text{is an isomorphism} \end{matrix} \right\}
\end{gather*}
Let us denote by $\varphi: \mathcal{J} \hookrightarrow \mathcal{O}_{X_T}$ the the inclusion of the $\mathcal{O}_{X_T}$-sheaf of ideals $\mathcal{J}$ corresponding to the closed subscheme $B \hookrightarrow X_T$. The fiber product $\Sect^{\underline{\Spec}_{X}\left(\Sym^{\bullet} \mathcal{E} \right)}_{s'} \times_{\Sect^{Y}_{s}} T$ is isomorphic to $Z_{\varphi}$. Hence it is represented by a closed subscheme of $T$ by Lemma \ref{lemma:vanishing section closed}.
\end{proof}
 
 \begin{proof}[Proof of Proposition \ref{prop: the affine grassmannian is indproper}]
Let $T \in \Aff_{S}$ and let $T \rightarrow  \Gr_{X_S, D, \mathcal{F}}$ be a morphism represented by a pair $(\mathcal{E}, \theta)$. By Definition \ref{defn: affine grassmannian} the restriction $\theta|_{Q_{T}}$ over the open subset $Q_{T}$ is an isomorphism. This yields an identification $\Red_{G}(\mathcal{F})|_{Q_{T}} \xrightarrow{\sim}\Red_{G}(\mathcal{E})|_{Q_T}$. We have a section $s$ induced by $\sigma$ as follows
\[ s : Q_{T}  \xrightarrow{\sigma|_{Q_T}} \Red_{G}(\mathcal{F})|_{Q_{T}} \xrightarrow{\sim}\Red_{G}(\mathcal{E})|_{Q_T}\]
After unraveling the definitions, it can be seen that the fiber product $\Gr_{X_S, D, \mathcal{F}, \sigma} \times_{\Gr_{X_S, D, \mathcal{F}}} T$ is naturally isomorphic to $\Sect^{\Red_{G}(\mathcal{E})}_s$. By Lemma \ref{lemma: representability of scheme of section outside of a divisor}, this is represented by a closed subscheme of $T$.
 \end{proof}
 
 Let $P \in \mathbb{Q}[n]$ be a rational polynomial. For each $N \geq 0$ we define the subfunctors
 \[\Gr_{X_S, D, \mathcal{F}, \sigma}^{ \leq N} \vcentcolon = \Gr_{X_S, D, \mathcal{F}, \sigma} \cap \Gr_{X_S, D, \mathcal{F}}^{\leq N}\; \; \; \text{and} \; \; \; \Gr_{X_S, D, \mathcal{F}, \sigma}^{ \leq N, P} \vcentcolon = \Gr_{X_S, D, \mathcal{F}, \sigma} \cap \Gr_{X_S, D, \mathcal{F}}^{\leq N, P}\]
 Both of these are closed sub ind-schemes of $\Gr_{X_S, D, \mathcal{F}, \sigma}$. Propositions \ref{prop: ind-representability of step 1 grassmannian} and \ref{prop: the affine grassmannian is indproper} imply that $\Gr_{X_S, D, \mathcal{F}, \sigma} = \underset{N>0}{\colim} \; \Gr_{X_S, D, \mathcal{F}, \sigma}^{\leq N}$. Furthermore each $\Gr_{X_S, D, \mathcal{F}, \sigma}^{ \leq N, P}$ is represented by a scheme that is projective and of finite presentation over $S$.
 
 There is a natural forgetful morphism $\Gr_{X_S, D, \mathcal{F}, \sigma} \rightarrow \Bun_{\rho}(X)$ that takes a point $(\mathcal{E}, \zeta, \theta)$ in $\Gr_{X_S, D, \mathcal{F}, \sigma}$ and forgets the isomorphism $\theta$. We can use this map to pullback the line bundle $L_n$ to $\Gr_{X_S, D, \mathcal{F}, \sigma}$. We will abuse notation and denote this restriction by $L_n$ as well. 

 \begin{coroll} \label{coroll: ampleness of line bundle on affine grassmannian}
 Let $N \in \mathbb{Z}_{\geq 0}$ and $P \in \mathbb{Q}[n]$. There exists some integer $m\gg0$ such that for all $n \geq m$, the restriction $L_n^{\vee}|_{\Gr_{X_S, D, \mathcal{F}, \sigma}^{\leq N, P}}$ is $S$-ample.
 \end{coroll}
 \begin{proof}
 This follows from Proposition \ref{prop: the affine grassmannian is indproper} and \cite[Prop. 3.19]{torsion-freepaper}, which states that $L_n^{\vee}$ is eventually ample on $\Gr_{X_S, D, \mathcal{F}}^{\leq N, P}$.
 \end{proof}
\end{subsection}

\begin{subsection}{Monotonicity of the numerical invariant} \label{subsection: monotonicity properties of the numerical invariant}
In this subsection we prove that the invariant $\nu$ on $\Bun_{\rho}(X)$ is strictly $\Theta$-monotone and strictly $S$-monotone. We recall what these notions mean, see also \cite[\S5]{halpernleistner2021structure} or \cite[\S2.5]{torsion-freepaper} for more details.

We need some setup first. Let $R$ be a complete discrete valuation ring over $k$. Choose a uniformizer $\pi$ of $R$. We define $Y_{\Theta_{R}} \vcentcolon = \mathbb{A}^1_{R}$ equipped with the $\mathbb{G}_m$-action that gives $t$ weight $-1$. We have $\Theta_{R} = \left[ \, \mathbb{A}^1_{R} / \, \mathbb{G}_m \, \right]$. Note that $\mathbb{A}^1_{R}$ contains a unique $\mathbb{G}_m$-invariant closed point cut out by the ideal $(t, \pi)$. We will denote this point by $(0,0)$.

Define $Y_{\overline{ST}_{R}} \vcentcolon = \Spec\left( \, R[t,s]/(st-\pi) \,\right)$. We equip $Y_{\overline{ST}_{R}}$ with the $\mathbb{G}_m$-action that gives $s$ weight $1$ and $t$ weight $-1$. The isomorphism class of the $\mathbb{G}_m$-scheme $Y_{\overline{ST}_{R}}$ is independent of the choice of uniformizer $\pi$. Define $\overline{ST}_{R} \vcentcolon = [\,Y_{\overline{ST}_{R}} / \, \mathbb{G}_m \, ]$. We will denote by $(0,0)$ the unique $\mathbb{G}_m$-fixed closed point in $Y_{\overline{ST}_{R}}$. We have that $(0,0)$ is cut out by the ideal $(s,t)$.

For any field $\kappa$ and integer $a \geq 1$, we denote by $\mathbb{P}^1_{\kappa}[a]$ the $\mathbb{G}_m$-scheme $\mathbb{P}^{1}_{\kappa}$ equipped with the $\mathbb{G}_m$-action determined by the equation $t \cdot [x:y] = [t^{-a}x : y]$. We set $0 = [0: 1]$ and $\infty = [1:0]$.
\begin{defn}[Monotonicity] \label{defn: strictly monotone} A polynomial numerical invariant $\nu$ on a stack $\mathcal{M}$ is strictly $\Theta$-monotone (resp. strictly $S$-monotone) if the following condition holds. 

Let $R$ be any complete discrete valuation ring and set $\mathfrak{X}$ to be $\Theta_{R}$ (resp. $\overline{ST}_{R}$). Choose a map $\varphi: \mathfrak{X} \setminus (0,0) \rightarrow \mathcal{M}$. Then, after maybe replacing $R$ with a finite DVR extension, there exists a reduced and irreducible $\mathbb{G}_m$-equivariant scheme $\Sigma$ with maps $f: \Sigma \rightarrow Y_{\mathfrak{X}}$ and $\widetilde{\varphi}: \left[\Sigma/ \, \mathbb{G}_m\right] \rightarrow  \mathcal{M}$ such that
\begin{enumerate}[({M}1)]
    \item The map $f$ is proper, $\mathbb{G}_m$-equivariant and its restriction over $Y_{\mathfrak{X}} \setminus(0,0)$ induces an isomorphism $f : \, \Sigma_{Y_{\mathfrak{X}} \setminus (0,0)} \xrightarrow{\sim} Y_{\mathfrak{X}} \setminus (0,0)$.
    \item The following diagram commutes
\[
\begin{tikzcd}
  \left[\left(\Sigma_{Y_{\mathfrak{X}} \setminus (0,0)}\right)/ \, \mathbb{G}_m \right] \ar[rd, "\widetilde{\varphi}"] \ar[d, "f"'] & \\   \mathfrak{X} \setminus (0,0) \ar[r, "\varphi"'] &  \mathcal{M}
\end{tikzcd}
\]
    \item Let $\kappa$ be a finite extension of the residue field of $R$. For any $a \geq 1$ and any finite $\mathbb{G}_m$-equivariant morphism $\mathbb{P}^1_{\kappa}[a] \to \Sigma_{(0,0)}$, we have $\nu\left( \;\widetilde{\varphi}|_{\left[\infty / \mathbb{G}_m\right]} \;\right) \geq  \nu\left(\; \widetilde{\varphi}|_{\left[0 / \mathbb{G}_m\right]} \;\right)$.
\end{enumerate}
\end{defn}

\begin{prop} \label{prop: invariant is monotone}
The invariant $\nu$ on the stack $\Bun_{\rho}(X)$ is both strictly $\Theta$-monotone and strictly $S$-monotone.
\end{prop}
\begin{proof}
The argument is a variation of the one in \cite[\S4.3]{torsion-freepaper} using our version of the affine Grassmannian. We let $Y = Y_{\Theta_{R}}$ or $Y_{\overline{ST}_{R}}$. Suppose that we are given a morphism $f: \left[\left( \, Y \setminus (0,0) \, \right)/ \,\mathbb{G}_m\right] \rightarrow \Bun_{\rho}(X)$. This consists of the data of a $\rho$-sheaf $(\mathcal{E}, \zeta)$ on $X_{Y \setminus (0,0)}$ that is $\mathbb{G}_m$-equivariant with respect to the action of $\mathbb{G}_m$ on $Y \setminus (0,0)$. By \cite[Lem. 4.7(i)]{torsion-freepaper}, there is a $\mathbb{G}_m$-equivariant family of torsion-free sheaves $\mathcal{F}$ on $X_Y$, a $\mathbb{G}_m$-stable $Y$-flat effective Cartier divisor $D \hookrightarrow X_Y$ and a $\mathbb{G}_m$-equivariant monomorphism $\theta: \mathcal{F}|_{X_{Y \setminus (0,0)}} \rightarrow \mathcal{E}$ such that $\theta$ is an isomorphism away from $D_{Y\setminus(0,0)}$. Set $Q \vcentcolon = X_{Y} \setminus D$. The isomorphism $\theta|_{Q_{Y \setminus(0,0)}}$ induces a $\mathbb{G}_m$-equivariant identification $\Red_{G}(\mathcal{E})|_{ Q_{Y \setminus(0,0)}} \xrightarrow{\sim} \Red_{G}(\mathcal{F})|_{Q_{Y \setminus(0,0)}}$. We have a $\mathbb{G}_m$-equivariant section induced by $\zeta$:
\[ \sigma_{Q_{Y \setminus(0,0)}} : Q_{Y \setminus(0,0)}  \xrightarrow{\zeta|_{Q_{Y \setminus(0,0)}}} \Red_{G}(\mathcal{E})|_{ Q_{Y \setminus(0,0)}} \xrightarrow{\sim} \Red_{G}(\mathcal{F})|_{Q_{Y \setminus(0,0)}}\]
The local ring $\mathcal{O}_{Y, (0,0)}$ at $(0,0)$ has depth $2$. For any point $u$ in the fiber $Q_{(0,0)}$, the local ring $\mathcal{O}_{Q, u}$ is flat over $\mathcal{O}_{Y, (0,0)}$. It follows that $\mathcal{O}_{Q, u}$ has depth at least $2$ for all $u \in Q_{(0,0)}$. The same type of argument as in Lemma \ref{lemma: lemma on very big open subsets} (c) shows that the $\mathbb{G}_m$-equivariant section $\sigma_{Q_{Y \setminus(0,0)}}$ extends uniquely to a $\mathbb{G}_m$-equivariant section $\sigma: Q \rightarrow \Red_{G}(\mathcal{F})$. We now have all of the necessary data to form the affine grassmannian $\Gr_{X_Y, D, \mathcal{F}, \sigma}$. Since everything is $\mathbb{G}_m$-equivariant, the ind-scheme $\Gr_{X_Y, D, \mathcal{F}, \sigma}$ acquires a natural $\mathbb{G}_m$-action. 

The data of the $\rho$-sheaf $(\mathcal{E}, \zeta)$ and the monomorphism $\theta$ is by definition a $\mathbb{G}_m$-equivariant section $s: Y \setminus (0,0) \rightarrow \Gr_{X_Y, D, \mathcal{F}, \sigma}$ of the structure morphism $\Gr_{X_Y, D, \mathcal{F}, \sigma} \rightarrow Y$. By construction, the composition $Y \setminus (0,0)  \xrightarrow{s} \Gr_{X_Y, D, \mathcal{F}, \sigma} \xrightarrow{\Forget } \Bun_{\rho}(X)$ is the original $\mathbb{G}_m$-equivariant morphism $f: Y \setminus (0,0) \rightarrow \Bun_{\rho}(X)$.

It follows from the definition of $\Gr_{X_Y, D, \mathcal{F}}^{\leq N, P}$ that the $\mathbb{G}_m$-action on $\Gr_{X_Y, D, \mathcal{F}, \sigma}$ preserves each stratum $\Gr_{X_Y, D, \mathcal{F}, \sigma}^{\leq N,P} $. Since $Y$ is affine and connected, the map $s: Y \setminus (0,0) \rightarrow \Gr_{X_Y, D, \mathcal{F}, \sigma}$ factors through $\Gr_{X_Y, D, \mathcal{F}, \sigma}^{\leq N, P}$ for some $N\gg0$ and $P \in \mathbb{Q}[n]$. Therefore, we have the following commutative diagram with $\mathbb{G}_m$-equivariant morphisms.
\item \[
\begin{tikzcd}\
  &   \Gr_{X_Y, D, \mathcal{F}, \sigma}^{\leq N,P} \ar[d] \ar[r, "\Forget"] &   \Bun_{\rho}(X) \\   Y \setminus (0,0)  \ar[ur, "s"] \ar[r, symbol= \xhookrightarrow{\;\;\;\;\;\; \; \; \; \;}] &  Y &
\end{tikzcd}
\]

Let $\Sigma \subset \Gr_{X_Y, D, \mathcal{F}, \sigma}^{\leq N, P}$ denote the scheme theoretic image of $Y \setminus (0,0)$ in $\Gr_{X_Y, D, \mathcal{F}, \sigma}^{\leq N, P}$. Note that $\Sigma$ is a reduced $\mathbb{G}_m$-scheme with a natural structure morphism to $Y$. The map $\Sigma \rightarrow Y$ is projective, because  $\Gr_{X_Y, D, \mathcal{F}, \sigma}^{\leq N, P}$ is projective over $Y$ and $\Sigma$ is a closed subscheme of $\Gr_{X_Y, D, \mathcal{F}, \sigma}^{\leq N, P}$. By construction, the morphism $\Sigma \rightarrow Y$ is $\mathbb{G}_m$-equivariant and restricts to an isomorphism over the open $Y \setminus (0,0)$.

The composition $\Sigma \rightarrow\Gr_{X_Y, D, \mathcal{F}, \sigma}^{\leq N, P} \xrightarrow{\Forget} \Bun_{\rho}(X)$ restricts to the original $\mathbb{G}_m$-equivariant morphism $f : Y \setminus (0,0)\rightarrow \Bun_{\rho}(X)$ over the open subset $Y\setminus (0,0) \, \subset \, \Sigma$. We therefore obtain a morphism $\widetilde{f}: \left[ \Sigma / \, \mathbb{G}_m \right] \rightarrow \Bun_{\rho}(X)$ satisfying condition (M2) in Definition \ref{defn: strictly monotone}.

We are left to check condition (M3) in Definition \ref{defn: strictly monotone}. By Corollary \ref{coroll: ampleness of line bundle on affine grassmannian}, there exists some $M\gg0$ such that for all $n \geq M$ the restriction $L_n^{\vee}|_{\Gr_{X_Y, D, \mathcal{F}, \sigma}^{\leq N, P}}$ is $Y$-ample, and therefore we can conclude (M3) exactly as in the end of the  proof of \cite[Thm. 4.8]{torsion-freepaper}.
\end{proof}
\end{subsection}

\begin{subsection}{HN boundedness} \label{subsection: hn boundedness}
We start by recalling the definition of HN boundedness (see also \cite[\S 5]{halpernleistner2021structure} or  \cite[\S 2.5]{torsion-freepaper}).
\begin{defn}[HN Boundedness] \label{defn: HN boundedness}
Let $\mathcal{M}$ be an algebraic stack locally of finite type over $k$. Let $\nu$ be a polynomial numerical invariant on $\mathcal{M}$. We say that $\nu$ satisfies the HN boundedness condition if the following is true.

Choose an algebraically closed field extension $K \supset k$. Let $S$ be a scheme of finite type over $K$ and let $g: S \rightarrow \mathcal{M}$ be a morphism. Then, there exists a quasicompact open substack $\mathcal{U}_{S} \subset \mathcal{M}$ such that the following holds:
For all closed points $s \in S$ and all nondegenerate filtrations $f: \Theta \rightarrow \mathcal{M}$ of the point $g(s)$ with $\nu(f)>0$, there exists another filtration $f'$ of $g(s)$ satisfying $\nu(f') \geq \nu(f)$ and $f'(0) \in \mathcal{U}_{S}$.
\end{defn}

 In the case of $\Bun_{\rho}(X)$ we can reduce HN boundedness to a more concrete statement.
\begin{lemma} \label{lemma: equivalent HN boundedness condition}
In order to show that $\nu$ satisfies HN boundedness on the stack $\Bun_{\rho}(X)$, it is sufficient to prove the following statement.

(*) Let $K \supset k$ be an algebraically closed field extension of $k$. Let $S$ be a scheme of finite type over $K$. Let $g: S \rightarrow \Bun_{\rho}(X)$ be a morphism represented by a $\rho$-sheaf $(\mathcal{F}, \sigma)$ on $X_{S}$. Then, there exists a real number $C_{S}$ such that the following holds:
For all closed points $s \in S$ and all nondegenerate weighted parabolic reductions $f= (\lambda, \mathcal{G}_{\lambda})$ of the point $(\mathcal{F}|_{X_{s}}, \sigma|_{X_{s}})$ with $\nu(f)>0$, there exists another weighted parabolic reduction $f' = (\lambda', \mathcal{G}_{\lambda'})$ of $(\mathcal{F}|_{X_{s}}, \sigma|_{X_{s}})$  with underlying associated filtration $(\mathcal{F}'_m)_{m \in \mathbb{Z}}$ such that the following are satisfied.
\begin{enumerate}[(1)]
\item $\nu(f') \geq \nu(f)$.
\item We have $\widehat{\mu}_{d-1} (\mathcal{F}'_{m} / \mathcal{F}'_{m+1}) \geq C_{S}$ for all $m \in \mathbb{Z}$.
\end{enumerate}
\end{lemma}
\begin{proof}
We can replace $k$ by $K$ in order to assume that the ground field is algebraically closed. Then, all reductive groups are automatically split. By Proposition \ref{prop: explicit description of filtrations singular G bundles}, all filtrations of $(\mathcal{F}|_{X_{s}}, \sigma|_{X_{s}})$ arise from weighted parabolic reductions of $(\mathcal{F}|_{X_{s}}, \sigma|_{X_{s}})$. Let $\mathfrak{S}$ denote the set of all pairs $(s, f')$, where $s \in S$ is a closed point and $f'$ is a weighted parabolic reduction of $(\mathcal{F}|_{X_{s}}, \sigma|_{X_{s}})$ satisfying $(2)$ above. It suffices to show that the set of associated graded points $\mathfrak{S}_2 \vcentcolon = \left\{ \; f'(0) \; \text{  for $(s, f') \in \mathfrak{S}$} \; \; \right\}$
is contained in a quasicompact open substack of $\Bun_{\rho}(X)$. Recall that we have a forgetful morphism $\Forget: \Bun_{\rho}(X) \rightarrow \Coh_{r}^{tf}(X)$ into the stack of rank $r$ torsion-free sheaves on $X$. By Proposition \ref{prop: forgetful morphism is affine}, the morphism $\Forget$ is of finite type. Hence it is sufficient to prove that the set $\mathfrak{S}_3 \vcentcolon = \left\{ \; \; \Forget(f'(0)) \; \text{  for $(s, f') \in \mathfrak{S}$} \; \; \right\}$ is contained in a quasicompact open substack of $\Coh_{r}^{tf}(X)$. Let $\mathfrak{S}'$ denote the set of pairs $(s, (\mathcal{F}_{m})_{m \in \mathbb{Z}})$, where $s \in S$ is a closed point and $(\mathcal{F}_{m})_{m \in \mathbb{Z}}$ is a decreasing $\mathbb{Z}$-filtration of the sheaf $\mathcal{F}|_{X_{s}}$ such that $\widehat{\mu}_{d-1}(\mathcal{F}_{m}/\mathcal{F}_{m+1}) \geq C_{S}$ for all $m \in \mathbb{Z}$. It is sufficient to prove that the set of sheaves
\[\mathfrak{S}'_1 = \left\{ \; \; \; \bigoplus_{m \in \mathbb{Z}} \mathcal{F}_m / \mathcal{F}_{m+1} \; \text{ \; \; \; for $(s, (\mathcal{F}_{m})_{m \in \mathbb{Z}}) \in \mathfrak{S}'$    } \; \; \;\right\}\]
is contained in a quasicompact open substack of $\Coh_{r}^{tf}(X)$. This follows from \cite[Lemma 5.4]{torsion-freepaper}.
\end{proof}

To prove HN boundedness for $\Bun_{\rho}(X)$ in Proposition \ref{prop: HN boundedness}, we need  to understand how the numerical invariant changes for parabolic subgroups associated to different filtrations. We will use Lemma \ref{lemma: equivalent HN boundedness condition}. After replacing $k$ with the extension $K$, we can assume that the ground field is algebraically closed (in particular all tori are automatically split). Fix a $\rho$-sheaf $(\mathcal{F}, \sigma)$ on $X$. Let $j: U \hookrightarrow X$ be a big open subset such that $\mathcal{F}|_{U}$ is locally-free. We will develop some setup and prove some preliminary lemmas before the proof of HN boundedness.

Let $Z$ denote the maximal central torus of $G$. Let $X_{*}(Z)$ denote the group of cocharacters of $Z$. We denote by $X^*(Z)$ the dual group of characters of $Z$. Set $X_{*}(Z)_{\mathbb{R}} \vcentcolon = X_{*}(Z) \otimes \mathbb{R}$ and similarly $X^{*}(Z)_{\mathbb{R}} \vcentcolon = X^{*}(Z) \otimes \mathbb{R}$. We fix once and for all a basis $z_1, \ldots , z_h$ of $X_{*}(Z)$.

Let $f=(\lambda, \mathcal{G}_{\lambda})$ be a weighted parabolic filtration of $(\mathcal{F}, \sigma)$, and let $P_{\lambda} \subset G$ be the parabolic subgroup associated to $\lambda$. Choose a maximal torus $T$ of $P_{\lambda}$ containing the image of the cocharacter $\lambda$. Set $T':=T\cap [G,G]$, the intersection of $T$ with the derived subgroup $[G,G]$ of $G$. Denote the groups of cocharacters of $T$ and $T'$ by $X_{*}(T)$ and  $X_{*}(T')$ respectively. Observe that there are natural inclusions $X_*(Z)_{\mathbb{R}} \hookrightarrow X_{*}(T)_{\mathbb{R}}$ and $X_{*}(T')_{\mathbb{R}} \hookrightarrow X_{*}(T)_{\mathbb{R}}$. These induce a direct sum decomposition $X_{*}(T)_{\mathbb{R}}  = X_{*}(Z)_{\mathbb{R}} \oplus X_*(T')_{\mathbb{R}}$.

Choose a Borel subgroup $B \supset T$ contained in $P_{\lambda}$. This yields a set of simple roots $\Delta = \{ \alpha_1, \alpha_2, \ldots, \alpha_l \} \subset X_{*}(T')_{\mathbb{R}}^{\vee}$, where $l$ is the rank of $G$, which form a basis for $X^*(T')_{\mathbb{R}}$. Denote by $\omega_{i}^{\vee} \in X_{*}(T')_{\mathbb{R}}$ the fundamental coweight dual to the simple root $\alpha_i$ and define the ordered basis $\mathcal{B}$ of the vector space $X_{*}(T)_{\mathbb{R}}$ as $\mathcal{B}= (z_1, z_2, \ldots , z_h, \omega_1^{\vee}, \omega_2^{\vee}, \ldots , \omega^{\vee}_l)$. 
This induces a dual basis $\mathcal{B}^{\vee}=(z_1^{\vee}, z_2^{\vee}, \ldots , z_h^{\vee}, \alpha_1, \alpha_2, \ldots , \alpha_l)$ on the dual vector space $X^{*}(T)_{\mathbb{R}}$.

The torus $T\subset G$ acts on the representation $V$. Choose a decomposition of $V$ as a direct sum of one-dimensional weight spaces $V = \bigoplus_{i=1}^{r} V_{\chi_i}$ where $(\chi_i)_{i=1}^r$ is an ordered $r$-tuple of characters in $X^{*}(T)_{\mathbb{R}}$ with some possible repetitions.
\begin{defn}
 Using coordinates in the dual basis $\mathcal{B}^{\vee}$, we denote by $M$ the $(h+l) \times r$-matrix with coefficients in $\mathbb{Q}$ corresponding to the tuple $(\chi_i)_{i=1}^r$.
\end{defn}
The matrix $M$ does not depend on the choice of torus $T$ and Borel subgroup $B$ (once we fix the basis of $X_{*}(Z)$), because all torus-Borel pairs $T \subset B$ are conjugate in $G$.

The parabolic $P_{\lambda}$ determines a subset $I_{P_{\lambda}} \subset \Delta$ consisting of those simple roots $\alpha_i \in \Delta$ such that the root group $U_{-\alpha_i}$ corresponding to $-\alpha_i$ is not contained in $P_{\lambda}$ (e.g. we have $I_{B} = \Delta$ and $I_{G} = \emptyset$). Let $J_{P_{\lambda}} \subset \{1, 2, \ldots, l\}$ be the set of indexes such that $I_{P_{\lambda}} = \{ \alpha_j \; \mid \; j \in J_{P_{\lambda}}\}$. 

\begin{defn}
We define $\mathcal{C}_{\lambda}$ to be the closed $\mathbb{R}$-span of the cone of $P_{\lambda}$-dominant coweights in $X_{*}(T)_{\mathbb{R}}$, with the origin removed:
\[ \mathcal{C}_{\lambda} = \left(\sum_{i=1}^h \mathbb{R} z_i + \sum_{ j \in J_{P_{\lambda}}} \mathbb{R}^{\geq 0}\omega^{\vee}_j\right) \setminus \{0\}\]
\end{defn}
For any cocharacter $\delta \in \mathcal{C}_{\lambda} \cap X_{*}(T)$, let $P_{\delta}$ be its associated parabolic subgroup. Note that $P_{\delta} \supset P_{\lambda}$. To every such $\delta$ we can associate a weighted parabolic reduction $(\delta, \mathcal{G}_{\delta})$, where $\mathcal{G}_{\delta}$ is the $P_{\delta}$-bundle obtained from the $P_{\lambda}$-bundle $\mathcal{G}_{\lambda}$ via extension of structure group. The evaluation of the numerical invariant $\nu$ on these weighted parabolic reductions $(\delta, \mathcal{G}_{\delta})$ defines a function 
\[\nu: \mathcal{C}_{\lambda} \cap X_{*}(T) \rightarrow \mathbb{R}[n],\;\;\; \delta\mapsto \nu(\delta, \mathcal{G}_{\delta}).\] 
By the formula in Proposition \ref{prop: value numerical invariant}, the polynomial $\nu(\delta, \mathcal{G}_{\delta})$ has degree less than or equal to $d-1$. 
\begin{defn}
We define $\nu_{d-1}: \mathcal{C}_{\lambda} \cap X_{*}(T) \rightarrow \mathbb{R}$ to be the function that associates to each $\delta$ the degree $(d-1)$-coefficient of the polynomial $\nu(\delta, \mathcal{G}_{\delta})$.
\end{defn}

We denote by $L_{\lambda}$ the unique Levi subgroup of the parabolic $P_{\lambda}$ containing the maximal torus $T$, as described in Subsection \ref{subsection: filtrations}. Let $Z_{\lambda}$ be the maximal central torus of $L_{\lambda}$. For each $i$, let $\chi_i^{\lambda}$ be the restriction of the character $\chi_i$ to $X_{*}(Z_{\lambda})$. Set $V_{\chi^{\lambda}_i} = \bigoplus_{\{j \, \mid \, \chi^{\lambda}_j = \chi^{\lambda}_i\}} V_{\chi^{\lambda}_j}$. We have a $X^{*}(Z_{\lambda})$-grading on the vector space $V$ with weight decomposition:
\[ V = \bigoplus_{\{[\chi^{\lambda}_i]\}} V_{\chi^{\lambda}_i} \]
Here the index set $\{[\chi^{\lambda}_i]\}$ runs over the set of all restrictions $\chi^{\lambda}_i \in X^{*}(Z_{\lambda})$ without repetitions, and $Z_{\lambda}$ acts on $V_{\chi^{\lambda}_i}$ with weight $\chi^{\lambda}_i$. Note that each $V_{\chi_i^{\lambda}}$ is a representation of $L_{\lambda}$, because $Z_{\lambda}$ is central in $L_{\lambda}$. Let $\mathcal{G}_{L_{\lambda}}$ denote the $L_{\lambda}$-bundle on $U$ obtained from $\mathcal{G}_{\lambda}$ via the natural quotient morphism $P_{\lambda} \twoheadrightarrow L_{\lambda}$. Let $\mathcal{G}_{L_{\lambda}} \times^{L_{\lambda}} V_{\chi_i^{\lambda}}$ denote the associated vector bundle on $U$.

\begin{defn}
 We define $\mathcal{E}_{\chi_i^{\lambda}}$ to be the reflexive extension $j_*\left(\mathcal{G}_{L_{\lambda}} \times^{L_{\lambda}} V_{\chi_i^{\lambda}}\right)$ on $X$.
\end{defn} 

We will denote by $H$ the hyperplane class in $X$ corresponding to the ample line bundle $\mathcal{O}(1)$, and denote by $c_1(\mathcal{E}_{\chi_i^{\lambda}})$ the first Chern class of $\mathcal{E}_{\chi_i^{\lambda}}$.
\begin{lemma} \label{lemma: first formula for nu_(d-1)}
Set $c_i \vcentcolon =  \frac{c_1(\mathcal{E}_{\chi^{\lambda}_i}) \cdot H^{d-1}}{\rk(\mathcal{E}_{\chi^{\lambda}_i})}$ and $c\vcentcolon = \frac{c_1(\mathcal{F}) \cdot H^{d-1}}{r}$. Then, we have
\begin{equation} \label{eqn_4}
\nu_{d-1}(\delta) = \frac{1}{\sqrt{A_{d}}(d-1)!} \, \cdot \, \frac{\sum_{i=1}^r \langle \delta , \chi^{\lambda}_i \rangle \cdot (c_i - c) }{\sqrt{\sum_{i=1}^r \left(\langle \delta, \chi^{\lambda}_i \rangle\right)^2}} 
\end{equation}
In particular, the function $\nu_{d-1}$ can be naturally interpreted as a scale-invariant continuous function on the whole cone $\mathcal{C}_{\lambda}$ of dominant coweights with real coefficients. 
\end{lemma}
\begin{proof}
For any $\delta \in \mathcal{C}_{\lambda} \cap X_{*}(T)$, denote by $L_{\delta}$ the unique Levi subgroup of the parabolic $P_{\delta}$ containing the maximal torus $T$. Note that $L_{\lambda}\subset L_{\delta}$. Denote by $Z_{\delta}$ the maximal central torus of $L_{\delta}$. Let $\mathcal{G}_{L_{\delta}}$ be the $L_{\delta}$--bundle obtained from $\mathcal{G}_{\delta}$ via the quotient $P_{\delta} \twoheadrightarrow L_{\delta}$. Observe that $\mathcal{G}_{L_{\delta}}$ is obtained from $\mathcal{G}_{L_{\lambda}}$ by extending the structure group via $L_{\lambda} \hookrightarrow L_{\delta}$. Let $\mathcal{E}_{U}^{\delta}$ be the vector bundle on $U$ obtained by extending the structure group via the inclusion $L_{\delta} \hookrightarrow G \xrightarrow{\rho} \GL(V)$. For all $\delta$ we have a canonical identification $\mathcal{E}_{U}^{\delta} \cong \mathcal{E}^{\lambda}_{U}$. Hence we will denote all such $\mathcal{E}^{\delta}_{U}$ by the same notation $\mathcal{E}_{U}$. Let $\mathcal{E}$ be the reflexive sheaf on $X$ obtained by pushing forward $\mathcal{E}_{U}$ under the open inclusion $U \hookrightarrow X$. Let $(\mathcal{F}^{\delta}_m)_{m \in \mathbb{Z}}$ be the underlying filtration of $\mathcal{F}$ associated to the weighted parabolic reduction $(\delta, \mathcal{G}_{\delta})$. The description of the associated graded at the end of Subsection \ref{subsection: filtrations} shows that there is an identification $\left(\bigoplus_{m \in \mathbb{Z}} \mathcal{F}^{\delta}_m / \mathcal{F}^{\delta}_{m+1}\right)|_{U} \xrightarrow{\sim} \mathcal{E}_{U}$ of associated graded sheaves on $U$. This induces a natural identification $\bigoplus_{m \in \mathbb{Z}} (\mathcal{F}^{\delta}_m / \mathcal{F}^{\delta}_{m+1})^{\vee \vee} \xrightarrow{\sim} \mathcal{E}$ of reflexive sheaves on $X$. We will describe a $\mathbb{Z}$-grading on $\mathcal{E}$ determined by the action of the cocharacter $\delta$. The inclusion $Z_{\lambda} \hookrightarrow T$ induces
an inclusion $X_{\ast}(Z_{\lambda})\hookrightarrow X_{\ast}(T)$. Recall that $\chi_{i}^{\lambda}$ denotes the restriction of each character $\chi_{i}$ to $X_{\ast}(Z_{\lambda})$. We have a $X^{*}(Z_{\lambda})$-grading on the vector space $V$ with weight decomposition $V = \bigoplus_{\{[\chi^{\lambda}_i]\}} V_{\chi^{\lambda}_i}$.

 For any representative $\chi_i^{\lambda}$ of $[\chi^{\lambda}_i]$ we have $V_{\chi^{\lambda}_i} = \bigoplus_{\{j \, \mid \, \chi^{\lambda}_j = \chi^{\lambda}_i\}} V_{\chi^{\lambda}_j}$. Hence $\dim( V_{\chi^{\lambda}_i}) = \#\{j \; \mid \;  \chi^{\lambda}_j = \chi^{\lambda}_i \, \}$. Since $Z_{\lambda}$ is central in $L_{\lambda}$, the action of $L_{\lambda}$ on $V$ preserves each $Z_{\lambda}$-weight space $V_{\chi^{\lambda}_i}$, hence each $V_{\chi_i^{\lambda}}$ is a representation of $L_{\lambda}$. This way we get the grading
\[  \mathcal{E}_{U} = \bigoplus_{\{[\chi^{\lambda}_i]\}} \mathcal{E}_{U, \chi^{\lambda}_i}\]
where $\mathcal{E}_{U, \chi^{\lambda}_i}$ is the vector bundle $\mathcal{G}_{L_{\lambda}} \times^{L_{\lambda}} V_{\chi_i^{\lambda}}$ associated to the representation $V_{\chi_i^{\lambda}}$ and $Z_{\lambda}$ acts with weight $\chi^{\lambda}_i$ on $\mathcal{E}_{U, \chi^{\lambda}_i}$. Note that $\rk(\mathcal{E}_{U, \chi^{\lambda}_i}) = \dim(V_{\chi_i^{\lambda}})$. This grading extends uniquely to a grading of the reflexive extension
\[ \bigoplus_{m \in \mathbb{Z}} (\mathcal{F}^{\delta}_m / \mathcal{F}^{\delta}_{m+1})^{\vee \vee} \cong \mathcal{E} = \bigoplus_{\{[\chi^{\lambda}_i]\}} \mathcal{E}_{ \chi^{\lambda}_i}\]

The cocharacter $\delta$ acts on $\mathcal{E}_{U}$ via its natural inclusion in $Z_{\delta}$. Using $Z_{\delta}\hookrightarrow Z_{\lambda}$, we can view $\delta$ as a cocharacter of $Z_{\lambda}$. It follows from our discussion that $\delta$ acts with weight $\langle \delta, \chi^{\lambda}_i\rangle$ on $\mathcal{E}_{U,\chi^{\lambda}_i}$. This shows that 
\[(\mathcal{F}_m^{\delta} / \mathcal{F}_{m+1}^{\delta})|_{U} = \bigoplus_{\{ [\chi^{\lambda}_i] \; \mid \; \langle \delta , \chi^{\lambda}_i \rangle = m \}} \mathcal{E}_{U,\chi^{\lambda}_i}\]
and therefore
\begin{equation} \label{eqn: identification of hull of graded piece}
(\mathcal{F}_m^{\delta} / \mathcal{F}_{m+1}^{\delta})^{\vee \vee} = \bigoplus_{\{ [\chi^{\lambda}_i] \; \mid \; \langle \delta , \chi^{\lambda}_i \rangle = m \}} \mathcal{E}_{\chi^{\lambda}_i}
\end{equation}
Observe that we have equality of slopes $\widehat{\mu}_{d-1}\left(\mathcal{F}_m^{\delta} / \mathcal{F}_{m+1}^{\delta}\right) = \widehat{\mu}_{d-1}\left((\mathcal{F}_m^{\delta} / \mathcal{F}_{m+1}^{\delta})^{\vee \vee}\right)$. Taking the degree $d-1$ coefficient in the formula (\ref{eqn_1}) in Proposition \ref{prop: value numerical invariant}, we get
\begin{equation} \label{eqn_2}
\nu_{d-1}(\delta) = \frac{\sqrt{A_d}}{(d-1)!} \, \cdot \, \frac{\sum\limits_{m \in \mathbb{Z}} m \cdot  \left(\widehat{\mu}_{d-1}(\mathcal{F}^{\delta}_m/ \mathcal{F}^{\delta}_{m+1}) - \widehat{\mu}_{d-1}(\mathcal{F})\right) \cdot \rk(\mathcal{F}^{\delta}_m/\mathcal{F}^{\delta}_{m+1})}{\sqrt{\sum\limits_{m \in \mathbb{Z}} m^2 \cdot \rk(\mathcal{F}^{\delta}_m / \mathcal{F}^{\delta}_{m+1})}}
\end{equation}
Using the description of the quotients $\mathcal{F}^{\delta}_m/\mathcal{F}^{\delta}_{m+1}$ in terms of the $\mathcal{E}_{\chi^{\lambda}_i}$, we can rewrite $(\ref{eqn_2})$ as:
\begin{gather*}\nu_{d-1}(\delta) =  \frac{\sqrt{A_d}}{(d-1)!} \, \cdot \, \frac{\sum\limits_{m \in \mathbb{Z}} \sum\limits_{\{[\chi_i^{\lambda}] \; \mid \; \langle \delta, \chi_i^{\lambda} \rangle = m\}} m \cdot  \left(\widehat{\mu}_{d-1}(\mathcal{E}_{\chi^{\lambda}_i}) - \widehat{\mu}_{d-1}(\mathcal{F})\right) \cdot \#\{j \; \mid \;  \chi^{\lambda}_j = \chi^{\lambda}_i \, \}}{\sqrt{\sum_{m \in \mathbb{Z}} m^2 \cdot \#\{j \; \mid \;  \langle \delta, \chi^{\lambda}_i \rangle = m \, \}}}\end{gather*}
\begin{equation}\label{eqn_3}
=\frac{\sqrt{A_d}}{(d-1)!} \, \cdot \, \frac{\sum\limits_{i=1}^r \langle \delta , \chi^{\lambda}_i \rangle \cdot \left(\widehat{\mu}_{d-1}(\mathcal{E}_{\chi^{\lambda}_i}) - \widehat{\mu}_{d-1}(\mathcal{F})\right) }{\sqrt{\sum_{i=1}^r \left(\langle \delta, \chi^{\lambda}_i \rangle\right)^2}} 
\end{equation}
Observe that in the sum above we are indexing over all $\chi_i^{\lambda}$ with repetitions. This is done in order to account for the ranks.

We can further rewrite this by using the Hirzebruch-Riemann-Roch theorem \cite[Ch. 15]{fulton-intersection}. Let $H$ denote the hyperplane class on $X$ corresponding to the line bundle $\mathcal{O}(1)$. Let $t_1$ be the degree $1$ term of the Todd class of $X$. For any torsion-free sheaf $\mathcal{A}$ on $X$ with first Chern class $c_1(\mathcal{A})$, Hirzebruch-Riemann-Roch implies that
\[ \widehat{\mu}_{d-1}(\mathcal{A}) = \frac{1}{A_{d} \, \rk(\mathcal{A})}c_1(\mathcal{A}) \cdot H^{d-1} + \frac{1}{A_d}t_1 \cdot H^{d-1}\]
Using this formula for $\mathcal{E}_{\chi_i^{\lambda}}$ and $\mathcal{F}$ yields
\[\widehat{\mu}_{d-1}(\mathcal{E}_{\chi^{\lambda}_i}) - \widehat{\mu}_{d-1}(\mathcal{F}|_{X_{s}}) = \frac{1}{A_{d}} \cdot \left( \frac{c_1(\mathcal{E}_{\chi^{\lambda}_i}) \cdot H^{d-1}}{\rk(\mathcal{E}_{\chi^{\lambda}_i})} - \frac{c_1(\mathcal{F}) \cdot H^{d-1}}{r}  \right)\]
Using $c_i =  \frac{c_1(\mathcal{E}_{\chi^{\lambda}_i}) \cdot H^{d-1}}{\rk(\mathcal{E}_{\chi^{\lambda}_i})}$ and $c = \frac{c_1(\mathcal{F}) \cdot H^{d-1}}{r}$, we can rewrite (\ref{eqn_3}) as:
\begin{equation*}
\nu_{d-1}(\delta) = \frac{1}{\sqrt{A_{d}}(d-1)!} \, \cdot \, \frac{\sum_{i=1}^r \langle \delta , \chi^{\lambda}_i \rangle \cdot (c_i - c) }{\sqrt{\sum_{i=1}^r \left(\langle \delta, \chi^{\lambda}_i \rangle\right)^2}} 
\end{equation*}
\end{proof}

For each $\delta \in X_*(T) \cap \mathcal{C}_{\lambda}$, the inclusion $Z_{\delta} \hookrightarrow T$ induces an identification
\[ X_{*}(Z_{\delta})_{\mathbb{R}} = \bigoplus_{i=1}^h \mathbb{R}z_i \oplus \bigoplus_{j \in J_{P_{\delta}}} \mathbb{R} \omega_j^{\vee}\hookrightarrow X_{*}(Z_{\lambda})_{\mathbb{R}}\]
where the last inclusion comes from $Z_{\delta} \subset Z_{\lambda}$. Let us denote by $\chi^{\delta}_i$ the restriction of each character $\chi_i$ to $X_{*}(Z_{\delta})_{\mathbb{R}}$. In coordinates of the basis $\mathcal{B}^{\vee}$, the tuple $(\chi_i^{\delta})_{i=1}^{r}$ is the matrix $M_{\delta}$ obtained from $M$ by deleting the rows corresponding to $\Delta \setminus I_{P_{\delta}}$. Note that the set of all such possible matrices $M_{\delta}$ as we vary $\delta$ is finite, because there are only finitely many possibilities for the set of indexes $J_{P_{\delta}}$.

As a consequence of Lemma \ref{lemma: first formula for nu_(d-1)} we get the following immediate corollary.
\begin{coroll} \label{coroll: second formula nu_(d-1)}
Set $\vec{a} = [c_i]_{i=1}^r$ to be the row vector of size $r$ with the $i^{th}$ element given by $c_i$, and let $\vec{1}$ denote the row vector of size $r$ such that all entries are $1$. Let $\vec{x} = [x_i]$ be the column vector of the coordinates of $\delta \in \mathcal{C}_{\lambda}$ in the basis $\mathcal{B}$. Then we can express $\nu_{d-1}$ in terms of the matrix $M_{\lambda}$ as follows. 
\begin{equation} \label{eqn_5}
\nu_{d-1}(\vec{x}) =  \frac{1}{\sqrt{A_{d}}(d-1)!} \, \cdot \, \frac{(\vec{a} - c\vec{1})M_{\lambda}\vec{x}}{\sqrt{\vec{x}^T M_{\lambda}^{T} M_{\lambda}\vec{x}}}
\end{equation}\qed
\end{coroll}

\begin{lemma} \label{lemma: construction of linear functional} There exists a linear functional $\psi: X^{*}(Z_{\lambda})_{\mathbb{R}}\rightarrow \mathbb{R}$ such that $c_i = \psi(\chi_i^{\lambda})$ for all $i$.
\end{lemma}
\begin{proof}
Recall that $j: U \hookrightarrow X_{s}$ is a big open set and $X_{s}$ is smooth. The pushforward $j_*$ induces an isomorphism of Picard groups $j_* : \Pic(U) \rightarrow \Pic(X_{s})$. Let $X^*(L_{\lambda})$ be the group of characters of the Levi subgroup $L_{\lambda}$. The $L_{\lambda}$-bundle $\mathcal{G}_{L_{\lambda}}$ on $U$ induces a homomorphism $X^*(L_{\lambda}) \rightarrow \Pic(U)$ where any character $\chi \in X^{*}(L_{\lambda})$ is mapped to the associated line bundle $\mathcal{G}_{L_{\lambda}} \times^{L_{\lambda}} \chi$. By post-composing with $j_{\ast}$ we get a homomorphism
\[X^*(L_{\lambda}) \rightarrow \Pic(U) \rightarrow \Pic(X_{s})  \]
On the other hand, the first Chern class induces a homomorphism 
\[\Pic(X_{s}) \rightarrow \mathbb{Z},\;\;\mathcal{L} \mapsto c_1(\mathcal{L}) \cdot H^{d-1} \in \mathbb{Z}.\] 
We compose this with the previous map to get a homomorphism 
\[\varphi: X^*(L_{\lambda}) \rightarrow \mathbb{Z},\;\; \chi\mapsto c_1\left(j_* (\mathcal{G}_{L_{\lambda}} \times^{L_{\lambda}} \chi)\right) \cdot H^{d-1}\] 
Tensor with $\mathbb{R}$ to obtain a linear functional $\varphi_{\mathbb{R}} : X^*(L_{\lambda})_{\mathbb{R}} \rightarrow \mathbb{R}$.

The inclusion $Z_{\lambda} \hookrightarrow L_{\lambda}$ yields a restriction map on characters $\res: X^{*}(L_{\lambda})_{\mathbb{R}} \rightarrow X^*(Z_{\lambda})_{\mathbb{R}}$. This morphism $\res$ is an isomorphism since the inclusion induces an isogeny between $Z_{\lambda}$ and $L_{\lambda}/\, [L_{\lambda}, L_{\lambda}]$. We set $\psi$ to be the linear functional
\[\psi=\varphi_{\mathbb{R}} \circ (\res)^{-1}: X^*(Z_{\lambda}) \rightarrow \mathbb{R},\;\; \; \; \; \overline{\chi}\mapsto  \varphi_{\mathbb{R}}(\chi)=c_1\left(j_*( \,\mathcal{G}_{L_{\lambda}} \times^{L_{\lambda}} \chi)\right) \cdot H^{d-1}\]
for a character $\overline{\chi}$ of $Z_{\lambda}$ that comes from the restriction of a character $\chi$ of $L_{\lambda}$. We are left to prove that $\psi(\chi_i^{\lambda}) = c_i$ for all $i$. 

By definition $\mathcal{E}_{\chi_i^{\lambda}}|_{U}$ is the associated vector bundle $\mathcal{G}_{L_{\lambda}} \times^{L_{\lambda}} V_{\chi^{\lambda}_i}$. Hence we have
$\det(\mathcal{E}_{\chi^{\lambda}_i}|_{U})=\mathcal{G}_{\lambda} \times^{L_{\lambda}} \det(V_{\chi_i^{\lambda}})$. Therefore:
\begin{gather*}c_i =  \frac{c_1(\mathcal{E}_{\chi^{\lambda}_i}) \cdot H^{d-1}}{\rk(\mathcal{E}_{\chi^{\lambda}_i})} = \frac{c_1(\det(\mathcal{E}_{\chi^{\lambda}_i})) \cdot H^{d-1}}{\rk(\mathcal{E}_{\chi_i^{\lambda}})}  = \frac{ c_1(j_{*} \det(\mathcal{E}_{\chi^{\lambda}_i}|_{U})) \cdot H^{d-1}}{\rk(\mathcal{E}_{\chi_i^{\lambda}})}= \frac{\varphi_{\mathbb{R}}\left(\det(V_{\chi_i^{\lambda}}) \right)}{\rk(\mathcal{E}_{\chi_i^{\lambda}})}
\end{gather*}

As a representation of $Z_{\lambda}$, the weight space $V_{\chi_i^{\lambda}}$ is a direct sum of $\rk(\mathcal{E}_{\chi_i^{\lambda}})$-many copies of the character $\chi^{\lambda}_i$. Hence $\det(V_{\chi_i^{\lambda}}) = \rk(\mathcal{E}_{\chi_i^{\lambda}}) \cdot \chi_i^{\lambda}$ as a $Z_{\lambda}$-character. This means
\[ c_i = \frac{\varphi_{\mathbb{R}}\left(\det(V_{\chi_i^{\lambda}}) \right)}{\rk(\mathcal{E}_{\chi_i^{\lambda}})}  = \frac{\psi\left(\rk(\mathcal{E}_{\chi_i^{\lambda}}) \cdot \chi_i^{\lambda} \right)}{\rk(\mathcal{E}_{\chi_i^{\lambda}})}  = \psi(\chi_i^{\lambda}) \]
\end{proof}
We are now ready to prove the main proposition of this subsection.
\begin{prop} \label{prop: HN boundedness}
The polynomial numerical invariant $\nu$ on the stack $\Bun_{\rho}(X)$ satisfies the HN boundedness condition.
\end{prop}
\begin{proof}
We shall prove property (*) in Lemma \ref{lemma: equivalent HN boundedness condition}. Fix an algebraically closed extension $K \supset k$ and a scheme $S$ of finite type over $K$. After replacing $k$ with the extension $K$, we can assume that the ground field $k$ is algebraically closed. Fix a family $g: S \rightarrow \Bun_{\rho}(X)$ represented by a $\rho$-sheaf $(\mathcal{F}, \sigma)$ on $X_{S}$.

Choose a closed point $s \in S$. Let $f = (\lambda, \mathcal{G}_{\lambda})$ be a nondegenerate weighted reduction of $(\mathcal{F}|_{X_{s}}, \sigma|_{X_{s}})$. Let $U$ be a big open subset of $X_{s}$ where $\mathcal{F}|_{X_{s}}$ is locally-free and $(\lambda, \mathcal{G}_{\lambda})$ is defined. Suppose that $\nu(f) >0$. The first step will be to construct another weighted parabolic reduction $f'$ that satisfies $\nu(f') \geq\nu(f)$.

Let $P_{\lambda}\subset G$ the parabolic subgroup associated to the cocharacter $\lambda$. Note that the discussion and lemmas before the proposition apply to the weighted parabolic reduction $(\lambda, \mathcal{G}_{\lambda})$ of $(\mathcal{F}|_{X_{s}}, \sigma|_{X_{s}})$. The scale-invariant continuous function $\nu_{d-1}$ (see (\ref{eqn_4}) in Lemma \ref{lemma: first formula for nu_(d-1)}) attains a maximal value $\nu_{d-1}(\delta^*)$ for some $\delta^*$ on the domain $\mathcal{C}_{\lambda}$ (by restricting to the compact intersection of $\mathcal{C}_{\lambda}$ with the unit sphere). 
We can assume without loss of generality that the maximum $\delta^*$ is in the interior of the cone $\mathcal{C}_{\lambda}$.  Otherwise, let $F$ be the unique facet of $\mathcal{C}_{\lambda}$ such that $\delta^* \in \Int(F)$. This facet $F$ is the smaller cone of dominant coweights of a parabolic subgroup $P_{\delta^*} \supset P_{\lambda}$. Then we can replace $\lambda$ with a $\delta^*$ and replace the parabolic reduction $(\lambda, \mathcal{G}_{\lambda})$ with $(\delta^*, \mathcal{G}_{\delta^*})$ for the bigger parabolic.

Since the maximum $\delta^*$ is in the interior of $\mathcal{C}_{\lambda}$, it must be a critical point of the function $\nu_{d-1}: \mathcal{C}_{\lambda} \rightarrow \mathbb{R}$ as in (\ref{eqn_5}) in Corollary \ref{coroll: second formula nu_(d-1)}. If $\nu_{d-1}$ is not identically $0$ (equivalently $(\vec{a}-c\vec{1})M_{\lambda} \neq 0$ in (\ref{eqn_5})), the computation of the gradient yields the coordinates $\vec{z}$ of a critical point, up to scaling:
\begin{equation} \label{eqn_6}
\vec{z} = (M_{\lambda}^{T} M_{\lambda})^{-1} M_{\lambda}^{T} (\vec{a}^{T} - c \vec{1}^{T})
\end{equation}

We consider three cases to construct a modified weighted parabolic reduction $f'$:
\begin{enumerate}[(C1)]
    \item Suppose that $\nu_{d-1}( \lambda)\neq \nu_{d-1}(\delta^*)$. Then $\nu_{d-1}$ is not identically $0$. Since all of the matrices and vectors in (\ref{eqn_6}) have rational coordinates, we can scale $\delta^*$ to have rational coordinates. After further scaling, we can assume that $\delta^*$ is a cocharacter in $\mathcal{C}_{\lambda} \cap X_{*}(T)$. It has coordinates
\[ \vec{y} = q\vec{z}=q (M_{\lambda}^{T} M_{\lambda})^{-1} M_{\lambda}^{T} (\vec{a}^{T} - c \vec{1}^{T}) \]
for $q\in \mathbb{Q}_{>0}$ (the negative multiples correspond to minimal critical points). Set the modified weighted parabolic reduction to be $f' = (\delta^*, \mathcal{G}_{\delta^*})$ with $\nu_{d-1}(\delta^*)>\nu_{d-1}(\lambda)$. By construction we have $\nu(f') > \nu(f)$. 

\item Suppose that $\nu_{d-1}(\lambda)=\nu_{d-1}(\delta^*)$ and the function $\nu_{d-1}$ is not identically $0$. Then $\lambda$ is a critical point for $\nu_{d-1}$ in the interior of the cone $\mathcal{C}_{\lambda}$ with coordinates $\vec{y}=q\vec{z}$. In this case we define the weighted parabolic reduction $f'$ to be the original $f = (\lambda, \mathcal{G}_{\lambda})$ and set $\delta^* = \lambda$.

\item Suppose that $\nu_{d-1}(\lambda)=\nu_{d-1}(\delta^*)$ and $\nu_{d-1}$ is identically $0$. In this case
$\vec{z}=0$. We set $f'=f$ and denote $\delta^* = \lambda$.
\end{enumerate}

Let $f'=(\delta^*, \mathcal{G}_{\delta^*})$ with maximizing $\delta^*$ as explained for each case above. Let $(\mathcal{F}_{m}^{\delta^*})_{m\in \mathbb{Z}}$ be its corresponding filtration with associated graded $\bigoplus_{m\in \mathbb{Z}} \mathcal{F}^{\delta^*}_m / \mathcal{F}^{\delta^*}_{m+1}$. 
We are left to find a uniform lower bound for the slopes  $\hat{\mu}_{d-1}(\mathcal{F}^{\delta^*}_m / \mathcal{F}^{\delta^*}_{m+1})$. Since the slope of $\mathcal{F}^{\delta^*}_m / \mathcal{F}^{\delta^*}_{m+1}$ is the same as the slope of its reflexive hull $(\mathcal{F}^{\delta^*}_m / \mathcal{F}^{\delta^*}_{m+1})^{\vee \vee}$, this is equivalent to giving a uniform lower bound for
\[ \widehat{\mu}_{d-1}\left( (\mathcal{F}_m^{\delta^*} / \mathcal{F}_{m+1}^{\delta^*})^{\vee \vee} \right) = \widehat{\mu}_{d-1}\left( \bigoplus_{\{ [\chi^{\lambda}_i] \; \mid \; \langle \delta^* , \chi^{\lambda}_i \rangle = m \}} \mathcal{E}_{\chi^{\lambda}_i} \right)\]
(See Equation (\ref{eqn: identification of hull of graded piece}) in the proof of Lemma \ref{lemma: first formula for nu_(d-1)}). By additivity of degrees and ranks, it suffices to give a uniform lower bound for the slope $\widehat{\mu}_{d-1}(\mathcal{E}_{\chi_i^{\lambda}})$ of each direct summand above. Hirzebruch-Riemann-Roch implies that 
\[\widehat{\mu}_{d-1}(\mathcal{E}_{\chi_i^{\lambda}}) = \frac{1}{A_{d}}\left(c_i + t_1\cdot H^{d-1}\right)\; ,\]
Therefore it suffices to give a lower bound for the numbers $c_i$. Denote the coordinates in $\mathcal{B}^{\vee}$ of the linear functional $\psi$ of Lemma \ref{lemma: construction of linear functional} by the column vector $\vec{\psi}$. Observe that $\psi(\chi_i^{\lambda}) = c_i$ translates into $\vec{a}^{T} = M_{\lambda} \vec{\psi}$. We shall prove uniform bounds for the vector $\vec{a}=[c_i]$. 
In cases (C1) and (C2) the coordinates $\vec{y}$ of $\delta^*$ are given by 
\[ \vec{y} = q\vec{z}=q (M_{\lambda}^{T} M_{\lambda})^{-1} M_{\lambda}^{T} (\vec{a}^{T} - c \vec{1}) \]
We can plug-in $\vec{a}^{T} = M_{\lambda}\vec{\psi}$ to obtain the following:
\[ \vec{y} = q(M_{\lambda}^{T} M_{\lambda})^{-1} M_{\lambda}^{T} M_{\lambda} \vec{\psi} - q c (M_{\lambda}^{T} M_{\lambda})^{-1} M_{\lambda}^{T} \vec{1}^{T}=q\vec{\psi} - q c (M_{\lambda}^{T} M_{\lambda})^{-1} M_{\lambda}^{T} \vec{1}^{T}\]
Multiply both sides by $M_{\lambda}$ to get
\begin{equation} \label{eqn_7}
\vec{a}^{T} = M_{\lambda}\vec{\psi}=\frac{1}{q} M_{\lambda} \vec{y} + c M_{\lambda} (M_{\lambda}^{T} M_{\lambda})^{-1} M_{\lambda}^{T} \vec{1}^{T}.
\end{equation}
Since the family $\mathcal{F}$ is bounded, we see that $c = \frac{c_1(\mathcal{F}|_{X_{s}}) \cdot H^{d-1}}{r}$ admits uniform upper and lower bounds that do not depend on $s$. Also we know that there are finitely many possibilities for the matrix $M_{\lambda}$, one for each possible choice of subset $I_{P_{\lambda}}$ of simple roots corresponding to the parabolic $P_{\lambda}$. This gives uniform upper and lower bounds $C_{max}$ and $C_{min}$ for the coefficients of matrix of the second term $c M_{\lambda} (M_{\lambda}^{T} M_{\lambda})^{-1} M_{\lambda}^{T} \vec{1}^{T}$ in (\ref{eqn_7}), independent of $s$ and the weighted parabolic reduction $f$.

The $i^{th}$ coefficient of the row vector $M_{\lambda}\vec{y}$ is the weight $\langle \delta^*, \chi_i^{\lambda}\rangle$ of the sheaf $\mathcal{E}_{\chi_i^{\lambda}}$. From (\ref{eqn_7}) we get for each index $i$:
\[ C_{min} + \frac{1}{q}\wgt(\mathcal{E}_{\chi_i^{\lambda}}) \leq c_i \leq C_{max} + \frac{1}{q}\wgt(\mathcal{E}_{\chi_i^{\lambda}})\]
We can add $t_1 \cdot H^{d-1}$ to get
\begin{equation}\label{eqn_8}
    C_{min} + \frac{1}{q}\wgt(\mathcal{E}_{\chi_i^{\lambda}}) + t_1 \cdot H^{d-1} \leq A_{d} \cdot \widehat{\mu}_{d-1}(\mathcal{E}_{\chi_i^{\lambda}}) \leq C_{max} + \frac{1}{q}\wgt(\mathcal{E}_{\chi_i^{\lambda}}) + t_1 \cdot H^{d-1}
\end{equation}
Let $\wgt(\mathcal{E}_{\chi_{i_{min}}^{\lambda}})$ be the smallest weight among all of the $\mathcal{E}_{\chi_i^{\lambda}}$. Since $q>0$, we obtain
\begin{gather*}
    \frac{1}{A_{d}} C_{min} + \frac{1}{qA_{d}}\wgt(\mathcal{E}_{\chi_{i_{min}}^{\lambda}}) + \frac{1}{A_{d}} t_1 \cdot H^{d-1} \leq \frac{1}{A_{d}} C_{min} + \frac{1}{q A_{d}}\wgt(\mathcal{E}_{\chi_{i}^{\lambda}}) + \frac{1}{A_{d}}t_1 \cdot H^{d-1}  \leq   \widehat{\mu}_{d-1}(\mathcal{E}_{\chi_i^{\lambda}})
\end{gather*}
for each $i$. 
So it suffices to give a uniform lower bound for $\frac{1}{qA_{d}}\wgt(\mathcal{E}_{\chi_{i_{min}}^{\lambda}})$. Replace $\mathcal{E}_{\chi_{i}^{\lambda}}$ by $\mathcal{E}_{\chi_{i_{min}}^{\lambda}}$ in the upper inequality of (\ref{eqn_8}) to get
\[\widehat{\mu}_{d-1}(\mathcal{E}_{\chi_{i_{min}}^{\lambda}}) - \frac{1}{A_{d}} C_{max} - \frac{1}{A_{d}} t_1 \cdot H^{d-1} \leq  \frac{1}{q A_{d}}\wgt(\mathcal{E}_{\chi_{i_{min}}^{\lambda}})\; .\]
Therefore we just need to give a uniform lower bound for $\widehat{\mu}_{d-1}(\mathcal{E}_{\chi_{i_{min}}^{\lambda}})$. 

Set $m_{min}:=\langle \delta^*, \chi_{i_{min}}^{\lambda}\rangle$. Note that $\mathcal{E}_{\chi_{i_{min}}^{\lambda}}$ is a direct summand of the reflexive hull of the graded piece $\mathcal{F}^{\delta^*}_{m_{{min}}}/ \mathcal{F}_{m_{min}+1}^{\delta^*}$ of the filtration $(\mathcal{F}^{\delta^*}_{m})_{m \in \mathbb{Z}}$. We can think of $\mathcal{E}_{\chi_{i_{min}}^{\lambda}}$ as a quotient of the reflexive hull $(\mathcal{F}^{\delta^*}_{m_{{min}}}/ \mathcal{F}_{m_{min}+1}^{\delta^*})^{\vee \vee}$, and there is a  corresponding quotient sheaf $\mathcal{F}^{\delta^*}_{m_{{min}}}/ \mathcal{F}_{m_{min}+1}^{\delta^*}\twoheadrightarrow \mathcal{Q}$ with the same slope as $\mathcal{E}_{\chi_{i_{min}}^{\lambda}}$. 

Since the weight of $\mathcal{E}_{\chi_{i_{min}}^{\lambda}}$ is minimal, it follows that $\mathcal{F}^{\delta^*}_{m_{{min}}}/ \mathcal{F}_{m_{min}+1}^{\delta^*}$ is the nontrivial graded piece with the smallest weight. This implies that $\mathcal{F}^{\delta^*}_{m_{{min}}}/ \mathcal{F}_{m_{min}+1}^{\delta^*}$ is a quotient of $\mathcal{F}|_{X_{s}}$. Hence $\mathcal{Q}$ is also a quotient of $\mathcal{F}|_{X_{s}}$. It follows that \[\widehat{\mu}_{d-1}(\mathcal{E}_{\chi_{i_{min}}^{\lambda}})=\widehat{\mu}_{d-1}(\mathcal{Q}) \geq \widehat{\mu}_{d-1}(\mathcal{F}|_{X_{s}})_{min}\; ,\]
where $\widehat{\mu}_{d-1}(\mathcal{F}|_{X_{s}})_{min}$ is the smallest slope appearing among the associated graded pieces of the Gieseker-Harder-Narasimhan filtration of $\mathcal{F}|_{X_{s}}$ (cf. \cite[Def. 1.3.2]{huybrechts.lehn}). The quantity  $\widehat{\mu}_{d-1}(\mathcal{F}|_{X_{s}})_{min}$ is uniformly bounded as we vary $s \in S$, by constructibility of the locus of each fixed Harder-Narasimhan type (see e.g. \cite{nitsure-schematicsheaves}). This concludes the proof of cases (C1) and (C2).

If we are in case (C3) we have $\vec{z}=0$ (see (\ref{eqn_6})). Then:
\[ (M_{\lambda}^{T} M_{\lambda})^{-1} M_{\lambda}^{T} \vec{a}^{T} = c (M_{\lambda}^{T} M_{\lambda})^{-1} M_{\lambda}^{T} \vec{1}^{T} \]
Using $\vec{a}^{T} = M_{\lambda}\vec{\psi}$ we get $\vec{a}^{T} = c M_{\lambda}(M_{\lambda}^{T} M_{\lambda})^{-1} M_{\lambda}^{T} \vec{1}^{T}$, which is the second term in (\ref{eqn_7}) and is bounded by our discussion of the cases before.
\end{proof}
\end{subsection}

\begin{subsection}{$\Theta$-stratification, semistability and leading term filtrations} \label{subsection: theta stratification} After all of our work in this section, we can now construct a $\Theta$-stratification on $\text{Bun}_{\rho}(X)$.

\begin{defn}[$\Theta$-stratification, {\cite[Defn. 2.1.2]{halpernleistner2021structure}}] \label{defn: theta stratification}
Let $\mathcal{M}$ be an algebraic stack locally of finite type and with affine diagonal over $k$. A $\Theta$-stratification of $\mathcal{M}$ consists of a collection of open substacks $(\mathcal{M}_{\leq c})_{c \in \Gamma}$ indexed by a totally ordered set $\Gamma$. We require the following conditions to be satisfied
\begin{enumerate}[(1)]
\item $\mathcal{M}_{\leq c} \subset \mathcal{M}_{\leq c'}$ for all $c< c'$.
\item $\mathcal{M} = \bigcup_{c \in \Gamma} \mathcal{M}_{\leq c}$.
\item For all $c$, there exists a $\Theta$-stratum $\mathfrak{S}_c \subset \text{Filt}(\mathcal{M}_{\leq c})$ of $\mathcal{M}_{\leq c}$ (in the sense of \cite[Defn. 2.1.1]{halpernleistner2021structure}) such that
\[ \mathcal{M}_{\leq c} \setminus \text{ev}_1(\mathfrak{S}_c) = \bigcup_{c' < c} \mathcal{M}_{\leq c'}\]
\item For every point $p \in \mathcal{M}$, the set $\left\{ c \in \Gamma \, \mid \, p \in \mathcal{M}_{\leq c}\right\}$ has a minimal element.
\end{enumerate}
\end{defn}

 \begin{thm} \label{thm: theta stratification}
The polynomial numerical invariant $\nu$ defines a $\Theta$-stratification on the stack $\Bun_{\rho}(X)$.
 \end{thm}
 \begin{proof}
 \cite[Thm. 2.26]{torsion-freepaper} (or \cite[Thm. B]{halpernleistner2021structure}) shows that it suffices to check that $\nu$ satisfies HN boundedness and is strictly $\Theta$-monotone on the stack $\Bun_{\rho}(X)$. This follows by Propositions \ref{prop: HN boundedness} and \ref{prop: invariant is monotone}.
 \end{proof}

The following is a translation of the notion of $\Theta$-semistability from \cite[\S 2.1]{halpernleistner2021structure} for our numerical invariant $\nu$.
\begin{defn}\label{defn: Gieseker semistable}
Let $K \supset k$ be an algebraically closed field extension. Let $(\mathcal{F}, \sigma)$ be a $\rho$-sheaf on $X_K$. We say that $(\mathcal{F},\sigma)$ is Gieseker semistable if for all weighted parabolic reductions $(\lambda, \mathcal{G}_{\lambda})$ with associated underlying filtrations $(\mathcal{F}_{m})_{m \in \mathbb{Z}}$ (as in Definition \ref{defn: weighted parabolic reduction singular G bundles}) we have
\[ \sum_{m \in \mathbb{Z}} m \cdot (\overline{p}_{\mathcal{F}_{m}/\mathcal{F}_{m+1}} - \overline{p}_{\mathcal{F}}) \cdot \rk(\mathcal{F}_m/\mathcal{F}_{m+1}) \leq 0 \]
In general, if $K \supset k$ is any field extension, then we say that a $\rho$-sheaf $(\mathcal{F}, \sigma)$ on $X_{K}$ is Gieseker semistable if the pullback $(\mathcal{F}|_{X_{\overline{K}}}, \sigma|_{X_{\overline{K}}})$ is semistable for any algebraic closure $\overline{K} \supset K$.
A $\rho$-sheaf $(\mathcal{F}, \sigma)$ on $X_{K}$ is called unstable if it is not Gieseker semistable.
\end{defn}

The summation by parts argument in \cite[Prop. 4.1]{torsion-freepaper} shows that the expression in Definition \ref{defn: Gieseker semistable} can be rewritten as 
\begin{equation}\label{eqn: ss1}
\sum_{m \in \mathbb{Z}} m \cdot (\overline{p}_{\mathcal{F}_{m}/\mathcal{F}_{m+1}} - \overline{p}_{\mathcal{F}}) \cdot \rk(\mathcal{F}_m/\mathcal{F}_{m+1}) = \sum_{m \in \mathbb{Z}} (\overline{p}_{\mathcal{F}_{m}} - \overline{p}_{\mathcal{F}})\cdot \rk(\mathcal{F}_m)
\end{equation}

By Theorem \ref{thm: theta stratification}, the polynomial numerical invariant $\nu$ defines a $\Theta$-stratification on $\Bun_{\rho}(X)$. In particular the locus of semistable $\rho$-sheaves is an open substack, which we will denote by $\Bun_{\rho}(X)^{ss}$. Note that $\Bun_{\rho}(X)^{ss}$ is a disjoint union of open and closed substacks where the Hilbert polynomial of the underlying torsion-free sheaf is constant. For every rational polynomial $P \in \mathbb{Q}[n]$, we will denote by $\Bun_{\rho}(X)^{ss, P}$ the open and closed substack of $\Bun_{\rho}(X)^{ss}$ parametrizing Gieseker semistable $\rho$-sheaves $(\mathcal{F}, \sigma)$ such that the Hilbert polynomial of $\mathcal{F}$ is $P$.

We will next explain what it means to have a $\Theta$-stratification on $\Bun_{\rho}(X)$ in terms of canonical filtrations.

\begin{defn}
Suppose that $K$ is algebraically closed. Let $(\mathcal{F}, \sigma)$ be an unstable $\rho$-sheaf on $X_{K}$. A weighted parabolic reduction $f= (\lambda, \mathcal{G}_{\lambda})$ is called canonical if the polynomial invariant $\nu(f)$ is maximal among all weighted parabolic reductions of $(\mathcal{F}, \sigma)$. We say that the associated underlying filtration $(\mathcal{F}_m)_{m \in \mathbb{Z}}$ to $f= (\lambda, \mathcal{G}_{\lambda})$ is a canonical filtration of $(\mathcal{F}, \sigma)$.
\end{defn}
We already know that a weighted parabolic reduction $(\lambda, \mathcal{G}_{\lambda})$ determines a unique underlying filtration $(\mathcal{F}_m)_{m \in \mathbb{Z}}$. By Proposition \ref{prop: filtrations of singular G bundles vs filtrations of torsion-free sheaves}, the converse is true. Namely, the underlying filtration $(\mathcal{F}_{m})_{m \in \mathbb{Z}}$ uniquely determines the weighted parabolic reduction $(\lambda, \mathcal{G}_{\lambda})$ over an open subset $U$ where $\bigoplus_{m \in \mathbb{Z}} \mathcal{F}_{m}/ \mathcal{F}_{m+1}$ is locally-free. Hence the data of a canonical weighted parabolic reduction is equivalent to specifying its underlying canonical filtration.
\begin{defn}
Let $K \supset k$ be an arbitrary field extension. Let $(\mathcal{F}, \sigma)$ be an unstable $\rho$-sheaf on $X_{K}$. A $\mathbb{Z}$-weighted filtration $(\mathcal{F}_m)_{m \in \mathbb{Z}}$ of $\mathcal{F}$ is called a canonical filtration if for any algebraic closure $\overline{K} \supset K$ we have that $(\mathcal{F}_m|_{X_{\overline{K}}})_{m \in \mathbb{Z}}$ is a canonical filtration of $(\mathcal{F}|_{X_{\overline{K}}}, \sigma|_{X_{\overline{K}}})$
\end{defn}
The fact that $\nu$ defines a $\Theta$-stratification on $\Bun_{\rho}(X)$ implies that any field valued point of $\Bun_{\rho}(X)$ admits a $\nu$-maximizing $\Theta$-filtration that is unique up to scaling the weights. We can use our concrete descriptions of $\Theta$-filtrations for $\Bun_{\rho}(X)$ to translate this into the following.
\begin{prop} \label{prop: canonical filtrations}
Let $K \supset k$ be an arbitrary field extension. Let $(\mathcal{F}, \sigma)$ be a $\rho$-sheaf on $X_{K}$. Then, $(\mathcal{F}, \sigma)$ admits a canonical filtration $(\mathcal{F}_{m})_{m \in \mathbb{Z}}$ which is uniquely determined up to scaling all of the indexes by a constant rational number.

Furthermore, if the group $G_{K}$ is split, then there exists a (uniquely determined) canonical weighted parabolic reduction $(\lambda, \mathcal{G}_{\lambda})$ defined over $K$ such that its associated underlying filtration is $(\mathcal{F}_m)_{m \in \mathbb{Z}}$.
\qed
\end{prop}

\begin{defn} \label{defn: leading term filtration}
Let $(\mathcal{F}, \sigma)$ be field valued point of $\Bun_{\rho}(X)$. Suppose that $(\mathcal{F}, \sigma)$ is unstable. We call the canonical filtration $(\mathcal{F}_m)_{m \in \mathbb{Z}}$ the leading term HN filtration of $(\mathcal{F}, \sigma)$. It is uniquely determined up to scaling the weights, by Proposition \ref{prop: canonical filtrations}.

On the other hand, if $(\mathcal{F}, \sigma)$ is Gieseker semistable, then we define the leading term HN filtration to be $(\mathcal{F}_m)_{m \in \mathbb{Z}}$ where
\[  \mathcal{F}_m = \begin{cases} \mathcal{F} \text{ \; \; if m $\leq$ 0} \; \\ 0 \; \; \; \; \; \; \; \text{otherwise} \; \end{cases}\]
\end{defn}

 One should think of the leading term HN filtration as an analogue of the slope filtration. Here the coefficient of degree $i$ in the reduced Hilbert polynomial plays the role of the slope. This is made precise in the following
\begin{example} \label{example: torsion free sheaves}
    Suppose that $G = \GL(V)$ and $\rho$ is the identity. Then $\Bun_{\rho}(X)$ is equivalent to the stack $\Coh_{r}^{tf}(X)$ of torsion-free sheaves of rank $r$ on $X$. In this case Proposition \ref{prop: filtrations of torsion free sheaves} gives a description of a filtration of a torsion-free sheaf $\mathcal{F}$ in $\Bun_{\rho}(X)$ as a decreasing $\mathbb{Z}$-filtration of $\mathcal{F}$ by saturated subsheaves $\mathcal{F}_m \subset \mathcal{F}$. The proof of \cite[Prop. 5.16]{torsion-freepaper} shows that a torsion-free sheaf in $\Bun_{\rho}(X)$ is Gieseker semistable in our sense if and only if it is Gieseker semistable in the sense of \cite[Defn. 1.2.4]{huybrechts.lehn}. 

The leading term HN filtration for torsion-free sheaves in $\Bun_{\rho}(X)$ recovers the weighted leading term filtration described in \cite[Thm. 5.11]{torsion-freepaper}, which we reproduce here. Let
\[ 0 = \mathcal{F}^{\text{HN}}_{0} \subset \mathcal{F}^{\text{HN}}_{1} \subset \mathcal{F}^{\text{HN}}_{2} \subset \cdots \subset \mathcal{F}^{\text{HN}}_{n} = \mathcal{F} \]
denote the Gieseker-Harder-Narasimhan filtration of $\mathcal{F}$ (cf. \cite[Thm. 1.3.4]{huybrechts.lehn}).
Then we have \[ \overline{p}_{\mathcal{F}_1^{\text{HN}}} > \overline{p}_{\mathcal{F}_2^{\text{HN}} / \mathcal{F}_1^{\text{HN}}} > \cdots > \overline{p}_{\mathcal{F}_n^{\text{HN}} / \mathcal{F}_{n-1}^{\text{HN}}} \]
Let $i$ denote the minimal index such that 
\[\widehat{\mu}_{h}(\mathcal{F}_1^{\text{HN}}) = \widehat{\mu}_{h}(\mathcal{F}_2^{\text{HN}} / \mathcal{F}_1^{\text{HN}}) = \cdots = \widehat{\mu}_{h}(\mathcal{F}_n^{\text{HN}} / \mathcal{F}_{n-1}^{\text{HN}})\]
for all $h >i$. In other words, $i$ is the largest integer such that the $i^{th}$ slopes of the graded pieces of the filtration are not all equal. We have
\[ \widehat{\mu}_{i}(\mathcal{F}_1^{\text{HN}}) \geq \widehat{\mu}_{i}(\mathcal{F}_2^{\text{HN}} / \mathcal{F}_1^{\text{HN}}) \geq \cdots \geq \widehat{\mu}_{i}(\mathcal{F}_n^{\text{HN}} / \mathcal{F}_{n-1}^{\text{HN}}) \]
where at least one of the inequalities above is strict. 

We define indexes $h_0 < h_1 < h_2 < \cdots < h_j = n$ as follows. We let $h_0 = 0$, and let
\[\{h_m\}_{m=1}^j = \left\{\; 1 \leq l \leq n  \; \left| \; \widehat{\mu}_i(\mathcal{F}_{l+1}^{\text{HN}} / \mathcal{F}_{l}^{\text{HN}}) < \widehat{\mu}_i(\mathcal{F}_{l}^{\text{HN}} / \mathcal{F}_{l-1}^{\text{HN}}) \right. \right\}\]
This ensures that $\widehat{\mu}_{i}(\mathcal{F}_{h_k}^{\text{HN}} / \mathcal{F}_{h_{k - 1}}^{\text{HN}}) > \widehat{\mu}_{i}(\mathcal{F}_{h_{k+1}}^{\text{HN}} / \mathcal{F}_{h_{k}}^{\text{HN}})$. For each $k$, we set $\mathcal{F}^{\text{l-term}}_{k} = \mathcal{F}^{\text{HN}}_{h_k}$.

The filtration
\[ 0 = \mathcal{F}^{\text{l-term}}_{0} \subset \mathcal{F}^{\text{l-term}}_{1} \subset \mathcal{F}^{\text{l-term}}_{2} \subset \cdots \subset \mathcal{F}^{\text{l-term}}_{j} = \mathcal{F} \]
is called the unweighted leading term filtration of $\mathcal{F}$. We call $i$ the leading term index.

We remark that, if $\mathcal{F}$ is slope unstable, then $i=d-1$ (recall that $\widehat{\mu}_{d-1}(\mathcal{F})$ is the slope of $\mathcal{F}$), and the leading term filtration is the slope-Harder-Narasimhan filtration.

Let $\mathcal{F}\in \Bun_{\rho}(X) = \Coh_{r}^{tf}(X)$ and assume that $\mathcal{F}$ is unstable. Suppose that the unweighted leading term filtration for $\mathcal{F}$ is given by
\[ 0 = \mathcal{F}^{\text{l-term}}_{0} \subset \mathcal{F}^{\text{l-term}}_{1} \subset \mathcal{F}^{\text{l-term}}_{2} \subset \cdots \subset \mathcal{F}^{\text{l-term}}_{j} = \mathcal{F} \]
with leading slope $\widehat{\mu}_{i}$. Set $\mu = \widehat{\mu}_{i}(\mathcal{F})$, $\mu_k = \widehat{\mu}_{i}\left(\mathcal{F}^{\text{l-term}}_{k} / \mathcal{F}^{\text{l-term}}_{k-1}\right)$. Choose a positive rational number $L$ such that $L\cdot \mu\in \mathbb{Z}$ and $L \cdot \mu_k \in \mathbb{Z}$ for all $k$.

Then, we define the weighted leading term filtration $(\mathcal{F}^{\text{Can}}_m)_{m \in \mathbb{Z}}$ as follows (up to scaling of the weights).
\begin{enumerate}[(1)]
    \item $\mathcal{F}^{\text{Can}}_m = 0$ for $m > L (\mu_1 - \mu)$.
    \item Let $1 \leq k \leq j-1$. For each $L(\mu_k - \mu) \geq m > L (\mu_{k+1} - \mu)$, we set $\mathcal{F}^{\text{Can}}_m = \mathcal{F}^{\text{l-term}}_{k}$. 
    \item $\mathcal{F}^{\text{Can}}_m = \mathcal{F}$ for $L(\mu_{j} - \mu) \geq m$.
\end{enumerate}
In other words, the weighted leading term filtration is given by the leading term filtration and the $k^{th}$ jump happens at the integer $L (\mu_k - \mu)$.
\end{example}

 The existence of a $\Theta$-stratification also implies that the notion of leading term HN filtration behaves well in families. In order to explain this precisely, we need to first define what is a filtration of a family of $\rho$-sheaves.
 Let $S$ be a $k$-scheme and let $(\mathcal{F}, \sigma)$ be a $\rho$-sheaf on $X_{S}$. A filtration of $\mathcal{F}$ is a $\mathbb{Z}$-weighted filtration $(\mathcal{F}_{m})_{m \in \mathbb{Z}}$ of $\mathcal{F}$ satisfying the following.
\begin{enumerate}[(1)]
    \item $\mathcal{F}_{m+1} \subset \mathcal{F}_{m}$.
    \item $\mathcal{F}_m = 0$ for $m \gg0$ and $\mathcal{F}_m = \mathcal{F}$ for $m\ll 0$.
    \item $\mathcal{F}_{m}/\mathcal{F}_{m+1}$ is a $S$-flat family of torsion-free sheaves.
\end{enumerate}

An argument similar to Proposition \ref{prop: filtrations of torsion free sheaves} shows that the data of such filtration $(\mathcal{F}_m)_{m \in \mathbb{Z}}$ is equivalent to the data of a morphism of stacks $\Theta_S \rightarrow \Coh_{r}^{tf}(X)$ such that the restriction to $1_S$ recovers $\mathcal{F}$. This amounts to a torsion-free sheaf $\widetilde{\mathcal{F}}$ on $X_{\Theta_S}$. Let $U$ be a big open subset of $X_{S}$ where $\bigoplus_{m \in \mathbb{Z}} \mathcal{F}_{m}/\mathcal{F}_{m+1}$ is locally-free. Then, the restriction $\widetilde{\mathcal{F}}|_{U \times_{S} \Theta_{S}}$ is a $\GL(V)$-bundle on $U \times_{S} \Theta_{S}$ and $\sigma$ induces a $G$-reduction of structure group of the restriction $\widetilde{\mathcal{F}}|_{U \times_{S} 1_S}$. 
\begin{defn}
We say that $(\mathcal{F}_{m})_{m \in \mathbb{Z}}$ as above is a filtration of $(\mathcal{F}, \sigma)$ if the $G$-reduction of structure group (uniquely) extends to $U \times_{S} \Theta_{S}$.
\end{defn}
An argument similar to Proposition \ref{prop: nonexplicit filtrations singular G bundles} implies that such a filtration of $(\mathcal{F}, \sigma)$ is equivalent to a morphism $\Theta_S \rightarrow \Bun_{\rho}(X)$ such that the restriction to $1_S$ recovers $(\mathcal{F}, \sigma)$.
\begin{defn}
Let $S$ be a $k$-scheme and let $(\mathcal{F}, \sigma)$ be a $\rho$-sheaf on $X_{S}$. We say that a filtration $(\mathcal{F}_m)_{m \in \mathbb{Z}}$ of $(\mathcal{F}, \sigma)$ is a relative leading term HN filtration if for all points $s \in S$ the restriction $(\mathcal{F}_m|_{X_{s}})_{m \in \mathbb{Z}}$ is the leading term HN filtration of $(\mathcal{F}|_{X_{s}}, \sigma|_{X_{s}})$.
\end{defn}

Suppose that $S$ is a connected scheme of finite type over a field extension $K \supset k$ and assume that $G_{K}$ is split. In this case,  every relative filtration comes from a relative weighted parabolic reduction. Indeed, let $(\mathcal{F}, \sigma)$ be a $\rho$-sheaf on $X_{S}$, and let $(\mathcal{F}_{m})_{m \in \mathbb{Z}}$ be a relative filtration of $(\mathcal{F}, \sigma)$. Let $\widetilde{\mathcal{G}}$ denote the corresponding $G$-bundle on $\Theta_{K} \times_{K} U$ for some big open subset $U \subset X_{S}$. We denote by $\mathcal{G}$ the $G$-bundle on $U$ corresponding to $(\mathcal{F}, \sigma)$ (i.e. the restriction $\widetilde{\mathcal{G}}|_{1 \times S}$). Since $X$ is geometrically connected and smooth, the fibers of $X_{S} \rightarrow S$ are integral, hence the fibers of $U \rightarrow S$ are also integral. Since $S$ is connected, we conclude that $U$ is a connected scheme of finite type over $K$.
Proposition \ref{prop: filtrations of G bundles over U} shows that $\widetilde{\mathcal{G}}$ arises from applying the Rees construction to a weighted parabolic reduction $(\lambda, \mathcal{G}_{\lambda})$ of the $G$-bundle $\mathcal{G}$ on $U$. We call $(\lambda, \mathcal{G}_{\lambda})$ a relative weighted parabolic reduction of $(\mathcal{F}, \sigma)$ and note that the filtration $(\mathcal{F}_{m})_{m \in \mathbb{Z}}$ is uniquely determined by $(\lambda, \mathcal{G}_{\lambda})$. Then we say that the filtration $(\mathcal{F}_m)_{m \in \mathbb{Z}}$ comes from the relative weighted parabolic reduction $(\lambda, \mathcal{G}_{\lambda})$.

Let $S$ be a $k$-scheme and $(\mathcal{F}, \sigma)$ be a $\rho$-sheaf on $X_{S}$. Let $H\rightarrow S$ be a locally closed stratification of $S$. This means that $H$ is the disjoint union of locally closed subschemes of $S$ and the morphism $H \rightarrow S$ is surjective.
\begin{defn}
 We say that the $S$-scheme $H$ is universal for leading term HN filtrations of $(\mathcal{F}, \sigma)$ if it represents the functor that classifies pairs $\left(T, [(\mathcal{F}_m)_{m \in \mathbb{Z}}]\right)$, where $T$ is a $S$-scheme and $[(\mathcal{F}_m)_{m \in \mathbb{Z}}]$ is an equivalence class of relative leading term HN filtrations of $(\mathcal{F}|_{X_{T}}, \sigma|_{X_{T}})$ up to rational scaling the weights.
\end{defn}
The fact that the numerical invariant $\nu$ induces a $\Theta$-stratification implies the following property of the leading term HN filtration in families, called a ``schematic stratification" in \cite{nitsure-schematicsheaves, gurjar-nitsure, nitsuregurjar2}.
\begin{prop} \label{prop: relative mixed Gieseker filtrations}
Relative leading term HN filtrations induce a stratification of $\Bun_{\rho}(X)$ by locally closed substacks. Equivalently, for any $k$-scheme $S$ and every $\rho$-sheaf $(\mathcal{F}, \sigma)$ on $X_{S}$ there is a unique locally closed stratification $H \to S$ that is universal for leading term HN filtrations of $(\mathcal{F}, \sigma)$. 

If, in addition, $S$ is locally of finite type over a field extension $K \supset k$ such that $G_{K}$ is split, then the restriction of the universal relative leading term HN filtration\footnote{
More precisely, we choose a representative up to scaling. 
} to each connected component of the stratification $H$ comes from a relative weighted parabolic reduction.
\qed
\end{prop}
\end{subsection}
\end{section}

\begin{section}{Moduli space for the semistable locus} \label{section: moduli for the semistable locus}

\begin{subsection}{Completeness of the stack $\Bun_{\rho}(X)$} \label{section: properness}
In this subsection we prove that the stack $\Bun_{\rho}(X)$ satisfies the existence part of the valuative criterion for properness \cite[\href{https://stacks.math.columbia.edu/tag/0CLK}{Tag 0CLK}]{stacks-project} in the case of complete discrete valuation rings (see Theorem \ref{thm: valuative criterion for properness rho sheaves}). 

The ideas that we use in this section (purely ramified extensions, Bruhat-Tits buildings) already appear in \cite[Section 6]{balaji.seshadri.1} (see \cite{balaji.parameswaran.2} for positive characteristic). Besides, V. Balaji has also considered extensions of families of principal bundles in higher dimensions, in the context of Donaldson-Uhlenbeck compactification \cite{balaji.du}. The case of positive characteristic for curves has also been studied by J. Heinloth \cite{heinloth.1,heinloth.2}. For torsion free sheaves, the valuative criteria was proved by S. Langton \cite{langton}.

Let $R$ be a complete discrete valuation $k$-algebra. We set $K = \Frac(R)$, and choose a uniformizer $\pi$. By the Cohen structure theorem, $R$ is isomorphic to the power series ring $\kappa\bseries{\pi}$, where $\kappa$ is the residue field. Consider the base change $X_R$.  Let $\eta_{\kappa}$ denote the generic point of the special fiber $X_{\kappa}$. Similarly, let $\eta_K$ be the generic point of the generic fiber $X_K$. Note that the local $\mathcal{O}_{X_R, \eta_{\kappa}}$ is a discrete valuation ring with generic point $\eta_K$ and special point $\eta_{\kappa}$.

\begin{lemma} \label{lemma: equivalence valuative criteria properness}
Let $(\mathcal{F}, \sigma)$ be a $\rho$-sheaf over $X_K$. Let $\mathcal{P}$ be the corresponding $G$-bundle over the generic fiber $\eta_K$. The following statements are equivalent:
\begin{enumerate}[(1)]
    \item There is a $\rho$-sheaf on $X_R$ that restricts to $(\mathcal{F}, \sigma)$ over $X_K$.
    \item There is a $G$-bundle on $\Spec(\mathcal{O}_{X_R, \eta_{\kappa}})$ that restricts to $\mathcal{P}$ over $\eta_K$.
\end{enumerate}
\end{lemma}
\begin{proof}
$(1) \Longrightarrow (2)$. Suppose that $(\widetilde{\mathcal{F}}, \zeta)$ is a $\rho$-sheaf on $X_R$ that extends $(\mathcal{F}, \sigma)$. Let $U$ be a big open subset of $X_R$ where $\widetilde{\mathcal{F}}$ is locally-free. Let $\widetilde{\mathcal{P}}$ be the $G$-bundle on $U$ corresponding to the restriction of $(\widetilde{\mathcal{F}}, \zeta)$. Then $\widetilde{\mathcal{P}}|_{\Spec(\mathcal{O}_{X_R, \eta_{\kappa}})}$ is a $G$-bundle on $\Spec(\mathcal{O}_{X_R, \eta_{\kappa}})$ that extends $\mathcal{P}$.

$(2) \Longrightarrow (1)$. Suppose that there is a $G$-bundle $\widetilde{\mathcal{P}}$ on $\Spec(\mathcal{O}_{X_R, \eta_{\kappa}})$ that extends $\mathcal{P}$. Let $\mathcal{A}$ be the corresponding $\GL(V)$-bundle on $\Spec(\mathcal{O}_{X_R, \eta_{\kappa}})$ obtained via extension of structure group using $\rho$. $\mathcal{A}$ is a vector bundle of rank $r$ on $\Spec(\mathcal{O}_{X_R, \eta_{\kappa}})$. By construction we have a canonical identification  $\mathcal{A}|_{\eta_{K}} \xrightarrow{\sim} \mathcal{F}|_{\eta_{K}}$. Now \cite[Prop. 6]{langton} implies that there is a (unique) family of torsion-free sheaves $\widetilde{\mathcal{F}}$ on $X_R$ along with isomorphisms $\widetilde{\mathcal{F}}|_{X_{K}} \xrightarrow{\sim} \mathcal{F}$ and $\widetilde{\mathcal{F}}|_{\Spec(\mathcal{O}_{X_R, \eta_{\kappa}})} \xrightarrow{\sim} \mathcal{A}$ that are compatible with the identification $\mathcal{A}|_{\eta_{K}} \xrightarrow{\sim} \mathcal{F}|_{\eta_{K}}$.

We have a section $s: \Spec(\mathcal{O}_{X_R, \eta_{\kappa}}) \rightarrow \Red_{G}(\widetilde{\mathcal{F}})|_{\Spec(\mathcal{O}_{X_R, \eta_{\kappa}})}$ corresponding to the $G$-reduction $\widetilde{P}$. It identifies over the point $\eta_{K}$ with the section $\sigma: X_{K} \rightarrow \Red_{G}(\mathcal{F})$. Since $\Red_{G}(\widetilde{\mathcal{F}})$ is of finite type over $X_{R}$, we can spread the compatible pair $s$ and $\sigma$ to a section 
\[\widetilde{\sigma} : X_R \setminus Z \rightarrow \Red_{G}(\widetilde{\mathcal{F}})\]
defined away from a proper closed subscheme $Z \subsetneq X_{\kappa}$ of the special fiber $X_{\kappa}$. Note that every point $z \in Z$ has codimension at least 2 in $X_{R}$. Since $X_{R}$ is smooth over the regular scheme $\Spec(R)$, it is itself regular \cite[\href{https://stacks.math.columbia.edu/tag/07NF}{Tag 07NF}]{stacks-project}. It follows that for every $z \in Z$, the local ring $\mathcal{O}_{X_{R}, z}$ has depth at least 2. By the proof of Lemma \ref{lemma: lemma on very big open subsets} (c), the section $\widetilde{\sigma}$ on $X_{R} \setminus Z$ uniquely extends to a section $\zeta: X_{R} \rightarrow \Red_{G}(\widetilde{\mathcal{F}})$ defined over the whole $X_{R}$. The $\rho$-sheaf $(\widetilde{\mathcal{F}}, \zeta)$ restricts to $(\mathcal{F}, \sigma)$, as desired.
\end{proof}

We would like to show that the extension result in Lemma \ref{lemma: equivalence valuative criteria properness} (2) is true after maybe passing to a finite extension $R' \supset R$ of the discrete valuation ring $R$. In order to do so, we prove the following general result.\footnote{V. Balaji pointed out to us that Proposition \ref{prop: extension after ramified cover} could also be deduced from \cite[Proposition 3.12]{raghunathan.ramanathan}.}
\begin{prop} \label{prop: extension after ramified cover}
Let $\kappa$ be a perfect field. Let $G$ be a quasi-split connected reductive group over $\kappa$. Suppose that the two following conditions are satisfied:
\begin{enumerate}[(1)]
    \item The maximal central torus $ Z \subset G$ is split over $\kappa$.
    \item Let $H$ denote the kernel of of the simply-connected cover $\widetilde{G/Z} \rightarrow G/Z$ of the semisimple group $G/Z$. The characteristic of $\kappa$ does not divide the order of $H$.
\end{enumerate}
Set $\kappa \bseries{\pi}$ to be the formal power series ring. Let $\mathcal{P}$ be a $G$-bundle on the fraction field $\Spec(\kappa\pseries{\pi})$. Then, there exists some $m \geq 1$ such that the pullback of $\mathcal{P}$ to $\Spec(\kappa\pseries{\pi^{\frac{1}{m}}})$ extends to a $G$-bundle on $\Spec(\kappa\bseries{\pi^{\frac{1}{m}}})$.
\end{prop}
\begin{proof}
Define $K \vcentcolon = \kappa\pseries{\pi}$ and $\mathcal{O}_{K} = \kappa\bseries{\pi}$. We fix the choice of a separable closure $K_s$ of $K$. Set $K_{unr}$ to be the maximal unramified extension of $K$ inside $K_{s}$. We will denote by $\mathcal{O}_{K_{unr}}$ the ring of integers of $K_{unr}$. Let $\Gamma_{K} := \Gal(K_{s}/K)$ be the absolute Galois group of $K$. We denote by $\Gamma_{unr} \vcentcolon = \Gal(K_{unr}/K)$ the unramified Galois group. We denote by $H^1(K, G)$ (resp. $H^1(\mathcal{O}_K, G)$ ) the pointed set of $G$-bundles on $\Spec(K)$ (resp. $\Spec(\mathcal{O}_{K})$). We have a natural restriction morphism $ H^1(\mathcal{O}_{K}, G)  \rightarrow H^1(K, G)$. We are given an element $\mathcal{P} \in H^1(K, G)$. Our goal is to show that, after maybe replacing $K$ with an extension $K( \pi^{\frac{1}{m}})$, there exists some $\widetilde{\mathcal{P}} \in H^1(\mathcal{O}_{K}, G)$ that restricts to $\mathcal{P}$.

Since $G$ is smooth, the base-change $\mathcal{P}_{K_{s}}$ is trivializable over $\Spec(K_{s})$. The descent data for $\mathcal{P}$ yield a class in the nonabelian Galois cohomology set $H^1(\Gamma_{K}, G(K_{s}))$. Since $G$ is connected and $\kappa$ is perfect, \cite[3.14 (3)]{bruhat-tits-iii} shows that $H^1(\Gamma_{K}, G(K_{s})) = H^1(\Gamma_{unr}, G(K_{unr}))$. So the torsor $\mathcal{P}$ comes from an unramified cohomology class $\overline{c} \in H^1(\Gamma_{unr}, G(K_{unr}))$. Notice that there is a natural action of $\Gamma_{unr}$ on $G(\mathcal{O}_{K_{unr}})$. There is a map $H^1(\Gamma_{unr}, G(\mathcal{O}_{K_{unr}})) \rightarrow H^1(\Gamma_{unr}, G(K_{unr}))$ induced by the inclusion of $\Gamma_{unr}$-groups $G(\mathcal{O}_{K_{unr}}) \hookrightarrow G(K_{unr})$. Our goal is to show that, after maybe replacing $K$ with $K(\pi^{\frac{1}{m}})$, we have that $\overline{c}$ comes from a class in $H^1(\Gamma_{unr}, G(\mathcal{O}_{K_{unr}}))$.

Let us first suppose that $G$ is semisimple. There is a finite unramified Galois extension $L \supset K$ contained in $K_{unr}$ such that $\mathcal{P}_{L}$ is trivializable. Let $\mathcal{O}_{L}$ denote the ring of integers of $L$. Set $\Gamma_{L/K}:= \Gal(L/K)$. This is a finite quotient of $\Gamma_{unr}$. The descent data for $\mathcal{P}$ yield a class $\overline{c}$ in the nonabelian Galois cohomology set $H^1(\Gamma_{L/K}, G(L))$.

Since $G_{K}$ is quasi-split, it follows that $G_{K}$ is residually quasi-split \cite[2.2]{bruhat-tits-iii}. Therefore, we can use the description of unramified Galois cohomology given in \cite{bruhat-tits-iii}, we just need to be a bit careful if $G$ is not simply-connected. We will denote by $G^{00}$ the subgroup of $G(K_{unr})$ generated by the $K_{unr}$-points of all parahoric subgroups of $G_{K_{unr}}$. By \cite[1.8]{bruhat-tits-iii}, $G^{00}$ admits the following description. Let $T$ denote a maximal torus of $G$ over $\kappa$. We can think of $T_{K}$ as the centralizer of a maximal split torus of $G_{K}$, since $G_{K}$ is quasi-split. Let $j: \widetilde{G} \rightarrow G$ be a simply connected cover of $G$. Then, we have $G^{00} = T(\mathcal{O}_{K_{unr}}) \cdot j(\widetilde{G}(K_{unr}))$.

Choose a cocycle $c \in Z^1(\Gamma_{L/K}, G(L))$ representing $\overline{c}$. We will show that, after maybe replacing $K$ with some ramified extension $K(\pi^{\frac{1}{m}})$, the cocycle $c$ takes values in the subgroup $G^{00}$. Let $H$ denote the finite kernel of the isogeny $j: \widetilde{G} \rightarrow G$. Note that $H$ is a finite product of groups $\mu_j$ of $j^{th}$ roots of unity. By assumption (2), the characteristic of $\kappa$ does not divide any of the orders of the $\mu_j$ appearing in $H$. After maybe enlarging the unramified extension $L$, we can assume that all of these roots of unity are contained in $L$. So $H_{L}$ is a finite constant \'etale group scheme. Since $H$ is smooth, the following sequence of groups is exact.
\[ 1 \rightarrow H(K_{s}) \rightarrow \widetilde{G}(K_{s}) \rightarrow G(K_{s}) \rightarrow 1 \]
We can take $\Gal(K_{s}/L)$ invariants to get the following exact sequence of pointed sets
\[ 1 \rightarrow H(L) \rightarrow \widetilde{G}(L) \rightarrow G(L) \rightarrow H^1\left(\Gal(K_{s}/L), H(K_{s})\right) \]
Hence the obstructions to lifting to $\widetilde{G}(L)$ the finitely many elements in the image of the cocycle $c \in Z^1(\Gamma_{L/K}, G(L))$ consist of finitely many cohomology classes in $H^1\left(\Gal(K_{s}/L), H(K_{s})\right)$. By Kummer theory, these are killed after taking some cyclic extensions of $L$. After further enlarging the unramified extension $L$, we can assume that these cyclic extensions are totally ramified. Such cyclic extension must be of the form $L(\pi_{L}^{\frac{1}{m}})$ for some uniformizer $\pi_{L}$ of $L$. Such uniformizer must be of the form $\pi_{L} = u \pi$, where $u \in \mathcal{O}_{L}^{\times}$. After replacing $L$ with $L(u^{\frac{1}{m}})$, we can assume that the obstructions in $H^1(\Gal(K_{s}/L), H(K_{s}))$ vanish on the extension $L(\pi^{\frac{1}{m}})$ for some $m \geq 1$. So after replacing $K$ with some sufficiently big ramified extension $K(\pi^{\frac{1}{m}})$, we can assume that $c$ takes values in the group $j(\widetilde{G}(L)) \subset G^{00}$. Hence $\overline{c}$ is a class in the Galois cohomology set $H^1\left(\Gamma_{L/K}, G^{00} \cap G(L)\right)$. Since $G_{K}$ is residually-quasisplit, the result \cite[3.15]{bruhat-tits-iii} and the construction of \cite[3.4 Lem. 2]{bruhat-tits-iii} show that $\overline{c}$ is represented by a cocycle $z \in Z^1(\Gamma_{L/K}, P(\mathcal{O}_{L}))$ taking values in the group of $\mathcal{O}_{L}$-points of a parahoric $P$ of $G_{K}$ defined over $K$. We can without loss of generality take $P$ to be in the apartment corresponding to a maximal split torus $S$ inside $T_{K}$, because $G(K)$ acts transitively on apartments. The argument in \cite[Lem. 2.4]{larsen-maximality} (see also \cite[Proposition 8]{serre} and \cite[Lemma I.1.3.2]{gille}) shows that there exists some $m \geq 1$ and some $K$-rational cocharacter $\lambda: \mathbb{G}_m \to S$ such that $\lambda(\pi^{\frac{1}{m}}) \in G(K[\pi^{\frac{1}{m}}])$ conjugates the base-changed parahoric $P_{\mathcal{O}_{K(\pi^{\frac{1}{m}})}}$ into the hyperspecial parahoric $G_{\mathcal{O}_{K(\pi^{\frac{1}{m}})}}$. After replacing $K$ with $K(\pi^{\frac{1}{m}})$, we can use the element $\lambda(\pi^{\frac{1}{m}})$ to conjugate the cocycle $z$ into a cocycle $z' \in Z^1(\Gamma_{L/K}, G(\mathcal{O}_{L}))$. This shows that $\overline{c}$ comes from a cohomology class in $H^1(\Gamma_{L/K}, G(\mathcal{O}_{L}))$, and thus concludes the proof in the case when $G$ is semisimple.

We now deal with the general case. Let $G$ be an arbitrary quasi-split connected reductive group over $\kappa$. Let $Z$ denote the maximal central torus of $G$. We have a short exact sequence of algebraic groups
\[  1 \rightarrow Z \rightarrow G \rightarrow \overline{G} \rightarrow 1  \]
Here $\overline{G} = G/Z$ is a quasi-split semisimple group over $\kappa$ satisfying condition (2) in the statement of this proposition. By assumption, $Z$ is split, and so it is isomorphic to a finite product $(\mathbb{G}_m)^{h}$ of copies of $\mathbb{G}_m$. We have an exact sequence
\[ 1 \rightarrow Z(K_{unr}) \rightarrow G(K_{unr}) \rightarrow \overline{G}(K_{unr}) \rightarrow H^1(\Gamma_{K_{s}/ K_{unr}}, Z(K_{s}))\]
Hilbert's Theorem 90 implies that $H^1(\Gamma_{K_{s}/ K_{unr}}, Z(K_{s})) = 0$, and so we have a short exact sequence of $\Gamma_{unr}$-groups
\[ 1 \rightarrow Z(K_{unr}) \rightarrow G(K_{unr}) \rightarrow \overline{G}(K_{unr}) \rightarrow 1\]
The same thing holds for $\mathcal{O}_{K_{unr}}$-points. After taking $\Gamma_{unr}$-invariants and replacing $Z$ with $(\mathbb{G}_m)^{h}$, we get a diagram of pointed sets
\[
\begin{tikzcd}[column sep=tiny]
  H^1(\Gamma_{unr}, \mathcal{O}_{K_{unr}}^{\times})^{\oplus h} \ar[r] \ar[d] &   H^1(\Gamma_{unr}, G(\mathcal{O}_{K_{unr}}))  \ar[d] \ar[r]& H^1(\Gamma_{unr}, \overline{G}(\mathcal{O}_{K_{unr}})) \ar[d] \ar[r] &  H^2(\Gamma_{unr}, \mathcal{O}_{K_{unr}}^{\times})^{\oplus h} \ar[d]\\    H^1(\Gamma_{unr},K_{unr}^{\times})^{\oplus h} \ar[r] &   H^1(\Gamma_{unr}, G(K_{unr})) \ar[r]& H^1(\Gamma_{unr}, \overline{G}(K_{unr})) \ar[r] &  H^2(\Gamma_{unr}, K_{unr}^{\times})^{\oplus h}
\end{tikzcd}
\]
After an application of Hilbert's Theorem 90 to the first column, we get
\[
\begin{tikzcd}
  0 \ar[r] &   H^1(\Gamma_{unr}, G(\mathcal{O}_{K_{unr}}))  \ar[d] \ar[r]& H^1(\Gamma_{unr}, \overline{G}(\mathcal{O}_{K_{unr}}))  \ar[d] \ar[r, "f_1"] &  H^2(\Gamma_{unr}, \mathcal{O}_{K_{unr}}^{\times})^{\oplus h} \ar[d, "g"]\\    0 \ar[r] &   H^1(\Gamma_{unr}, G(K_{unr})) \ar[r, "f_2"]& H^1(\Gamma_{unr}, \overline{G}(K_{unr})) \ar[r, "f_3"] &  H^2(\Gamma_{unr}, K_{unr}^{\times})^{\oplus h}
\end{tikzcd}\]
where the two horizontal rows are exact.

Let $\overline{c} \in H^1(\Gamma_{unr}, G(K_{unr}))$ be our unramified Galois cohomology class coming from $\mathcal{P}$. Our goal is to lift $\overline{c}$ to a class in $H^1(\Gamma_{unr}, G(\mathcal{O}_{K_{unr}}))$. Consider the image $f_2(\overline{c}) \in H^1(\Gamma_{unr}, \overline{G}(K_{unr}))$. By the proof of the semisimple case, after maybe replacing $K$ with an extension $K(\pi^{\frac{1}{m}})$ we can lift $f_2(\overline{c})$ to a cohomology class $\overline{d} \in H^1(\Gamma_{unr}, \overline{G}(\mathcal{O}_{K_{unr}}))$. We now want to lift $\overline{d}$ to a class in $H^1(\Gamma_{unr}, G(\mathcal{O}_{K_{unr}}))$ whose image in $H^1(\Gamma_{unr}, G(\mathcal{O}_{K_{unr}}))$ is the original $\overline{c}$. For this it suffices to show that $f_1(\overline{d}) = 0$.

We have $g(f_1(\overline{d})) = f_3(f_2(\overline{c})) = 0$. Therefore, it suffices to show that $g$ is injective. The morphism $g$ is the direct sum of $h$ copies of the morphism $H^2(\Gamma_{unr}, \mathcal{O}_{K_{unr}}^{\times}) \rightarrow H^2(\Gamma_{unr}, K_{unr}^{\times})$ coming from the inclusion of $\Gamma_{unr}$-groups $\mathcal{O}_{K_{unr}}^{\times} \hookrightarrow K_{unr}^{\times}$. These groups fit into a short exact sequence
\[ 1 \rightarrow \mathcal{O}_{K_{unr}}^{\times} \rightarrow K_{unr}^{\times} \rightarrow \mathbb{Z} \rightarrow 1 \]
Here $\Gamma_{unr}$ acts trivially on $\mathbb{Z}$. We get an exact sequence of abelian groups
\begin{equation}\label{eqn_Gextension}
H^1(\Gamma_{unr}, \mathbb{Z}) \rightarrow  H^2(\Gamma_{unr}, \mathcal{O}_{K_{unr}}^{\times}) \rightarrow H^2(\Gamma_{unr}, K_{unr}^{\times}) \end{equation}
Since $\Gamma_{unr}$ acts trivially on $\mathbb{Z}$, we have that $H^1(\Gamma_{unr}, \mathbb{Z})$ is the group of continuous homomorphisms from $\Gamma_{unr}$ to $\mathbb{Z}$. Since $\Gamma_{unr}$ is profinite, the image of such a homomorphism is finite. Therefore it must be $0$, because $\mathbb{Z}$ is torsion-free. We conclude that $H^1(\Gamma_{unr}, \mathbb{Z}) = 0$. Hence (\ref{eqn_Gextension}) shows that $g$ is injective, as desired. 
\end{proof}
\begin{remark}
The argument that allows us to go from semisimple to reductive groups is similar to the argument
of G. Faltings in \cite[Theorem II.4]{faltings}.
\end{remark}

\begin{remark}
It follows from the argument that $m\geq 1$ can be chosen independently of $\mathcal{P}$. One can give a uniform upper bound that depends only on the root data of $G/Z$.
\end{remark}

\begin{remark}
The argument above works verbatim for a general complete discrete valuation ring $R$ with perfect residue field $\kappa$ (i.e. mixed characteristic). In that case we can take $G$ to be a reductive group over $R$ with connected fibers such that the generic fiber $G_{K}$ is residually quasi-split and the analogous conditions (1) and (2) are satisfied.
\end{remark}

In order to construct our generic $G$-bundle extension, we will make use of formal gluing, also known as Beauville-Laszlo gluing \cite{beauville-laszlo-gluing}. Suppose that $S \to \widehat{S}$ is a faithfully flat morphism of rings. Let $f \in S$ be a nonzero divisor. We denote by $\text{Glue}_{\text{Aff}}$ the category consisting of triples $(\widehat{A}, A_f, \varphi)$, where 
\begin{enumerate}[(1)]
    \item $\widehat{A}$ is an affine scheme over $\Spec(\widehat{S})$.
    \item $A_f$ is an affine scheme over $\Spec(S_f)$.
    \item $\varphi$ is an isomorphism $\varphi: \widehat{A}|_{\Spec(\widehat{S}_f)} \xrightarrow{\sim} A_f|_{\Spec(\widehat{S}_f)}$.
\end{enumerate}
A morphism $(\widehat{A}, A_f, \varphi) \to (\widehat{B}, B_f, \psi)$ consists of a pair of morphisms $(g_1 : \widehat{A} \to \widehat{B}, \, g_2: A_f \to B_f)$ such that their base-changes to $\Spec(\widehat{S}_f)$ fit into a commutative diagram
\[
\begin{tikzcd}
     \widehat{A}|_{\Spec(\widehat{S}_f)} \ar[d, "\varphi"] \ar[r, "g_1|_{\Spec(\widehat{S}_f)}"]  &  A_f|_{\Spec(\widehat{S}_f)} \ar[d, "\psi"] \\ \widehat{B}|_{\Spec(\widehat{S}_f)}
      \ar[r, "g_2|_{\Spec(\widehat{S}_f)}"] & B_f|_{\Spec(\widehat{S}_f)}
\end{tikzcd}
\]
Given an affine scheme $A$ over $\Spec(S)$, we can form the two pullbacks $A_{\Spec(\widehat{S})}$ and $A_{\Spec(S_{f})}$. Note that there is a canonical isomorphism $\varphi_{A, can}: \left(A_{\Spec(\widehat{S})}\right)|_{\Spec(\widehat{S}_f)} \xrightarrow{\sim} \left(A_{\Spec(S_f)}\right)|_{\Spec(\widehat{S}_f)}$. This defines a functor
\[ \text{Aff}_{can}: \text{Aff}_{\Spec(S)} \to \text{Glue}_{\text{Aff}}, \; \; \; A \mapsto (A_{\Spec(\widehat{S})}, A_{\Spec(S_f)}, \varphi_{A,can})\]
Let $G$ be an affine flat group scheme over $\Spec(S)$. We can similarly define a category $\text{Glue}_{G\text{Bun}}$ by replacing the tuple of affine schemes with $G$-bundles and requiring that $\varphi$ is a morphism of $G$-bundles. Denote by $BG(S)$ the groupoid of $G$-bundles over $\Spec(S)$. The functor $\text{Aff}_{can}$ induces another functor
\[ G\text{Bun}_{can}: BG(S) \to  \text{Glue}_{G\text{Bun}}, \; \; \; \mathcal{P} \mapsto (\mathcal{P}|_{\Spec(\widehat{S})}, \mathcal{P}|_{\Spec(S_f)}, \varphi_{\mathcal{P},can})\]
The following lemma is usually known in the literature as Beauville-Laszlo gluing.
\begin{lemma} \label{lemma: formal gluing}
Suppose that the faithfully flat morphism $S \to \widehat{S}$ induces an isomorphism $S/(f\cdot S) \xrightarrow{\sim} \widehat{S}/(f\cdot\widehat{S})$. Then both of the functors $\text{Aff}_{can}$ and $G\text{Bun}_{can}$ described above are equivalences of categories.
\end{lemma}
\begin{proof}
We shall use \cite[\href{https://stacks.math.columbia.edu/tag/05E5}{Tag 05ES}]{stacks-project}. Since the gluing of modules preserves the tensor product structure on modules (cf. \cite[\href{https://stacks.math.columbia.edu/tag/05EU}{Tag 05EU}]{stacks-project}), the equivalence in \cite[\href{https://stacks.math.columbia.edu/tag/05E5}{Tag 05ES}]{stacks-project} induces an equivalence of monoid objects (i.e. algebras). By applying the $\Spec$ functor we deduce that $\text{Aff}_{can}$ is an equivalence, and that it is compatible with fiber products relative to the base. This implies that $\text{Aff}_{can}$ induces an equivalence of categories of affine schemes equipped with a right $G$-action\footnote{One defines the corresponding gluing category similarly as in the case for $G$-bundles above.}. If $\mathcal{P}$ is a scheme equipped with a right $G$-action, then by descent for the flat cover $\Spec(\widehat{S}) \to \Spec(S)$ we know that $\mathcal{P}$ is a $G$-bundle if and only if $\mathcal{P}|_{\Spec(\widehat{S})}$ is a $G$-bundle. This shows that the equivalence on affine schemes equipped with a right $G$-action restricts to an equivalence $G\text{Bun}_{can}$ on the subcategories of $G$-bundles.
\end{proof}

Now we are ready to prove the main theorem of this subsection.
\begin{thm} \label{thm: valuative criterion for properness rho sheaves}
The structure morphism $\Bun_{\rho}(X) \rightarrow \Spec(k)$ satisfies the existence part of the valuative criterion for properness \cite[\href{https://stacks.math.columbia.edu/tag/0CLK}{Tag 0CLK}]{stacks-project} in the case of complete discrete valuation rings. 
\end{thm}
\begin{proof}
Let $R$ be a discrete valuation $k$-algebra, with fraction field $K$. Let $(\mathcal{F}, \sigma)$ be a $\rho$-sheaf on $X_K$. We will use the notation from Lemma \ref{lemma: equivalence valuative criteria properness}. By Lemma \ref{lemma: equivalence valuative criteria properness}, it suffices to find a finite extension $R' \supset R$ of discrete valuation rings such that the base-change $\mathcal{P}_{K'}$ on the generic point $\eta_{K'}$ of $\Spec(\mathcal{O}_{X_{R'}, \eta_{\kappa'}})$ admits an extension to a $G$-bundle $\widetilde{\mathcal{P}}$ on the whole $\Spec(\mathcal{O}_{X_{R'}, \eta_{\kappa'}})$. 

We can replace $k$ with $\kappa$ in order to assume that $\kappa$ is the ground field. By the Cohen structure theorem, we have $R = \kappa\bseries{\pi}$. After replacing $\kappa$ with a finite extension $\kappa'$, we can assume that $G$ is split. Set $S= \mathcal{O}_{X_{R}, \eta_{\kappa}}$. Note that $S \supset R$ is a discrete valuation ring with uniformizer $\pi \in R$. Let $\widehat{S}$ denote the completion $\widehat{\mathcal{O}}_{X_{R}, \eta_{\kappa}}$. Let $\widehat{\mathcal{P}}$ denote the base change of $\mathcal{P}$ to $\Spec(\Frac(\widehat{S}))$. We have that $\widehat{S} = \kappa_{S}\bseries{\pi}$, where $\kappa_{S} \supset \kappa$ is the residue field of $S$ (equivalently $\kappa_{S}$ is the fraction field of the special fiber $X_{\kappa}$). 

By formal gluing (Lemma \ref{lemma: formal gluing} applied to $S \to \widehat{S}$ and $f= \pi$), extending $\mathcal{P}$ to $\Spec(S)$ is equivalent to extending $\widehat{\mathcal{P}}$ to $\Spec(\widehat{S})$. Indeed, such an extension $(\widehat{\mathcal{P}})^{\sim}$ to $\Spec(\widehat{S})$ yields a triple $(\mathcal{P}, (\widehat{\mathcal{P}})^{\sim}, \varphi)$, which by Lemma \ref{lemma: formal gluing} glues to an extension of $\mathcal{P}$ over $\Spec(S)$. Hence we may replace $S$ with $\widehat{S}$. Note that we are allowed to replace $\kappa\bseries{\pi}$ with a finite totally ramified extension $\kappa\bseries{\pi^{\frac{1}{m}}}$ for some $m \geq 1$. This replaces $\widehat{S}$ with $\kappa_{S}\bseries{\pi^{\frac{1}{m}}}$. The theorem is therefore implied by Proposition \ref{prop: extension after ramified cover} (by taking $\kappa = \kappa_{S}$ in that proposition).
\end{proof}

Since the semistable locus $\Bun_{\rho}(X)^{ss}$ comes from a $\Theta$-stratification, we automatically have semistable reduction for $\rho$-sheaves by \cite[Thm. 6.5, Cor. 6.12]{alper2019existence}. We record this result as a separate theorem.
\begin{thm}[Semistable Reduction] \label{thm: semistable reduction}
The substack $\Bun_{\rho}(X)^{ss}$ of Gieseker semistable points satisfies the existence part of the valuative criterion for properness.
\qed
\end{thm}

The arguments in this subsection yield an interesting result in the case of curves, where $\rho$-sheaves are the same thing as $G$-bundles.
\begin{coroll} \label{coroll: extension G-bundles curves}
Let $X$ be a smooth projective geometrically connected curve over a field $\kappa$ of arbitrary characteristic. Let $G$ be a quasi-split connected reductive group over $\kappa$ satisfying the following two conditions:
\begin{enumerate}[(1)]
    \item The maximal central torus of $G$ is split over $\kappa$.
    \item Let $H$ denote the kernel of of the simply-connected cover $\widetilde{[G,G]} \rightarrow [G,G]$ of the derived subgroup of $G$. Then, the characteristic of $\kappa$ does not divide the order of $H$.
\end{enumerate}
Let $R$ be a complete discrete valuation ring over $\kappa$ with fraction field $K$. Suppose that the residue field of $R$ is perfect. Let $\mathcal{P}$ be a $G$-bundle on the generic fiber $X_{K}$. Then, after maybe taking a root of the uniformizer of $R$, we can extend $\mathcal{P}$ to a $G$-bundle on $X_{R}$.
\end{coroll}
\begin{proof}
This follows by applying the proof of Lemma \ref{lemma: equivalence valuative criteria properness}, using formal gluing as in Theorem \ref{thm: valuative criterion for properness rho sheaves} and then applying Proposition \ref{prop: extension after ramified cover}.
\end{proof}

\end{subsection}
\begin{subsection}{The degree of a $\rho$-sheaf} Fix the choice of an algebraic closure $\overline{k} \supset k$. We denote by $\Gamma_{k}$ the absolute Galois group $\Gamma_k \vcentcolon = \Gal(\overline{k}/ k)$. Let $X^*(G_{\overline{k}})$ be the group of characters of the base-change $G_{\overline{k}}$. Let $X^*(G_{\overline{k}})^{\vee}$ denote the dual group $\Hom\left(X^*(G_{\overline{k}}), \, \mathbb{Z}\right)$. This group admits a natural continuous action of the profinite absolute Galois group $\Gamma_{k}$. We write $\Lambda(G) \vcentcolon = \left(X^*(G_{\overline{k}})^{\vee} \right)_{\Gamma_{k}}$ to denote the quotient group of $\Gamma_{k}$-coinvariants of $(X^*(G_{\overline{k}}))^{\vee}$. Note that $\Lambda(G)$ does not depend on the choice of algebraic closure $\overline{k}$.

\begin{example} \label{example: degree split group}
If the group $G$ is split over $k$, then $X^*(G_{\overline{k}})$ coincides with the group $X^*(G)$ of characters defined over $k$. The action of the Galois group $\Gamma_{k}$ on $X^*(G)$ is trivial. Therefore $\Lambda(G) = X^*(G)^{\vee}$.
\end{example}

\begin{example}
If the group $G$ is semisimple, then $\Lambda(G) = 0$.
\end{example}

Let $K \supset k$ be an arbitrary algebraically closed field extension and let $(\mathcal{F}, \sigma)$ be a $\rho$-sheaf on $X_{K}$. Let $U$ be any big open subset of $X_{K}$ where $\mathcal{F}$ is locally-free and denote by $j$ be the open inclusion $j: U \hookrightarrow X_{K}$. If $\mathcal{G}$ is the corresponding $G_{K}$-bundle on $U$, then for any character $\chi \in X^*(G_{K})$ we can form the associated line bundle $\mathcal{G} \times^{G_{K}} \chi$ over $U$. Since $X_{K}$ is smooth, the pushforward $j_*(\mathcal{G} \times^{G_{K}} \chi)$ is a line bundle on $X_{K}$ whose first Chern class is  $c_1\left( j_*(\mathcal{G} \times^{G_{K}} \chi)\right)$. Let $H$ be the hyperplane class on $X$ corresponding to the line bundle $\mathcal{O}_{X}(1)$ and define the group homomorphism (c.f Lemma \ref{lemma: construction of linear functional}):
\[\psi: X^*(G_{K}) \rightarrow \mathbb{Z},\;\; \psi(\chi) \vcentcolon = c_1\left( j_*(\mathcal{G} \times^{G_{K}} \chi)\right) \cdot H^{d-1} .\]

If we denote by $L$ the algebraic closure of $k$ in $K$, then base-change induces a canonical identification $X^*(G_{L}) \xrightarrow{\sim} X^*(G_{K})$. Choose an isomorphism $\overline{k} \xrightarrow{\sim} L$ over $k$. This yields an isomorphism of abelian groups $X^*(G_{\overline{k}}) \xrightarrow {\sim} X^*(G_{L})$ which identifies $\psi$ with an element of $X^*(G_{\overline{k}})^{\vee}$. Since any two choices of isomorphisms $\overline{k} \xrightarrow{\sim} L$ differ by an element of the Galois group $\Gamma_{k}$, $\psi$ yields a well-defined element of the quotient $\Lambda(G)$ that is independent of the choice of $\overline{k} \xrightarrow{\sim} L$.
\begin{defn}\label{defn: degree}
Let $K \supset k$ be an algebraically closed field extension. Let $(\mathcal{F}, \sigma)$ be a $\rho$-sheaf on $X_{K}$. We define the degree $\deg \,(\mathcal{F}, \sigma)$ to be the element $\vartheta\in \Lambda(G)$ constructed above. For an arbitrary field extension $K \supset k$, we set $\deg \,(\mathcal{F}, \sigma)$ to be $\deg \,(\mathcal{F}|_{X_{\widetilde{K}}}, \sigma|_{X_{\widetilde{K}}})$ for any choice of an algebraically closed field extension $\widetilde{K} \supset K$.
\end{defn}

\begin{prop} \label{prop: degree is locally constant on families}
Let $S$ be a $k$-scheme and let $(\mathcal{F}, \sigma)$ be a $\rho$-sheaf on $X_{S}$. Let $\deg: S \rightarrow \Lambda(G)$ denote the function defined by $\deg(s) \vcentcolon = \deg\, (\mathcal{F}|_{X_{s}}, \sigma|_{X_{s}})$.
Then $\deg$ is locally constant on $S$.
\end{prop}
\begin{proof}
Since the stack $\Bun_{\rho}(X)$ is locally of finite type over $k$, we can reduce to the case when $S$ is of finite type over $k$. Given that the notion of degree is well-behaved under taking field extensions, we can replace $k$ with an algebraic closure to assume that the ground field is algebraically closed and, hence $\Gamma_{k}$ is trivial.

Let $U$ be a big open subset of $X_{S}$ where $\mathcal{F}$ is locally-free and denote by $j: U \hookrightarrow X_{S}$ the open immersion. For each point $s \in S$, let $j_s$ be the base-change $j_s: U_s \hookrightarrow X_{s}$. Denote by $\mathcal{G}$ the corresponding $G$-bundle on $U$. Let $\chi \in X^*(G)$ be a character of $G$ and denote by $\mathcal{L}$ the associated line bundle $\mathcal{G} \times^{G} \chi$ on $U$. We want to show that the function $f: S \rightarrow \mathbb{Z}$ given by
\[ f(s) \vcentcolon = c_1\left( (j_{s})_* (\mathcal{L}|_{U_{s}}) \right) \cdot H^{d-1}\]
is locally constant on $S$. 

We first claim that $f$ is constructible on $S$. By Noetherian induction \footnote{By Noetherianity, there is a minimal closed subscheme $Z \subset S$ such that the restriction of $f$ to $Z$ is constructible. By looking at the irreducible components and passing to the reduced subscheme, we may assume that $Z$ is integral. If we show that $f$ is constant on an open $U \subset Z$, then constructibility on $Z$ will follow from the constructibility on $Z \setminus U$ (by minimality) and on $U$.} we need to prove that if $S$ is integral then $f$ is constant on an open subset of $S$. Assume that $S$ is integral. Since $k$ is perfect, there exists an open dense subset of $S$ that is smooth over $k$. Therefore we can reduce to the case when $S$ is a smooth variety. For each $s \in S$, consider the following Cartesian diagram
\[
\begin{tikzcd}
     U_s \ar[r] \ar[d, "j_s"] & U \ar[d, "j"]\\ X_s \ar[r] & X_{S}
\end{tikzcd}
\]
We get an induced base-change morphism of quasicoherent sheaves $\varphi: j_*(\mathcal{L}) |_{X_{s}} \rightarrow (j_s)_* \left(\mathcal{L}|_{U_{s}}\right)$ where the right-hand side is a line bundle. Since $S$ is smooth over $k$ and $X_{S} \rightarrow S$ is a smooth morphism, we conclude that $X_{S}$ is also smooth over $k$, so $X_{S}$ is a regular scheme. The complement of the open subset $U$ has codimension $\geq 2$ inside $X_{S}$, therefore the pushforward $j_*(\mathcal{L})$ is a line bundle. It follows that $j_*(\mathcal{L}) |_{X_{s}}$ is a line bundle, and so $\varphi$ is a morphism of line bundles on $X_s$. Note that the restriction of $\varphi$ to the big open subset $U_{s} \subset X_{s}$ is an isomorphism by construction. We conclude that $\varphi$ must be an isomorphism of line bundles. This shows that we can equivalently write $f$ as
\[ f(s) = c_1\left( (j_*\mathcal{L})|_{X_{s}}\right) \cdot H^{d-1}\]
Since $j_* \mathcal{L}$ is a line bundle its first Chern class is locally constant in families, then $f$ is constant on the variety $S$. This concludes the proof that $f$ is constructible.

In order to show that $f$ is locally constant, it suffices to check that $f$ is constant when the base $S$ is of the form $\Spec(R)$ for some discrete valuation ring $R$. In this case $S$ is a regular integral scheme. Since $X_{S} \rightarrow S$ is smooth, it follows that $X_{S}$ is regular \cite[\href{https://stacks.math.columbia.edu/tag/07NF}{Tag 07NF}]{stacks-project} and we can apply the same argument as in the proof of constructibility to see that $f$ is constant on $S$.
\end{proof}

As a consequence of Proposition \ref{prop: degree is locally constant on families}, we have the following.

\begin{coroll}\label{coroll: stratification_by_degree}
For every $\vartheta \in \Lambda(G)$, there exists an open and closed substack $\Bun_{\rho}(X)_{\vartheta} \subset \Bun_{\rho}(X)$ parametrizing $\rho$-sheaves with degree $\vartheta$. Then
\[ \Bun_{\rho}(X) = \bigsqcup_{\vartheta \in \Lambda(G)} \Bun_{\rho}(X)_{\vartheta}.\]
\end{coroll}
\end{subsection}

\begin{subsection}{Boundedness and moduli space for the semistable locus} \label{subsection: boundedness of the semistable locus}
Let $P \in \mathbb{Q}[n]$ be a Hilbert polynomial and let $\vartheta \in \Lambda(G)$. We define $\Bun_{\rho}(X)^{ss, P}_{\vartheta} \vcentcolon = \Bun_{\rho}(X)^{ss, P} \cap \Bun_{\rho}(X)_{\vartheta}$. By Corollary \ref{coroll: stratification_by_degree} we get a decomposition of the semistable locus
\[ \Bun_{\rho}(X)^{ss, P}  = \bigsqcup_{\vartheta \in \Lambda(G)} \Bun_{\rho}(X)^{ss, P}_{\vartheta}  \]
The next proposition shows that each piece $\Bun_{\rho}(X)^{ss, P}_{\vartheta}$ is bounded.
\begin{prop} \label{prop: boundedness of the semistable locus}
For all $\vartheta \in \Lambda(G)$ and all $P \in \mathbb{Q}[n]$, the stack $\Bun_{\rho}(X)^{ss, P}_{\vartheta}$ is quasicompact.
\end{prop}
\begin{proof}
After replacing $k$ with a field extension, we can assume that $k$ is algebraically closed. Since the forgetful morphism $\Forget: \Bun_{\rho}(X) \rightarrow \Coh_{r}^{tf}(X)$ is of finite type, it is sufficient to show that the restriction $\Forget|_{\Bun_{\rho}(X)^{ss, P}_{\vartheta}}$ factors through a quasicompact open substack of $\Coh_{r}^{tf}(X)$. We shall show that the following set:
\[ \mathfrak{S} = \left\{\; \Forget\left( (\mathcal{F}, \sigma) \right) \; \; \mid \; \; (\mathcal{F}, \sigma) \in \Bun_{\rho}(X)^{ss, P}_{\vartheta}(k) \; \right\} \] 
is a bounded family of torsion-free sheaves.

Let $(\mathcal{F}, \sigma)\in \Bun_{\rho}(X)^{ss, P}_{\vartheta}(k)$. We claim that there exists a uniform upper bound $\mu_{max}$ (only depending on $\vartheta$) such that for all torsion-free subsheaves $\mathcal{V} \subset \mathcal{F}$ we have $\widehat{\mu}_{d-1}(\mathcal{V}) \leq \mu_{max}$. Since the Hilbert polynomial $P$ of $\mathcal{F}$ is fixed, it will follow that the family $\mathfrak{S}$ is bounded by \cite[Thm. 3.3.7]{huybrechts.lehn}.

We are left to show the claim. Since $G$ is a reductive group over $k$ of characteristic $0$, $G$ is linearly reductive and the representation $\rho: G \rightarrow \GL(V)$ decomposes into a direct sum of irreducible representations $V = \bigoplus_{i=1}^q V_i$. If we denote $H \vcentcolon = \prod_{i=1}^q \GL(V_i)$, then $\rho$ factors through a faithful homomorphism $\varphi: G \rightarrow H$. Since each representation $V_i$ is irreducible, it follows that the image of $\varphi$ is not contained in any proper parabolic subgroup of $H$.

Let $U$ be a big open subset where  $\mathcal{F}$ is locally-free and let $\mathcal{G}$ be the corresponding $G$-bundle on $U$ which, by Proposition \ref{prop: gieseker semistable implies slope semistable}, is slope semistable in the sense of Ramanathan \cite{Ram79}. By the analogue of \cite[Thm. 5.1]{biswas-gomez-restriction} for rational principal bundles (as explained in \cite[\S7]{biswas-gomez-restriction}), the restriction of $\mathcal{G}$ to a general smooth complete intersection curve $C \hookrightarrow X$ of degree sufficiently large is a (slope) semistable $G$-bundle as in \cite{ramanathan-stable}. Let $\mathcal{H} \vcentcolon = \varphi_*(\mathcal{G})$ be the associated $H$-bundle on $U$. Notice that $\varphi_*(\mathcal{G}|_{C})$ is the restriction $\mathcal{H}|_{C}$. By \cite[Prop. 5.3]{biswas-holla-hnreduction} $\varphi_*(\mathcal{G}|_{C})$ is (slope) semistable on $C$, so the restriction of $\mathcal{H}$ to a general smooth complete intersection curve $C$ of degree sufficiently large is slope semistable in the sense of Ramanathan. By the argument in \cite[Remark 5.2]{biswas-gomez-restriction}, it follows that, since $(\mathcal{F},\sigma)$ is slope semistable, then $\mathcal{H}$ is also slope semistable.

This $H$-bundle $\mathcal{H}$ amounts to the data of a $q$-tuple $(\mathcal{E}_{U, i})_{i=1}^q$ of vector bundles on $U$. Let $j$ denote the open inclusion $j:U \hookrightarrow X$. For each $i$, define $\mathcal{E}_i \vcentcolon = j_*(\mathcal{E}_{U,i})$. Note that $\mathcal{E}_i$ is torsion free because it is the push-forward of a torsion free sheaf. The fact that $\mathcal{H}$ is semistable implies that each $\mathcal{E}_i$ is a slope semistable sheaf: otherwise the slope filtration of $\mathcal{E}_i$ would restrict over $U$ to a destabilizing filtration of the $\mathcal{E}_{U,i}$, thus contradicting the fact that $\mathcal{H}$ is semistable. 

For each $i$, $\det(V_i)$ is a character of $G$ and, by the definition of degree, $\vartheta(\det(V_{i})) = c_1(\mathcal{E}_i) \cdot H^{d-1}$. By Hizebruch-Riemann-Roch, the slope $\widehat{\mu}_{d-1}(\mathcal{E}_i)$ is given by
\begin{gather*} \widehat{\mu}_{d-1}(\mathcal{E}_i) = \frac{1}{A_{d} \, \rk(\mathcal{E}_i)}c_1(\mathcal{E}_{i}) \cdot H^{d-1} + \frac{1}{A_d}t_1 \cdot H^{d-1} =  \frac{1}{A_{d} \, \dim(V_i)}\vartheta(\det(V_i)) + \frac{1}{A_d}t_1 \cdot H^{d-1}\end{gather*}
Therefore $\widehat{\mu}_{d-1}(\mathcal{E}_i)$ is completely determined by $\vartheta$, and does not depend on $(\mathcal{F}, \sigma)$. Denote by $\mu_{max}$ the maximum of the slopes $\widehat{\mu}_{d-1}(\mathcal{E}_1), \widehat{\mu}_{d-1}(\mathcal{E}_2), \ldots , \widehat{\mu}_{d-1}(\mathcal{E}_q)$.

We have $\bigoplus_{i=1}^q \mathcal{E}_i = \mathcal{F}^{\vee \vee}$, where  $\mathcal{F}^{\vee \vee}$ denotes the reflexive hull of $\mathcal{F}$, 
because both sides restrict to the same associated vector bundle $\rho_*(\mathcal{G})$ on $U$. If we set $\mathcal{F}_j \vcentcolon = \left(\bigoplus_{i=1}^j\mathcal{E}_i\right) \cap \mathcal{F}\subset \mathcal{F}^{\vee \vee}$, this defines a filtration of $\mathcal{F}$ by saturated subsheaves:
\[0 = \mathcal{F}_0 \subset \mathcal{F}_1 \subset \cdots \subset \mathcal{F}_{q-1} \subset \mathcal{F}_q = \mathcal{F}\]
For each $j$ we have $(\mathcal{F}_j/\mathcal{F}_{j-1})|_{U} = \mathcal{E}_{U,j}$, and therefore $(\mathcal{F}_j/\mathcal{F}_{j-1})^{\vee \vee} = \mathcal{E}_j$. This implies that $\widehat{\mu}_{d-1}(\mathcal{F}_j/\mathcal{F}_{j-1}) = \widehat{\mu}_{d-1}(\mathcal{E}_j)$. Since $\mathcal{E}_j$ is slope semistable, it also follows that $\mathcal{F}_j/\mathcal{F}_{j-1}$ is a slope semistable sheaf.
We have therefore found a filtration of $\mathcal{F}$ such that each graded piece $\mathcal{F}_j/\mathcal{F}_{j-1}$ is a slope semistable sheaf such that $\widehat{\mu}_{d-1}(\mathcal{F}_j/\mathcal{F}_{j-1}) \leq \mu_{max}$. This implies that subsheaves $\mathcal{V} \subset \mathcal{F}$ satisfy $\widehat{\mu}_{d-1}(\mathcal{V}) \leq \mu_{max}$, as desired.
\end{proof}

\begin{thm} \label{thm: moduli space single general linear}
For each choice of Hilbert polynomial $P \in \mathbb{Q}[n]$ and degree $\vartheta \in \Lambda(G)$, the stack $\Bun_{\rho}(X)^{ss, P}_{\vartheta}$ admits a proper good moduli space $M_{\rho}(X)^{ss, P}_{\vartheta}$ in the sense of \cite{alper-good-moduli}.
\end{thm}
\begin{proof}
    It suffices to check the hypotheses of \cite[Thm. 2.26]{torsion-freepaper} (see also \cite[Thm. B]{halpernleistner2021structure}). Proposition \ref{prop: invariant is monotone} implies that the numerical invariant $\nu$ is strictly $\Theta$-monotone and strictly $S$-monotone. By Proposition \ref{prop: boundedness of the semistable locus}, the stack $\Bun_{\rho}(X)^{ss, P}_{\vartheta}$ is bounded. By Theorem \ref{thm: valuative criterion for properness rho sheaves}, the stack $\Bun_{\rho}(X)$ satisfies the existence part of the valuative criterion of properness.
\end{proof}
If the group $G$ is semisimple, then the representation $\rho$ factors through the subgroup $\SL(V) \subset \GL(V)$. In this case the moduli space of semistable $\rho$-sheaves was already constructed in \cite{schmitt.singular, glss.singular.char} using GIT. For semisimple groups the proof above provides an alternative stack-theoretic approach which is intrinsic, in the sense that it does not appeal to an embedding into an external GIT problem.

On the other hand, suppose that the representation $\rho$ is central as in \Cref{defn: central representation}. Then, the degree $\vartheta$ is determined by the Hilbert polynomial $P$ and each $\Bun_{\rho}(X)^{ss,P}$ admits a proper good moduli space $M_{\rho}(X)^{ss, P}$. Let $\Bun_{G}(X)^s$ denote the stack of slope stable $G$-bundles on $X$. By Corollary \ref{coroll: semistability comparison central reps}, the stack $\Bun_{\rho}(X)^{ss}$ of Gieseker semistable $\rho$-sheaves contains $\Bun_{G}(X)^s$ as an open substack, which is saturated in the sense of \cite[Defn. 6.1]{alper-good-moduli} by a standard argument as in the case when $X$ is a curve. Therefore, the separated moduli space $M_{\rho}(X)^{ss} = \bigsqcup_{P\in \mathbb{Q}[n]} M_{\rho}(X)^{ss, P}$ contains an open subspace $M_{G}(X)^s$ which is a good moduli space for the stack $\Bun_{G}(X)^s$.
\end{subsection}
\end{section}

\begin{section}{Generalization to products of general linear groups and comparison with previous work}\label{section: generalization}

\begin{subsection}{Generalization: $\rho$-sheaves for products of general linear groups}
In order to deal with reductive groups $G$ with large center, it is convenient to allow the codomain of $\rho$ to be a product of general linear groups. All of the previous results extend to this more general context.

Fix a tuple $V^{\bullet} = (V^i)_{i=1}^b$ of finite-dimensional vector spaces over $k$ with $r_i \vcentcolon = \dim \, V^i$. Define $\GL(V^{\bullet}) \vcentcolon = \prod_{i=1}^b \GL(V^i)$. We fix once and for all a faithful representation $\rho: G \rightarrow \GL(V^{\bullet})$, which induces representations $\rho_i: G \rightarrow \GL(V^i)$ by composing with the natural projections $\GL(V^{\bullet}) \rightarrow \GL(V^i)$.

Let $S$ be a $k$-scheme. Suppose that we are given a tuple $\mathcal{F}^{\bullet} = (\mathcal{F}^i)_{i=1}^b \in \prod_{i=1}^b \Coh_{r_i}^{tf}(X)\,(S)$ of families of torsion-free sheaves on $X_{S}$. We set
\[ E_{V^i}(\mathcal{F}^i) \vcentcolon = \Sym^{\bullet}(\mathcal{F}^i \otimes_k (V^i)^{\vee}) \otimes_{\mathcal{O}_{X_{S}}} \Sym^{\bullet}(\det(\mathcal{F}^i)^{\vee} \otimes_k \det(V^i)) \]
We also define $H(\mathcal{F}^i, V^i) \vcentcolon = \underline{\Spec}_{X_S}(E_{V^i}(\mathcal{F}^i))$
and $H(\mathcal{F}^{\bullet}, V^{\bullet}) \vcentcolon = \prod_{i=1}^b H(\mathcal{F}^i, \, V^i) \;$.
Each scheme $H(\mathcal{F}^i, V^i)$ admits a right action of the group $G$ via the representation $\rho_i$, as described in Subsection \ref{subsection: scheme of G-reductions}. Equip $H(\mathcal{F}^{\bullet}, V^{\bullet})$ with the corresponding diagonal action of $G$ and form the relatively affine GIT quotient $H(\mathcal{F}^{\bullet}, V^{\bullet})/ \! / \! \, G$ by taking the ring of $G$-invariants. Let $U$ be a big open subscheme such that $\mathcal{F}^i|_{U}$ is locally-free for all $i$. The same construction as in Subsection \ref{subsection: scheme of G-reductions} applies to yield a divisor $\delta^i_{U} -1$ on $H(\mathcal{F}^i, V^i)|_{U}$. This divisor is stable under the right action of $G$. We descend these divisors to get a closed subscheme $\prod_{i=1}^{b} \left(H(\mathcal{F}^i, V^i)|_{U} \right)_{\alpha^i_{U} = 0}\hookrightarrow \left(H(\mathcal{F}^{\bullet}, V^{\bullet})|_{U} \right)/ \! / \! \, G$. Define the scheme of $G$-reductions $\Red_{G}(\mathcal{F}^{\bullet})$ to be the scheme-theoretic closure of this subset in $H(\mathcal{F}^{\bullet}, V^{\bullet})/ \! / \! \, G$.
\begin{defn}
\label{defn:principalsheaf2}
A $\rho$-sheaf on $X_{S}$ is a pair $\left( \mathcal{F}^{\bullet}, \sigma \right)$, where $\mathcal{F}^{\bullet} \in \prod_{i=1}^b \Coh_{r_i}^{tf}(X)\, (S)$ and $\sigma$ is a section $\sigma: X_{S} \rightarrow \Red_{G}(\mathcal{F}^{\bullet})$.
\end{defn}
The tuple $\mathcal{F}^{\bullet}|_{U}$ of locally-free sheaves defines a $\GL(V^{\bullet})$-bundle $\mathcal{P}$ on $U$. The same argument as in Subsection \ref{subsection: scheme of G-reductions} shows that the data of a section $\sigma: X_{S} \rightarrow \Red_{G}(\mathcal{F}^{\bullet})$ is equivalent to a $G$-reduction of structure group for the $\GL(V^{\bullet})$-bundle $\mathcal{P}$. In particular, we obtain a rational $G$-bundle $\mathcal{G}$ defined on $U$. This allows to define the notion of degree $\text{deg}(\mathcal{F}^{\bullet}, \sigma)$ of a $\rho$-sheaf as in Definition \ref{defn: degree}.

We denote by $\Bun_{\rho}(X)$ the stack that takes an affine $k$-scheme $T$ to the groupoid of $\rho$-sheaves on $X_{T}$. There is a natural forgetful morphism $\Forget: \Bun_{\rho}(X) \rightarrow \prod_{i=1}^b \Coh_{r_i}^{tf}(X)$ defined by $\Forget(\mathcal{F}^{\bullet}, \sigma) = \mathcal{F}^{\bullet}$, which is affine and of finite type by the same reasoning as in Proposition \ref{prop: forgetful morphism is affine}. Therefore, $\Bun_{\rho}(X)$ is an algebraic stack with affine diagonal and locally of finite type over $k$. All of the discussion for filtrations in Subsection \ref{subsection: filtrations} applies in this more general context, the only difference being that the underlying associated filtration is now a tuple of filtrations $(\mathcal{F}^{\bullet}_m)_{m \in \mathbb{Z}}$, where each $(\mathcal{F}^i_m)_{m \in \mathbb{Z}}$ is a filtration of the torsion-free sheaf $\mathcal{F}^i$.

We define a family of line bundles $L_n$ on $\Bun_{\rho}(X)$ as follows. For each $i$, we have a family of line bundles $L_n^i$ on $\Coh^{tf}_{r_i}(X)$ defined by the same formulas as in Definition \ref{defn: line bundles}. Then, we define the box product $L_n := \boxtimes_{i=1}^b L^i_n$ on the stack $\prod_{i=1}^b \Coh^{tf}_{r_i}(X)$, and use the morphism $\Forget$ to pullback $L_n$ to a line bundle on $\Bun_{\rho}(X)$, also denoted by $L_n$. 

More explicitly, for every morphism $f: T \rightarrow \Bun_{\rho}(X)$ represented by a $\rho$-sheaf $(\mathcal{F}^{\bullet}, \sigma)$ on $X_{T}$, we have
\[ f^{*}\, L_n  = \left(\bigotimes_{i=1}^b f^* \, M^i_n \right) \otimes \left(\bigotimes_{i=1}^b (f^*\, b^i_d)^{-\otimes \, \overline{p}_{\mathcal{F}^i}(n)} \right) \]
where $f^{*} M^i_n \vcentcolon = \det \, \left(R\pi_{T \, *}\left(\mathcal{F}^i(n) \right) \right)$
and $b^i_d \vcentcolon = \bigotimes_{j = 0}^d (M^i_{j})^{\otimes (-1)^{d+j} \binom{d}{j}}$.

We will use this family of line bundles $L_n$ to define a polynomial numerical invariant on $\Bun_{\rho}(X)$. Let $(\mathcal{F}^{\bullet}, \sigma)$ be a $\rho$-sheaf on $X_{K}$ for some field extension $K \supset k$. For every nondegenerate filtration $f$ with underlying associated filtration $(\mathcal{F}^{\bullet}_m)_{m\in \mathbb{Z}}$, we set $\nu(f)$ to be the polynomial
\[ \nu(f) = \frac{\wgt(L_n|_{0})}{\sqrt{\sum_{i=1}^b \sum_{m \in \mathbb{Z}} m^2 \cdot \rk(\mathcal{F}^i_m / \mathcal{F}^i_{m+1})}}\]
The same computation as in Proposition \ref{prop: value numerical invariant} shows that this can be rewritten in terms of the reduced Hilbert polynomials as follows:
\[ \nu(f) = \sqrt{A_d} \cdot \frac{\sum_{i=1}^b \sum_{m \in \mathbb{Z}} m \cdot  (\overline{p}_{\mathcal{F}^i_m/ \mathcal{F}^i_{m+1}} - \overline{p}_{\mathcal{F}^i}) \cdot \rk(\mathcal{F}^i_m/\mathcal{F}^i_{m+1})}{\sqrt{\sum_{i=1}^b \sum_{m \in \mathbb{Z}} m^2 \cdot \rk(\mathcal{F}^i_m / \mathcal{F}^i_{m+1})}}\]

One can define a version of the affine Grassmannian in this more general context. Let $S$ be a quasicompact $k$-scheme and let $\mathcal{F^{\bullet}}$ be a tuple of torsion-free sheaves in $\prod_{i=1}^b \Coh_{r_i}^{tf}(X)\, (S)$. Let $D \hookrightarrow X_S$ be an effective Cartier divisor that is flat over $S$ and set $Q = X_S \setminus D$. Assume that we have a section $\sigma : Q \rightarrow \Red_{G}(\mathcal{F}^{\bullet})$ defined over $Q$.

The affine Grassmannian $\Gr_{X_S, D, \mathcal{F}^{\bullet}, \sigma}$ is the functor from $\left(\Aff_S\right)^{op}$ into sets defined as follows. For each affine scheme $T$ in $\Aff_S$, $\Gr_{X_S, D, \mathcal{F}^{\bullet}, \sigma}(T)$ is the set of equivalence classes of the following data
\begin{enumerate}[(1)]
    \item A $\rho$-sheaf $(\mathcal{E}^{\bullet}, \zeta)$ on $X_T$.
    \item A morphism $\theta^{\bullet}: \mathcal{F}^{\bullet}|_{X_T} \to \mathcal{E}^{\bullet}$ of tuples such that the restriction $\theta^{\bullet}|_{Q_{T}}$ is an isomorphism. Moreover, we require that the following diagram induced by the isomorphism $\theta^{\bullet}|_{Q_T}$ commutes.
\[
\begin{tikzcd}
  Q_T \ar[r, "\zeta|_{Q_T}"] \ar[rd, "\sigma|_{Q_T}", labels=below left] & \Red_{G}(\mathcal{E}^{\bullet}|_{Q_T}) \\   & \Red_{G}(\mathcal{F}^{\bullet}|_{Q_T}) \ar[u, symbol = \xrightarrow{\; \; \sim \; \;} ]
\end{tikzcd}
\]
\end{enumerate}

There is a forgetful morphism $\Gr_{{X_S, D, \mathcal{F}^{\bullet}, \sigma}} \rightarrow \prod_{i=1}^b \Gr_{X_S, D, \mathcal{F}^i}$ that takes a point $(\mathcal{E}^{\bullet}, \zeta, \theta^{\bullet})$ to the tuple $(\mathcal{E}^i, \theta^i)_{i=1}^{b}$.
The same proof as in Proposition \ref{prop: the affine grassmannian is indproper} shows that this forgetful morphism exhibits $\Gr_{{X_S, D, \mathcal{F}^{\bullet}, \sigma}}$ as a closed ind-subscheme of $\prod_{i=1}^b \Gr_{X_S, D, \mathcal{F}^i}$.

 Let $P^{\bullet} = (P^i)_{i=1}^b$ be a tuple of rational polynomials. For each $N \geq 0$ we define the subfunctor 
 \[\Gr_{X_S, D, \mathcal{F}^{\bullet}, \sigma}^{ \leq N, P^{\bullet}} \vcentcolon = \Gr_{X_S, D, \mathcal{F}^{\bullet}, \sigma} \cap \prod_{i=1}^b \Gr_{X_S, D, \mathcal{F}^i}^{\leq N, P^i} \]
 which is a closed subscheme of $\Gr_{X_S, D, \mathcal{F}^{\bullet}, \sigma}$ and it is projective over $S$. There exists an integer $m\gg0$ such that, for all $n \geq m$, the restriction of the dual $L_n^{\vee}$ is $S$-ample on the projective stratum $\Gr_{X_S, D, \mathcal{F}^{\bullet}, \sigma}^{ \leq N, P^{\bullet}}$.
 
The proof of Proposition \ref{prop: invariant is monotone} applies without change in this more general context. The argument of HN boundedness in Proposition \ref{prop: HN boundedness} also applies, one only needs to keep track of one extra index $1 \leq i \leq b$ in all the formulas. In particular, Theorem \ref{thm: theta stratification} holds and the numerical invariant $\nu$ induces a $\Theta$-stratification on the stack $\Bun_{\rho}(X)$. All of the arguments in Subsection \ref{subsection: theta stratification} and Section \ref{section: properness} apply without change. 

The tuple of rational polynomials $P^{\bullet} = (P^i)_{i=1}^b$ defines an open substack $\Bun_{\rho}(X)^{ss, P^{\bullet}} \subset \Bun_{\rho}(X)$ parametrizing Gieseker semistable $\rho$-sheaves $(\mathcal{F}^{\bullet}, \sigma)$ such that the Hilbert polynomial of $\mathcal{F}^i$ is $P^i$. If we fix the degree $\vartheta \in \Lambda(G)$ of the $\rho$-sheaf, we obtain an open and closed substack $\Bun_{\rho}(X)^{ss, P^{\bullet}}_{\vartheta} \subset \Bun_{\rho}(X)^{ss, P^{\bullet}}$, which is quasicompact by the same proof as in Proposition \ref{prop: boundedness of the semistable locus}. It follows by \cite[Thm. 2.26]{torsion-freepaper} that $\Bun_{\rho}(X)^{ss, P^{\bullet}}_{\vartheta}$ admits a proper good moduli space $M_{\rho}(X)^{ss, P^{\bullet}}_{\vartheta}$. 

\begin{thm} \label{thm: main theorem products gen linear groups}
The polynomial numerical invariant $\nu$ induces a $\Theta$-stratification on the stack $\Bun_{\rho}(X)$. In particular the leading term HN filtration induces a stratification of $\Bun_{\rho}(X)$ by locally closed susbtacks.

For each choice of degree $\vartheta \in \Lambda(G)$ and tuple of Hilbert polynomials $P^{\bullet}$, the open substack $\Bun_{\rho}(X)^{ss, P^{\bullet}}_{\vartheta}$ admits a proper good moduli space $M_{\rho}(X)^{ss, P^{\bullet}}_{\vartheta}$. \qed
\end{thm}

If the representation $\rho$ is central, then the degree $\vartheta$ is determined by the tuple of Hilbert polynomials $P^{\bullet}$ and each $\Bun_{\rho}(X)^{ss,P^{\bullet}}$ admits a proper good moduli space $M_{\rho}(X)^{ss, P^{\bullet}}$. In this case, the stack $\Bun_{\rho}(X)^{ss}$ of Gieseker semistable $\rho$-sheaves contains the stack $\Bun_{G}(X)^s$ of slope stable $G$-bundles as an open substack, by Corollary \ref{coroll: semistability comparison central reps}. Therefore the separated moduli space $M_{\rho}(X)^{ss} = \bigsqcup_{P^{\bullet}\in (\mathbb{Q}[n])^b} M_{\rho}(X)^{ss, P^{\bullet}}$ contains an open subspace $M_{G}(X)^s$ which is a good moduli space for the stack $\Bun_{G}(X)^s$. Hence $M_{\rho}(X)^{ss}$ is a locally proper compactification of the moduli of slope stable $G$-bundles on $X$ (cf.  \cite[\S7]{biswas-gomez-restriction}). 
\end{subsection}

\begin{subsection}{Comparison with previous work} \label{subsection: comparison previous work}

The problem of finding a compactification of the moduli space by
introducing singular degenerations of principal bundles has already been
considered in the literature
(\cite{schmitt.singular, gomezsols.tensors, gomezsols.principalsheaves,
glss.singular.char, glss.large})\footnote{We warn the reader that, in \cite{schmitt.singular, glss.singular.char, glss.large},
a principal $\rho$-sheaf was defined using the
dual representation.}. 
All notions fit the same framework:
we fix a representation $\rho$ of $G$ (in \cite{gomezsols.principalsheaves,glss.large} $\rho$ is not required to
be injective) and the objects are triples
consisting of a principal $G$-bundle $P$ on a big open subset $U\subset X$, a
torsion free sheaf $\mathcal{E}$ on $X$, 
and an isomorphism $\psi$ on $U$ between the vector bundle associated to $P$
using the representation $\rho$ and the restriction $\mathcal{E}|_U$.
In this subsection we will see that our definition of Gieseker
semistability for principal $\rho$-sheaves coincides with the
definitions that have already appeared.

By \Cref{prop: relation scheme and G-reductions}, the following is equivalent to \Cref{defn:principalsheaf2}
\begin{defn}
\label{defn:principalsheaf3}
Let $G$ be a connected reductive group.
Let
$$
\rho=(\rho_1,\ldots,\rho_b):G \longrightarrow \GL(V^1)\times \cdots\times \GL(V^b)
$$
be an injective homomorphism.
Let $S$ be a $k$-scheme.
A principal $\rho$-sheaf on $X_S$ is a triple
$(\mathcal{G},\mathcal{F}^\bullet,\psi^\bullet)$
where $\mathcal{G}$ is a principal $G$-bundle on a big open set $U\subset X_S$,
$\mathcal{F}^\bullet=(\mathcal{F}^i)_{i=1}^{b}$ is a tuple of 
torsion free sheaves on $X_S$, and
$\psi^\bullet=(\psi^i)_{i=1}^b$ is a tuple of isomorphisms
$(\rho_i)_*\mathcal{G}\cong \mathcal{F}^i|_{U}$.
\end{defn}

We generalize Definition \ref{defn: Gieseker semistable} in the following way, using the summation by parts argument as in Equation \eqref{eqn: ss1}.
\begin{defn}\label{defn: Gieseker stable}
Let $K\supset k$ be an algebraically closed field.
A principal $\rho$-sheaf $(\mathcal{G}, \mathcal{F}^{\bullet}, \psi^{\bullet})$ on $X_K$ is 
Gieseker semistable
if for all cocharacters $\lambda: (\mathbb{G}_m)_{K} \longrightarrow G_{K}$ and all weighted parabolic reductions $(\lambda,
\mathcal{G}_{\lambda})$ 
with associated underlying filtrations $(\mathcal{F}^i_{m})_{m \in
  \mathbb{Z}}$ of $\mathcal{F}^i$, 
(as in Definition \ref{defn: weighted parabolic reduction singular G bundles}) we have
\begin{equation}
\label{eq:semistabilitycondition}
  \sum_{i=1}^b
  \sum_{m \in \mathbb{Z}} 
  \left(\rk(\mathcal{F}^i) \cdot P_{\mathcal{F}^i_{m}} -
  \rk(\mathcal{F}^{i}_m)\cdot P_{\mathcal{F}^i}
  \right) \leq 0 
\end{equation}
\end{defn}

We start by comparing our notion of Gieseker semistability with Ramanathan's notion of slope semistability for the underlying rational $G$-bundle.
\begin{defn}[\cite{Ram79}]
\label{defn: ramanathan rational}
A rational $G$-bundle on a polarized
smooth projective variety $X$ is a $G$-bundle $\mathcal{G}$ 
on a big open set $U\subset X$.
It is called slope stable (respectively, semistable) if for all
reductions of structure group
$\mathcal{G}_P\to U'\subset U$ to a proper parabolic subgroup $P\subset G$ on 
a big open subset $U'$ and
all dominant characters $\chi$ of $P$, the associated line bundle $\chi_*(\mathcal{G}_P)$ satisfies
$$
\deg \chi_*(\mathcal{G}_P) <  0\quad \text{(respectively, $\leq$)}
$$
Here the degree is defined with respect to the fixed polarization. A rational $G$-bundle is called slope unstable if it is not slope semistable.
\end{defn}

\begin{prop} \label{prop: gieseker semistable implies slope semistable}
Let $K \supset k$ be a field extension. Let $(\mathcal{F}^{\bullet}, \sigma)$ be a $\rho$-sheaf on $X_{K}$. Suppose that $U$ is a big open subset of $X_{K}$ such that the tuple $\mathcal{F}^{\bullet}|_{U}$ is locally-free. Let $\mathcal{G}$ denote the $G$-bundle on $U$ corresponding to $(\mathcal{F}^{\bullet}, \sigma)$. If $(\mathcal{F}^{\bullet}, \sigma)$ is Gieseker semistable, then the rational $G$-bundle $\mathcal{G}$ is slope semistable in the sense of \cite{Ram79}.
\end{prop}
\begin{proof}
We freely use the notation from the proof of Proposition \ref{prop: HN boundedness} and assume, after replacing $K$ with a field extension, that $K$ is algebraically closed. Suppose that the rational $G$-bundle $\mathcal{G}$ is not slope semistable. 
Then the Harder-Narasimhan reduction \cite[Theorem 6 (1)]{anchouche-hassan-biswas} gives a proper parabolic subgroup $P_{\lambda}\subset G$ and a $P_{\lambda}$-reduction of structure group $\mathcal{G}_{P_{\lambda}}$ defined over a big open subset $Q$ contained in $U$. After replacing $U$ with the smaller subset $Q$, we can assume without loss of generality that $\mathcal{G}_{P_{\lambda}}$ is defined over $U$. 
The rational bundle $\mathcal{G}_{P_{\lambda}}$  determines a linear functional $\psi: X^*(Z_{\lambda})_{\mathbb{R}} \rightarrow \mathbb{R}$, given by $\psi(\chi):=\deg(\mathcal{G}_{P_{\lambda}}(\chi))$ (as in the proof of Lemma \ref{lemma: construction of linear functional}). Since $\psi$ actually lies in $X_*(Z_{\lambda})_{\mathbb{Q}}$, there is a multiple $N \psi$ that lies in the lattice of cocharacters $X_*(Z_{\lambda})$. By the second property of the Harder-Narasimhan reduction \cite[Theorem 6 (2)]{anchouche-hassan-biswas}, $\deg(\mathcal{G}_{P_{\lambda}}(\chi))>0$ whenever $\chi$ is a non-trivial non-negative linear combination of simple roots, therefore $N\psi$ lies in the interior of the cone of $P_{\lambda}$-dominant coweights. It follows that $P_{\lambda}= P_{(N \psi)}$. Therefore $f = (N\psi, \mathcal{G}_{P_{\lambda}})$ is a nondegenerate weighted parabolic reduction of $(\mathcal{F}, \sigma)$. 

The group $G$ acts on each vector space $V^j$ via the homomorphism $\rho_j$. For each $j$, we denote by $(\chi_{i,j})_{i=1}^{r_j}$ the tuple of weights of the split maximal torus $T$ acting on $V^j$ (with possible repetitions). Recall the function $\nu_{d-1}$ on the dominant cone $\mathcal{C}_{\lambda}$ described in Lemma \ref{lemma: first formula for nu_(d-1)} (\ref{eqn_4}). In this case the formula (\ref{eqn_4}) includes an extra summation
    \begin{equation} \label{eqn_new}
    \nu_{d-1}(\delta) = \frac{1}{\sqrt{A_{d}}(d-1)!} \, \cdot \, \frac{\sum_{j=1}^b \sum_{i=1}^{r_j} \langle \delta , \chi^{\lambda}_{i,j} \rangle \cdot (c_{i,j} - c_j) }{\sqrt{\sum_{j=1}^{b}\sum_{i=1}^{r_j} \left(\langle \delta, \chi^{\lambda}_{i,j} \rangle\right)^2}}
    \end{equation} 
    where we define $c_{i,j} = \langle \psi, \chi_{i,j}^{\lambda}\rangle$ and $c_j = \frac{1}{r_j} \langle \psi, \sum_{i=1}^{r_j} \chi_{i,j}^{\lambda}\rangle$. Expression (\ref{eqn_new}) for $\delta=N\psi$ yields
\begin{gather*} \nu_{d-1}(N\psi) = \frac{1}{\sqrt{A_{d}}(d-1)!} \, \cdot \, \frac{\sum_{j=1}^b \sum_{i=1}^{r_j} \langle N \psi , \chi^{\lambda}_{i,j} \rangle \cdot \left(\langle \psi, \chi_{i,j}^{\lambda} \rangle - \frac{1}{r_j}\sum_{l=1}^{r_j} \langle \psi, \chi_{i,l}^{\lambda}\rangle\right) }{\sqrt{\sum_{j=1}^b\sum_{i=1}^{r_j} \left(\langle N \psi, \chi^{\lambda}_{i,j} \rangle\right)^2}}= \end{gather*}
\begin{equation}\label{eqn_10}
\frac{1}{\sqrt{A_{d}}(d-1)!} \, \cdot \, \frac{\sum_{j=1}^b \left(\sum_{i=1}^{r_j} \left(\langle \psi , \chi^{\lambda}_{i,j} \rangle\right)^2 \right)-  \frac{1}{r_j}\left(\sum_{i=1}^{r_j}\langle \psi, \chi_{i,j}^{\lambda} \rangle \right) \cdot \left(\sum_{l=1}^{r_j} \langle \psi, \chi_{l,j}^{\lambda}\rangle\right) }{\sqrt{\sum_{j=1}^b\sum_{i=1}^{r_j} \left(\langle N \psi, \chi^{\lambda}_{i,j} \rangle\right)^2}}
\end{equation}

We shall show that the numerator of (\ref{eqn_10}) is positive. This will contradict Gieseker semistability of $(\mathcal{F}, \sigma)$. 
For each index $j$, let $\vec{v}_j \vcentcolon =[\langle \psi, \chi_i^{\lambda} \rangle]_{i=1}^{r_j} \in \mathbb{R}^{r_j}$ and set $\vec{1}_j \in \mathbb{R}^{r_j}$ to be the vector with all coordinates equal to $1$. The numerator in (\ref{eqn_10}) is
\begin{gather*}\sum_{j=1}^b\left(\sum_{i=1}^{r_j} \left(\langle \psi , \chi^{\lambda}_{i,j} \rangle\right)^2 \right) - \frac{1}{r_j} \left(\sum_{i=1}^{r_j}\langle \psi, \chi_{i,j}^{\lambda} \rangle \right)^2= \sum_{j=1}^b \frac{1}{r_j}\left[\left(\vec{1}_j \cdot_{\mathbb{R}} \vec{1}_j\right) \left(\vec{v}_j \cdot_{\mathbb{R}} \vec{v}_j\right) - \left(\vec{1}_j \cdot_{\mathbb{R}} \vec{v}_j \right)^2 \right]\end{gather*}
where $\cdot_{\mathbb{R}}$ denotes the standard inner product on $\mathbb{R}^{r_j}$. The Cauchy-Schwarz inequality implies that the quantity above is always nonnegative and, moreover, it is strictly positive if there exists some index $j$ such that $\vec{1}_j$ and $\vec{v}_j$ are not scalar multiples of each other. Therefore, $\nu_{d-1}(N\psi)>0$  whenever $\langle \psi, \chi_{i,j}^{\lambda}\rangle \neq \langle \psi, \chi_{l,j}^{\lambda}\rangle$ for some $j$ and  $i \neq l$. We are left to show that there are $j$ and two indexes $i \neq l$ satisfying this condition.

Let $\{\zeta_{l,j}\}$ denote the subset of $\{\chi_{i,j}\}$ consisting of $B$-highest weights for an irreducible component of the representation $\rho_j$. Each $\zeta_{l,j}$ can be uniquely expressed as a linear combination $\sum_{i=1}^h x_{z_{i}} z_i^{\vee} + \sum_{\alpha \in \Delta} x_{\alpha} \alpha$, where all coefficients $x_{\alpha}$ are nonnegative. Since the parabolic $P_{\lambda}$ is proper, we can choose some $\alpha \in I_{P_{\lambda}}$. Since $\bigoplus \rho_j$ is faithful, at least one of the $\zeta_{l,j}$ must have a nonzero coefficient $x_{\alpha} >0$; we fix such $l,j$. The restriction $\chi_{l,j}^{\lambda} \in X^*(Z_{\lambda})$ is a $P_{\lambda}$-dominant character such that the coefficient of $\alpha$ is not $0$. Since $\psi$ is in the interior of the the $P_{\lambda}$-dominant cone in $X_*(Z_{\lambda})_{\mathbb{R}}$, it follows that $\langle \psi, \chi_{l,j}^{\lambda} \rangle >0$.

Let $W$ denote the Weyl group of $G$ permutting the set of characters $\{\chi_{i,j}\}_{i=1}^{r_j}$ of the representation $\rho_j$. Let $W_{P_{\lambda}}$ denote the subgroup of $W$ generated by the reflections of the simple roots in $I_{P_{\lambda}}$. $W_{P_{\lambda}}$ acts on the characters $X^*(Z_{\lambda})$, and the action is transitive on the set of Weyl chambers in $X^*(Z_{\lambda})_{\mathbb{R}}$. The set of restrictions $\{ \chi_{i,j}^{\lambda}\}_{i=1}^{r_j} \subset X^*(Z_{\lambda})$ is preserved by $W_{P_{\lambda}}$. The character $\chi_{l,j}^{\lambda}$ is not fixed by the action of $W_{P_{\lambda}}$, because the coefficient of $\alpha$ is not $0$. Therefore, there exists some $w \in W_{P_{\lambda}}$ such that $\chi_{i,j}^{\lambda} \vcentcolon = w \cdot \chi_{l,j}^{\lambda}$ is contained in the $P_{\lambda}$-anti-dominant cone of $X^*(Z_{\lambda})_{\mathbb{R}}$. This implies that $\langle \psi, \chi_{i,j}^{\lambda}\rangle < 0$. Since $\langle \psi, \chi_{i,j}^{\lambda}\rangle < 0 < \langle \psi, \chi_{l,j}^{\lambda}\rangle$, we conclude that $\langle \psi, \chi_i^{\lambda} \rangle \neq \langle \psi, \chi_j^{\lambda} \rangle$, as desired.
\end{proof}

The following example shows that the converse is not true. 
\begin{example} \label{example: noncentral rho stable rational bundle}
We construct a principal $\rho$-sheaf which is Gieseker unstable, but the underlying rational
principal bundle is stable in the sense of Ramanathan (Definition \ref{defn: ramanathan
  rational}). Let $W$ be a vector space over $k$. Set $V = W \oplus W$
and $G = \GL(W) \times \GL(W)$. 
We define $\rho: G \rightarrow \GL(V)$ to be the faithful 
representation given by $
\rho(A,B)=
\left(
  \begin{array}{cc}
    A & 0\\
    0 & B\\
  \end{array}
\right)
$. Let $(\mathcal{E},\mathcal{E}')$ be a pair of slope stable vector
bundles of rank $\dim W$ on $X$ and 
with $d=\deg \mathcal{E} > \deg \mathcal{E}' =d'$.
This pair defines a principal $G$-bundle $P$ on $X$. 
In particular, it defines a principal $\rho$-sheaf.  
It is  stable in the sense of Ramanathan for the identity representation of $\GL(W) \times \GL(W)$).
However it is not Gieseker semistable in the sense of Definition
\ref{defn: Gieseker stable}.
\end{example}

In order to obtain a reasonable notion of stability, it is useful to impose the condition that $\rho$ is central in the sense of \Cref{defn: central representation}. First, we note that under this additional condition we can check semistability only using cocharacters of the derived subgroup of $G$.
\begin{prop}
\label{prop: lambda to derived}
Let $K \supset k$ be an algebraically closed field, and suppose that $\rho$ is central. Then in the semistability condition of Definition \ref{defn: Gieseker stable} it is enough
to look at cocharacters $ \lambda: (\mathbb{G}_m)_{K} \longrightarrow [G_{K},G_{K}]$ whose image is in the derived subgroup $[G_{K},G_{K}]\subset G_K.$
\end{prop}

\begin{proof}
Let $Z\subset G_{K}$ be the center and $[G_{K},G_{K}]\subset G_{K}$ be
the
commutator subgroup. Consider the exact sequence
$$
1 \longrightarrow
F \longrightarrow
Z\times [G_{K},G_{K}] \longrightarrow
G_{K} \longrightarrow 1
$$
where $F=Z\cap [G_{K},G_{K}]$ is a finite group (it is the
center of the semisimple group $[G_{K},G_{K}]$). 
Let $\lambda:(\mathbb{G}_m)_K \to G_K$ be a cocharacter. Since $F$ is
finite, there is an integer $N$ such that
$\lambda^N$ lifts to $Z\times [G_K,G_K]$. 
If we replace $\lambda$ by $\lambda^N$ then the parabolic
subgroup $\mathcal{G}_{\lambda}$ associated to $\lambda$ does not change. The only change in
the associated filtration
$(\mathcal{F}^i_{m})_{m \in
  \mathbb{Z}}$ is that the indexes are multiplied by $N$, and 
the expression \eqref{eq:semistabilitycondition} does not change.
Therefore, to check (semi)stability we can assume that the
cocharacter $\lambda$ lifts to the product 
$Z\times [G_{K},G_{K}]$. 
Assume $\lambda$ lifts, and 
write $\lambda=(\lambda_1,\lambda_2)$ where
$\lambda_1$ is the lift to $Z$ and $\lambda_2$ is the lift to 
$[G_{K},G_{K}]$.

We claim that the filtrations associated to $(\lambda_1,\lambda_2)$ 
and $(1,\lambda_2)$ on each $V^i$ only differ by a constant
shift in the indexes.
Indeed, since $\rho$ is central, the image of the composition 
$\rho_i\circ\lambda_1$ lies in the center of $\GL(V^i)$, so
$(\mathbb{G}_m)_K$ acts on $V^i$, via $\rho_i\circ(\lambda_1,1)$, with a fixed
weight $m_i$.  If $v\in V^i$ is a vector where $(\mathbb{G}_m)_K$ acts
with weight $m$ via $\rho_i \circ (\lambda_1,\lambda_2)$, 
then the action, via
$\rho_i \circ (1,\lambda_2)$, has weight $m-m_i$.
Therefore, the filtrations $(\mathcal{F}^i_{m})_{m \in
  \mathbb{Z}}$ induced by $\rho_i\circ (1,\lambda_2)$ and
$\rho_i\circ(\lambda_1,\lambda_2)$ just differ by a shift by $m_i$, as
claimed. 

Finally, since the sum in
\eqref{eq:semistabilitycondition} does not change by a shift by $m_i$
on each $V^i$, we may assume that $\lambda_1$ is equal to the
constant $1$. In other words, we may assume that $\lambda$ factors
through $[G_K,G_K]$.
\end{proof}

Proposition \ref{prop: lambda to derived} shows that the following definition is consistent with our earlier notions.
\begin{defn}\label{defn: slope stable}
Suppose that $\rho$ is central. Let $K\supset k$ be an algebraically closed field.
A principal $\rho$-sheaf on $X_K$ is 
Gieseker stable (respectively, Gieseker semistable)
if for all nontrivial cocharacters $\lambda: (\mathbb{G}_m)_{K} \longrightarrow [G_{K}, G_{K}]$ and all weighted parabolic reductions $(\lambda,
\mathcal{G}_{\lambda})$ 
with associated underlying filtrations $(\mathcal{F}^i_{m})_{m \in
  \mathbb{Z}}$ of $\mathcal{F}^i$, 
(as in Definition \ref{defn: weighted parabolic reduction singular G bundles}) we have
\[
  \sum_{i=1}^b
  \sum_{m \in \mathbb{Z}} 
  \left(\rk(\mathcal{F}^i) \cdot P_{\mathcal{F}^i_{m}} -
  \rk(\mathcal{F}^{i}_m)\cdot P_{\mathcal{F}^i}
  \right) < 0  \quad \quad (\text{respectively, $\leq$})
\]
A principal $\rho$-sheaf is slope stable (respectively, semistable)
if we replace the numerical condition
with
\[ 
  \sum_{i=1}^b
  \sum_{m \in \mathbb{Z}} 
  \left(\rk(\mathcal{F}^i) \cdot \deg{\mathcal{F}^i_{m}} -
  \rk(\mathcal{F}^{i}_m) \cdot \deg{\mathcal{F}^i}
  \right) < 0 \quad \quad (\text{respectively, $\leq$}).
\]
We say that a principal  $\rho$-sheaf is Gieseker unstable (resp. slope unstable) if it is not Gieseker semistable (resp. it is not slope semistable).
\end{defn}

We immediately obtain that slope stable implies Gieseker stable, and Gieseker semistable implies slope semistable.

In \cite[\S 1.1, p. 282]{glss.large}
it is assumed that the field is algebraically
closed, $G$ is an arbitrary nonconnected reductive group, and the homomorphism $\rho$ is assumed to be central. 
Furthermore, \cite{glss.large} requires that the image of $\rho$ lies in the
special linear group $\left(\prod_{i=1}^b\GL(V^i)\right) \cap
\SL(\bigoplus_{i=1}^bV^i)$.
But this is not an essential condition. If we are given a
representation $\rho=(\rho_1,\ldots,\rho_b)$ we can add one more
summand $V^{b+1}=k$ and the representation
$$
\rho'=(\rho_1,\ldots,\rho_b,(\det(\rho_1)\cdots\det( \rho_b))^{-1})
$$
will automatically land in the special linear group. A $\rho'$-sheaf is 
$\rho$-sheaf together with an isomorphism
$\det(\mathcal{F}^1) \otimes \cdots \otimes\det(\mathcal{F}^b)\cong (\mathcal{F}^{b+1})^{\vee}$. The $\rho'$-sheaf
is (semi)stable in the sense of \cite{glss.large}
if and only if the corresponding $\rho$-sheaf is
(semi)stable. Therefore, all the comparisons that are made in \cite[\S 6]{glss.large}
also hold for our notion of semistability in the case when $\rho$ is central. The following corollary follows from \cite[Lemma 2.5.4 and Remark 2.5.5]{glss.large}.

\begin{coroll} \label{coroll: semistability comparison central reps}
If $\rho$ is central, then a principal $\rho$-sheaf
is slope (semi)stable if and only if the underlying rational principal
bundle is slope (semi)stable in the sense of Ramanathan (Definition \ref{defn: ramanathan rational}). \qed
\end{coroll}

We end this section by comparing our notion of leading term filtration with previous work in the literature. We return to our interpretation of $\rho$-sheaves as pairs $(\mathcal{F}, \sigma)$ as in \Cref{defn:principalsheaf2}.
\begin{prop} \label{prop: comparison mixed gieseker filtration vs slope filtration}
Suppose that the representation $\rho$ is central. Let $K \supset k$ be a field extension such that $G_{K}$ is split. Let $(\mathcal{F}^{\bullet}, \sigma)$ be a $\rho$-sheaf on $X_{K}$ and choose a big open subset $U \subset X_{K}$ such that $\mathcal{F}|_{U}$ is locally-free. Let $\mathcal{G}$ denote the rational $G$-bundle on $U$ corresponding to $(\mathcal{F}^{\bullet}, \sigma)$ and suppose that $\mathcal{G}$ is (slope) unstable in the sense of Ramanathan. If  $f= (\lambda, \mathcal{G}_{\lambda})$ is the leading term HN weighted parabolic reduction of $(\mathcal{F}^{\bullet}, \sigma)$, then $\mathcal{G}_{\lambda}$ is the canonical (slope) parabolic reduction of $\mathcal{G}$ as defined in \cite{anchouche-hassan-biswas}.
\end{prop}
\begin{proof}
We will use the setup and notation from Propositions \ref{prop: HN boundedness} and \ref{prop: gieseker semistable implies slope semistable}. The group $G$ acts on each vector space $V^j$ via the homomorphism $\rho_j$. For each $j$, we denote by $(\chi_{i,j})_{i=1}^{r_j}$ the tuple of weights of the split maximal torus $T$ acting on $V^j$ (with possible repetitions). Since the representation $\rho$ is faithful, the collection of all the weights $(\chi_{i,j})$ determines a positive definite inner product $\cdot_{\chi}$ on $X_*(T)_{\mathbb{R}}$ given by
    \[ \delta \cdot_{\chi} \gamma \vcentcolon = \sum_{j=1}^b \sum_{i=1}^{r_j} \langle \delta, \chi_{i,j} \rangle \cdot \langle \gamma, \chi_{i,j} \rangle \]
    which is invariant under the action of the Weyl group $W$. We can restrict this inner product to $X_{*}(Z_{\lambda})_{\mathbb{R}}$ in order to obtain a $W_{P_{\lambda}}$-invariant positive definite inner product.
    
    Recall that we have
    \begin{equation} \label{eqn_new2}
    \nu_{d-1}(\delta) = \frac{1}{\sqrt{A_{d}}(d-1)!} \, \cdot \, \frac{\sum_{j=1}^b \sum_{i=1}^{r_j} \langle \delta , \chi^{\lambda}_{i,j} \rangle \cdot (c_{i,j} - c_j) }{\sqrt{\sum_{j=1}^{b}\sum_{i=1}^{r_j} \left(\langle \delta, \chi^{\lambda}_{i,j} \rangle\right)^2}}
    \end{equation} 
    where we define $c_{i,j} = \langle \psi, \chi_{i,j}^{\lambda}\rangle$ and $c_j = \frac{1}{r_j} \langle \psi, \sum_{i=1}^{r_j} \chi_{i,j}^{\lambda}\rangle$. Set $\det_j = \det(V^j) \circ \rho_j = \sum_{i=1}^{r_j} \chi_{i,j}$. Rewrite (\ref{eqn_new2}) as
    \begin{gather*} \nu_{d-1}(\delta) = \frac{1}{\sqrt{A_{d}}(d-1)!} \, \cdot \, \frac{\sum_{j=1}^b \sum_{i=1}^{r_j} \langle \delta , \chi^{\lambda}_{i,j} \rangle \cdot \langle \psi, \chi_{i,j}^{\lambda} \rangle  - \frac{1}{r_j}\langle \delta , \det_j \rangle \cdot \langle \psi, \det_j \rangle }{\sqrt{\sum_{j=1}^b\sum_{i=1}^{r_j} \left(\langle \delta, \chi^{\lambda}_{i,j} \rangle\right)^2}} =
    \end{gather*}
    \begin{equation}\label{eqn_11}
    \frac{1}{\sqrt{A_{d}}(d-1)!} \, \cdot \, \frac{\delta \cdot_{\chi} \psi  - \sum_{j=1}^b \frac{1}{r_j} \langle \delta , \det_j \rangle \cdot \langle \psi, \det_j \rangle }{\sqrt{\delta \cdot_{\chi} \delta}} 
    \end{equation}
    Recall that
    \[ X_{*}(Z_{\lambda})_{\mathbb{R}} = \bigoplus_{l=1}^h \mathbb{R} z_l \oplus \bigoplus_{n \in J_{P_{\lambda}}} \mathbb{R} \omega_n^{\vee}\]
    Choose a pair $z_l$ and $\omega_n^{\vee}$. Let $s_n$ denote the simple reflection of $X_{*}(T)_{\mathbb{R}}$ determined by the root $\alpha_n$. Since $\cdot_{\chi}$ is $W_{P_{\lambda}}$-invariant, we have
    \[ z_l \cdot_{\chi} \omega_n^{\vee} = (s_n z_l) \cdot_{\chi} (s_n \omega_n^{\vee}) = z_l \cdot_{\chi}(-\omega_n^{\vee}).\]
    This shows that $z_l \cdot_{\chi} \omega_n^{\vee} = 0$, and so the central cocharacters $X_{*}(Z)_{\mathbb{R}} = \bigoplus_{i=1}^h \mathbb{R}z_i$ are orthogonal to the fundamental coweights. Therefore, the projection $\pi_{Z}: X_{*}(Z_{\lambda})_{\mathbb{R}} \rightarrow \bigoplus_{n \in J_{P_{\lambda}}} \mathbb{R} \omega_n^{\vee}$ is an orthogonal projection with respect to $\cdot_{\chi}$.
    
    By assumption, the central torus $Z$ maps to the center of $\prod_{j=1}^b\GL(V^j)$. For every index $l$, since $\rho_j \circ z_l$ lands in the center of $\GL(V^j)$, we must have $\langle z_l, \chi_{i,j} \rangle =\langle z_l, \chi_{p,j} \rangle $ for any $i,p$. It follows that $\langle z_l, \det_j \rangle = r_j \langle z_l, \chi_{1,j} \rangle$ for the fixed index $i=1$. This implies that
    \[ \delta \cdot_{\chi} \psi  - \sum_{j=1}^b\frac{1}{r_j} \langle \delta , {\det}_{j} \rangle \cdot \langle \psi, {\det}_j \rangle = \pi_{Z}(\delta) \cdot_{\chi} \pi_{Z}(\psi)\]
    Therefore, (\ref{eqn_11}) turns out to be:
    \begin{equation}\label{eqn_12} \nu_{d-1}(\delta) = \frac{1}{\sqrt{A_{d}}(d-1)!} \, \cdot \, \frac{\pi_{Z}(\delta) \cdot_{\chi} \pi_{Z}(\psi)}{\sqrt{\delta \cdot_{\chi} \delta}}
    \end{equation}

By Corollary \ref{coroll: semistability comparison central reps}, the $\rho$-sheaf $(\mathcal{F}, \sigma)$ is unstable. Using the proof of Proposition \ref{prop: lambda to derived} we can conclude that the image of $\lambda$ is not contained in the center of $G$. Therefore the parabolic $P_{\lambda}$ is a proper subgroup of $G$. We denote by $f = (\lambda, \mathcal{G}_{\lambda})$ the leading term HN weighted parabolic reduction of $(\mathcal{F}, \sigma)$. To prove that $\mathcal{G}_{\lambda}$ is the canonical parabolic reduction of the rational $G$-bundle $\mathcal{G}$, we need to show the following (cf. \cite[Thm. 6]{anchouche-hassan-biswas}):
\begin{enumerate}[(i)]
    \item The linear functional $\psi \in X_*(Z_{\lambda})_{\mathbb{R}}$ lies in the interior $\Int(\mathcal{C}_{\lambda})$ of the dominant cone $\mathcal{C}_{\lambda}$.
    \item The rational $L_{\lambda}$-bundle $\mathcal{G}_{L_{\lambda}}$ is slope semistable.
\end{enumerate}
We proceed with the proof of these statements.

(i) Since $f = (\lambda, \mathcal{G}_{\lambda})$ maximizes $\nu(f)$ in the interior of $\mathcal{C}_{\lambda}$, we know that $\lambda$ is a critical point of $\nu_{d-1}$. It follows from Cauchy-Schwarz that this critical point is a scalar multiple of $\pi_{Z}(\psi)$. We must have $\lambda = Q \cdot \pi_{Z}(\psi)$ for some constant $Q>0$, because $\nu(f)>0$ and then $\nu_{d-1}(\lambda) \geq 0$. Since $\lambda \in \Int(\mathcal{C}_{\lambda})$ by definition, it follows that $\pi_{Z}(\psi) \in \Int(\mathcal{C}_{\lambda})$. Therefore $\psi \in \Int(\mathcal{C}_{\lambda})$, as desired.

(ii) Suppose for the sake of contradiction that $\mathcal{G}_{L_{\lambda}}$ is not slope semistable. Let $\mathcal{G}_{\overline{P}}$ be the canonical slope parabolic reduction of $\mathcal{G}_{L_{\lambda}}$ for some proper parabolic subgroup $\overline{P} \subset L_{\lambda}$, as in \cite{anchouche-hassan-biswas}. This reduction is defined over a big open subset. After shrinking the original open subset $U$, we can assume that everything is defined over $U$.

Let $P$ denote the inverse image of $\overline{P}$ under the canonical quotient map $P_{\lambda} \twoheadrightarrow L_{\lambda}$. Note that $P$ is a parabolic subgroup of $G$ (cf. second paragraph of the proof of \cite[Prop. 3.1]{biswas-holla-hnreduction}). Let $\iota: P \hookrightarrow P_{\lambda}$ denote the inclusion and let $q: P \twoheadrightarrow \overline{P}$ denote the quotient map. By \cite[Lem. 2.1]{biswas-holla-hnreduction}, there exists a unique $P$-parabolic reduction $\mathcal{G}_{P}$ of $\mathcal{G}$ such that $\iota_*(\mathcal{G}_{P}) = \mathcal{G}_{\lambda}$ and $q_*(\mathcal{G}_{P}) = \mathcal{G}_{\overline{P}}$. Let $L$ be the unique Levi subgroup of $P$ containing the maximal torus $T \subset P_{\lambda}$, and let $Z_{L}$ denote the maximal central torus of $L$. The canonical inclusion $Z_{\lambda} \subset Z_{L}$ induces an inclusion $X_{*}(Z_{\lambda})_{\mathbb{R}} \subset X_{*}(Z_{L})_{\mathbb{R}}$.

Since $\lambda = Q \cdot \pi_{Z}(\psi)$ for $Q >0$, (\ref{eqn_12}) is given by \[\nu_{d-1}(\lambda)=\frac{1}{\sqrt{A_{d}}(d-1)!} \cdot\sqrt{\pi_{Z}(\psi)\cdot_{\chi} \pi_{Z}(\psi)}.\]

On the other hand, let
\[ \psi_{L}=z + \sum_{n \in J_{P}} a_n \omega_n^{\vee} \in X_{*}(Z_{L})_{\mathbb{R}} =  X_{*}(Z)_{\mathbb{R}} \oplus \bigoplus_{n \in J_{P}} \mathbb{R} \omega_n^{\vee}\]
be the linear functional corresponding to the $L$-bundle $\mathcal{G}_{L}$ defined as in the proof of Lemma \ref{lemma: construction of linear functional}. Here $z \in X_{*}(Z)_{\mathbb{R}}$ and $a_n \in \mathbb{R}$.

By construction, $\mathcal{G}_{L_{\lambda}}$ is isomorphic to the $L_{\lambda}$-bundle obtained from $\mathcal{G}_{L}$ by extension of structure group via the inclusion $L \subset L_{\lambda}$. This implies that $\psi = z + \sum_{n \in J_{P_{\lambda}}} a_n \omega_n^{\vee}$, hence \begin{equation}\label{eqn_13}
\psi_{L} = \psi + \sum_{n \in J_{P} \setminus J_{P_{\lambda}}} a_n \omega_n^{\vee}
\end{equation}
In particular we have $a_n >0$ for all $n \in J_{P_{\lambda}}$, because $\psi \in \Int(\mathcal{C}_{\lambda})$.

We know that the bundle $q_*(\mathcal{G}_{P})$ is the canonical slope parabolic reduction of $\mathcal{G}_{L_{\lambda}}$. This implies that $a_n>0$ for all the other indexes $n \in J_{P} \setminus J_{P_{\lambda}}$. We conclude that $\psi_{L}$ is in the interior of the $P$-dominant cone $\mathcal{C}_{P}$. This shows that there exists some positive integer $N >0$ such that the tuple $f' = (N \cdot  \pi_{Z}(\psi_{L}),\mathcal{G}_{P})$ is a well-defined weighted parabolic reduction of $(\mathcal{F}^{\bullet}, \sigma)$. The version of (\ref{eqn_12}) for this parabolic reduction yields 
\[\nu_{d-1}(f')=\frac{1}{\sqrt{A_{d}}(d-1)!} \cdot\sqrt{\pi_{Z}(\psi_{L})\cdot_{\chi} \pi_{Z}(\psi_{L})}.\]

Equation (\ref{eqn_13})  implies
$\pi_{Z}(\psi_{L}) = \pi_{Z}(\psi) + \sum_{n \in J_{P} \setminus J_{P_{\lambda}}} a_n \omega_n^{\vee}$, therefore:
\begin{gather}
    \pi_{Z}(\psi_{L})\cdot_{\chi} \pi_{Z}(\psi_{L}) = \pi_{Z}(\psi) \cdot_{\chi} \pi_{Z}(\psi) + 2 \pi_{Z}(\psi) \cdot_{\chi}\left( \sum_{n \in J_{P} \setminus J_{P_{\lambda}}} a_n \omega_n^{\vee} \right) + \nonumber\\
    \left(\sum_{n \in J_{P} \setminus J_{P_{\lambda}}} a_n \omega_n^{\vee}\right) \cdot_{\chi} \left(\sum_{n \in J_{P} \setminus J_{P_{\lambda}}} a_n \omega_n^{\vee} \right)\label{eqn_14}
\end{gather}
Note that $\cdot_{\chi}$ is a Weyl group invariant positive definite inner product on $X_{*}(T)_{\mathbb{R}}$. In particular we must have $\omega_l^{\vee} \cdot_{\chi} \omega_n^{\vee} \geq 0$ for all $l,n$. Since $a_n>0$ for all $n \in J_{P}$, it follows that
\begin{gather*} \pi_{Z}(\psi) \cdot_{\chi}\left( \sum_{n \in J_{P} \setminus J_{P_{\lambda}}} a_n \omega_n^{\vee} \right) = \left( \sum_{l \in J_{P_{\lambda}}} a_l \omega_l^{\vee} \right) \cdot_{\chi}\left( \sum_{n \in  J_{P} \setminus J_{P_{\lambda}}} a_n \omega_n^{\vee} \right) \geq 0
\end{gather*}
On the other hand, since $\cdot_{\chi}$ is positive definite we have
\[ \left(\sum_{n \in J_{P} \setminus J_{P_{\lambda}}} a_n \omega_n^{\vee}\right) \cdot_{\chi} \left(\sum_{n \in J_{P} \setminus J_{P_{\lambda}}} a_n \omega_n^{\vee} \right) >0.  \]
Plugging the last two inequalities in (\ref{eqn_14}) we get $\pi_{Z}(\psi_{L})\cdot_{\chi} \pi_{Z}(\psi_{L}) >  \pi_{Z}(\psi)\cdot_{\chi} \pi_{Z}(\psi)$. Therefore $\nu_{d-1}(f')>\nu_{d-1}(\lambda) = \nu_{d-1}(f)$.
Hence, the leading coefficient of $\nu(f')$ is strictly bigger than the leading coefficient of $\nu(f)$, thus contradicting the maximality of $\nu(f)$.
\end{proof}

\begin{remark} \label{rem: comparison with nitsure slope parabolic stratifications}
When the representation $\rho$ is central, Proposition \ref{prop: comparison mixed gieseker filtration vs slope filtration} shows that the leading term HN $\Theta$-stratification of $\Bun_{\rho}(X)$ is a refinement of the slope stratification for the stack of $G$-bundles defined in \cite{nitsuregurjar2, gurjar2020hardernarasimhan} in the split case.
\end{remark}

\end{subsection}

\end{section}

\begin{section}{Gieseker-Harder-Narasimhan filtrations for $\rho$-sheaves} \label{section: GHN filtrations}
In this section we describe a procedure that iterates the leading term HN filtration in order to construct a finer filtration for $\rho$-sheaves, called the Gieseker-Harder-Narasimhan filtration (GHN filtration). We prove that this filtration induces a stratification of $\Bun_{\rho}(X)$ by locally closed substacks. To understand the notation and constructions in this section, it might help the reader to look beforehand at the examples at the end of Subsection \ref{subsection: pointwise Gieseker filtration} and in Subsection \ref{subsection: example HN filtration}.

\begin{subsection}{Lexicographic filtrations}
Fix a positive integer $q$. We equip the group $\mathbb{Z}^q$ with the lexicographic order where two tuples $\vec{m} = (m_i)_{i=1}^q, \vec{n} = (n_i)_{i=1}^q \in \mathbb{Z}^q$ satisfy $\vec{m} > \vec{n}$ if and only if there exists an index $j$ such that $m_j > n_j$ and $m_i = n_i$ for all $i < j$. The lexicographic order endows $\mathbb{Z}^q$ with the structure of a totally well-ordered abelian group. For $\vec{m}=(m_1, m_2,\ldots,m_{q-1},m_q)\in \mathbb{Z}^{q}$, we define the successor by $\vec{m}+1:=(m_1, m_2,\ldots,m_{q-1}, m_q+1)$.

Let $K \supset k$ be a field extension and let $\mathcal{F}$ be a torsion-free sheaf on $X_{K}$.
\begin{defn}\label{defn: lexicographic filtration}
A lexicographic $\mathbb{Z}^q$-filtration of $\mathcal{F}$ is a sequence of subsheaves $f = (\mathcal{F}_{\vec{m}})_{\vec{m} \in \mathbb{Z}^q}$ indexed by $\mathbb{Z}^q$ such that the following are satisfied.
\begin{enumerate}[(1)]
    \item $\mathcal{F}_{\vec{m}} \subset \mathcal{F}_{\vec{n}}$ whenever $\vec{m} > \vec{n}$.
    \item There exists an $M_{max} \in \mathbb{Z}^q$ such that $\mathcal{F}_{\vec{m}} = 0$ for all $\vec{m} > M_{max}$, and there exists an $M_{min} \in \mathbb{Z}^q$ such that $\mathcal{F}_{\vec{m}} = \mathcal{F}$ for all $\vec{m} < M_{min}$.
    \item For all $\vec{m}$, we have that $\mathcal{F}_{\vec{m}}/\mathcal{F}_{\vec{m}+1}$ is torsion-free.
\end{enumerate}
\end{defn}

The associated graded of a lexicographic $\mathbb{Z}^q$-filtration $f = (\mathcal{F}_{\vec{m}})_{\vec{m}\in \mathbb{Z}^{q}}$ is defined as $\text{gr}(f) \vcentcolon =\bigoplus_{\vec{m}\in\mathbb{Z}^{q}} \mathcal{F}_{\vec{m}}/\mathcal{F}_{\vec{m}+1}$. Conditions (2) and (3) in Definition \ref{defn: lexicographic filtration} ensure that $\text{gr}(f)$ is a torsion-free sheaf with the same rank as $\mathcal{F}$. Note that $\text{gr}(f)$ is naturally equipped with a $\mathbb{Z}^q$-grading. We say that $\vec{m} \in \mathbb{Z}^q$ is a jumping point of $f$ if the quotient $\mathcal{F}_{\vec{m}}/\mathcal{F}_{\vec{m}+1}$ is not the zero sheaf. There are finitely many jumping points.
\begin{defn}
Let $f$ be a lexicographic $\mathbb{Z}^q$-filtration. Let $\vec{p} = (p_i)_{i=1}^q$ be a tuple of positive rational numbers. For all coordinates $i$, we can scale the weights of $f$ at the $i^{th}$ coordinate by the rational number $p_i$. Doing this for all $i$ yields a lexicographic multi-weighted filtration denoted by $\vec{p} \cdot f$ that could possibly have rational weights. If $\vec{p}\cdot f$ has integer weights, then we say that the lexicographic $\mathbb{Z}^q$-filtration $\vec{p}\cdot f$ is obtained by scaling the weights of $f$ by $\vec{p}$.
\end{defn}
The scheme $\mathbb{A}_{k}^q= \Spec(k[t_1, t_2, \ldots, t_q])$ admits an action of $\mathbb{G}_m^q$ over $k$, where the $i^{th}$ coordinate of $\mathbb{G}_m^q$ acts with weight $-1$ on $t_i$ and weight $0$ on the other coordinates. Set $\Theta^q _{k} \vcentcolon = \left[ \, \mathbb{A}^q_{k} / \, \mathbb{G}_m^q \right]$ and denote by $\Theta^q _{K} \vcentcolon = \left[ \, \mathbb{A}^q_{K} / \, (\mathbb{G}_m^q)_{K} \right]$ the base change. A multi-graded version of the Rees construction in Proposition \ref{prop: filtrations of torsion free sheaves} can be applied to associate to any lexicographic $\mathbb{Z}^q$-filtration a morphism of stacks $\Theta^q _{K} \rightarrow \Coh^{tf}(X)$ in the following way. 
\begin{defn}
Given a lexicographic $\mathbb{Z}^q$-filtration $f = (\mathcal{F}_{\vec{m}})_{\vec{m} \in \mathbb{Z}^q}$ of $\mathcal{F}$, we set multi-graded Rees construction $\widetilde{\mathcal{F}}_f$ to be the $\mathbb{G}_m^q$-equivariant $\mathbb{A}^q_{K}$-flat sheaf on $\mathbb{A}^q_{K} \times X$ defined by
\[\widetilde{\mathcal{F}}_f = \bigoplus_{\vec{m} \in \mathbb{Z}^q} \left(\mathcal{F}_{\vec{m}} \cdot \prod_{i=1}^q t_i^{m_i}\right)\]
\end{defn}
Note that this is a coherent sheaf by condition (2) in Definition \ref{defn: lexicographic filtration}. If $\vec{1}_{K}$ denotes the point $(1,1,\ldots ,1) \in \mathbb{A}^q_{K}$,  the fiber at $\vec{1}_{K} \times X$ is isomorphic to the original sheaf $\mathcal{F}$ and the fiber at $0_K \times X$ is the $\mathbb{Z}^q$-graded sheaf $\text{gr}(f)$. The locus of points in $\mathbb{A}^q_{K}$ such that the fiber of $\widetilde{\mathcal{F}}_f$ is torsion-free is open by the slicing criterion for flatness \cite[\href{https://stacks.math.columbia.edu/tag/046Y}{Tag 046Y}]{stacks-project}, and the fiber at the origin $0_{K}$ is torsion-free by condition (3) in Definition \ref{defn: lexicographic filtration}. By $\mathbb{G}^q_m$-equivariance, it follows that $\widetilde{\mathcal{F}}_f$ is a family of torsion-free sheaves. So indeed $\widetilde{\mathcal{F}}_f$ represents a morphism $\Theta^q _{K} \rightarrow \Coh^{tf}(X)$.

More generally, a lexicographic $\mathbb{Z}^q$-filtration $f$ of a tuple $\mathcal{F}^{\bullet} = (\mathcal{F}_i)_{i=1}^b$ of torsion-free sheaves on $X_K$ is defined to be a tuple $f = (f_i)_{i=1}^b$ where $f_i$ is a lexicographic $\mathbb{Z}^q$-filtration of $\mathcal{F}^{i}$. Here it also makes sense to scale the weights by a tuple $\vec{p}= (p_j)_{j=1}^q$ of rational numbers; we just scale the weights of each $f_i$ by $\vec{p}$. If $\mathcal{F}^{\bullet}$ is a point in $\prod_{i=1}^b \Coh^{tf}_{r_i}(X)$, then the multi-graded Rees construction applied to each component yields a morphism $\Theta^q_{K} \rightarrow \prod_{i=1}^b \Coh^{tf}_{r_i}(X)$ represented by a tuple $\widetilde{\mathcal{F}^{\bullet}}_f$ such that the fiber at $\vec{1}_{K}\times X$ recovers $\mathcal{F}^{\bullet}$.

Fix a tuple $V^{\bullet}$ of $k$-vector spaces and let $\rho: G \rightarrow \GL(V^{\bullet})$ be faithful representation, as in Section \ref{section: generalization}. Let $(\mathcal{F}^{\bullet}, \sigma)$ be a $\rho$-sheaf on $X_{K}$ and suppose that we are given a lexicographic $\mathbb{Z}^q$-filtration $f$ of $ \mathcal{F}^{\bullet}$. Applying the multi-graded Rees construction, we obtain a homomorphism $\Theta^q_{K} \rightarrow \prod_{i=1}^b \Coh^{tf}_{r_i}(X)$ represented by a tuple $\widetilde{\mathcal{F}^{\bullet}}_f$. If $U$ is a big open subset of $X_{K}$ such that $\text{gr}(f)$ is locally-free, the restriction $\widetilde{\mathcal{F}^{\bullet}}_f|_{\Theta^q_{K} \times U}$ is a tuple of locally-free sheaves, which we can view as a $\GL(V^{\bullet})$-bundle on $\Theta^q_{K} \times_{K} U$. The section $\sigma$ induces a $G$-reduction of structure group of the restriction $\widetilde{\mathcal{F}^{\bullet}}_f|_{\vec{1}_{K} \times_{K} U}$. 

\begin{defn}
We say that $f$ is a lexicographic $\mathbb{Z}^q$-filtration of $(\mathcal{F}^{\bullet}, \sigma)$ if the $G$-reduction of $\widetilde{\mathcal{F}^{\bullet}}_f|_{\vec{1}_{K} \times_{K} U}$ induced by $\sigma$ extends to a $G$-reduction defined over $\Theta^q_{K} \times U$.
\end{defn} 
Since the $(\mathbb{G}_m^q)_{K}$-orbit of $\vec{1}_{K}$ is scheme-theoretically dense in $\mathbb{A}^q_{K}$, if such extension exists then it must be unique. So, again, being a filtration of $(\mathcal{F}^{\bullet}, \sigma)$ is a property, not a structure.
\end{subsection}

\begin{subsection}{Multi-weighted parabolic reductions}
For this subsection we fix a field extension $K \supset k$. 

Let $H$ be a reductive group over $K$ and let $\vec{\lambda} = (\lambda_i)_{i=1}^q$ be a tuple of $q$ cocharacters of $H$. The first cocharacter $\lambda_1$ defines a parabolic subgroup $P_{\lambda_1}$ with corresponding Levi subgroup $L_{\lambda_1}$. We denote by $q_{\lambda_1}: P_{\lambda_1} \rightarrow L_{\lambda_1}$ the quotient morphism. Suppose that the cocharacter $\lambda_2: \mathbb{G}_m \rightarrow H$ factors through the subgroup $L_{\lambda_1}$ and let $\overline{\lambda_2}: \mathbb{G}_m \rightarrow L_{\lambda_1}$ be the corresponding cocharacter. This induces a parabolic subgroup $P_{\overline{\lambda_2}} \subset L_{\lambda_1}$ with corresponding Levi subgroup $L_{\overline{\lambda_2}}$. The preimage $P_{(\lambda_1, \lambda_2)} := (q_{\lambda_1})^{-1}(P_{\overline{\lambda_2}})$ is a parabolic subgroup of $H$ contained in $P_{\lambda_1}$ (cf. second paragraph of the proof of \cite[Prop. 3.1]{biswas-holla-hnreduction}). We set $L_{(\lambda_1,\lambda_2)} \vcentcolon = L_{\overline{\lambda_2}}$, the joint centralizer of $\lambda_1$ and $\lambda_2$. By definition, $L_{(\lambda_1, \lambda_2)}$ is a Levi subgroup of $P_{(\lambda_1, \lambda_2)}$; we denote by $q_{(\lambda_1, \lambda_2)}: P_{(\lambda_1, \lambda_2)} \rightarrow L_{(\lambda_1, \lambda_2)}$ the quotient morphism. If the cocharacter $\lambda_3$ factors through the $L_{(\lambda_1, \lambda_2)}$, then we can similarly use $q_{(\lambda_1, \lambda_2)}$ to define a smaller parabolic subgroup $P_{(\lambda_1, \lambda_2, \lambda_3)}$ with Levi subgroup $L_{(\lambda_1, \lambda_2, \lambda_3)}$ given by the joint centralizer of the cocharacters $\lambda_1$, $\lambda_2$ and $\lambda_3$, and iterate this process. 
\begin{defn}
We say that the tuple $\vec{\lambda}$ is nested if at the $i^{th}$ step of the procedure described above we have that $\lambda_i$ factors through the subgroup $L_{(\lambda_1, \lambda_2, \ldots, \lambda_{i-1})}$.
\end{defn}
For a nested tuple $\vec{\lambda}= (\lambda_1,\ldots,\lambda_q)$, this procedure terminates at the $q^{th}$ step and yields a parabolic subgroup $P_{\vec{\lambda}}$ of $H$ with Levi subgroup $L_{\vec{\lambda}}$. Furthermore these groups fit into nested chains of subgroups:
\[  P_{\vec{\lambda}} \subset P_{(\lambda_1, \lambda_2, \ldots, \lambda_{q-1})} \subset \cdots \subset P_{(\lambda_1, \lambda_2)} \subset P_{\lambda_1} \subset H\]
\[ L_{\vec{\lambda}} \subset L_{(\lambda_1, \lambda_2, \ldots, \lambda_{q-1})} \subset \cdots \subset L_{(\lambda_1, \lambda_2)} \subset L_{\lambda_1} \subset H \]
\begin{defn}
Let $Y$ be a $K$-scheme. Let $\mathcal{H}$ be an $H$-bundle on $Y$. A $q$-weighted parabolic reduction of $\mathcal{H}$ consists of a pair $(\vec{\lambda}, \mathcal{H}_{\vec{\lambda}})$, where $\vec{\lambda}$ is a tuple of nested cocharacters of $H$ and $\mathcal{H}_{\vec{\lambda}}$ is a reduction of structure group of $\mathcal{H}$ to the subgroup $P_{\vec{\lambda}}$.
\end{defn}
Let $\theta: H \rightarrow F$ be a homomorphism of reductive groups over $K$. Let $\vec{\lambda} = (\lambda_i)_{i=1}^q$ be a nested tuple of cocharacters of $H$. We set $\theta_*(\vec{\lambda}) = (\theta \circ \lambda_i)_{i=1}^q$, which can be checked to be a nested tuple of cocharacters of $F$. The homomorphism $\theta$ restricts to a well-defined homomorphism $\theta : P_{\vec{\lambda}} \rightarrow P_{\theta_*(\vec{\lambda})}$. Let $Y$ be a $K$-scheme and let $\mathcal{H}$ be a $H$-bundle on $Y$. Given a $q$-weighted parabolic reduction $(\vec{\lambda}, \mathcal{H}_{\vec{\lambda}})$ of $\mathcal{H}$, we can associate a $q$-weighted parabolic reduction $(\theta_{*}(\vec{\lambda}), \theta_{*} \mathcal{H}_{\vec{\lambda}})$ of $\theta_{*}\mathcal{H}$. Here $\theta_{*}\mathcal{H}_{\vec{\lambda}}$ is the parabolic reduction obtained from $\mathcal{H}_{\vec{\lambda}}$ via the restriction $\theta: P_{\vec{\lambda}} \rightarrow P_{\theta_*(\vec{\lambda})}$.

We now return to our original group $G_{K}$. Suppose that we are given a $G_{K}$-bundle $\mathcal{G}$ on $Y$. Given a $q$-weighted parabolic reduction $(\vec{\lambda}, \mathcal{G}_{\vec{\lambda}})$ as above, we shall construct a $G_{K}$-bundle on $\Theta^q_{K} \times_{K} Y$ by iterating the Rees construction. 

Take the last cocharacter $\lambda_{q}$ of $\vec{\lambda}$ and consider the $\mathbb{G}_m$-equivariant $P_{\vec{\lambda}}$-bundle $\widetilde{\mathcal{G}}^{1}$ on $\mathbb{A}^1_{K} \times_{K} Y$ given by the Rees construction $\Rees(\mathcal{G}_{\vec{\lambda}}, \lambda_q)$. Using $P_{\vec{\lambda}} \subset P_{(\lambda_1, \lambda_2, \ldots, \lambda_{q-1})}$, we extend the structure group of $\widetilde{\mathcal{G}}^{1}$ to the bigger parabolic $P_{(\lambda_1, \lambda_2, \ldots, \lambda_{q-1})}$. Given that the subsequent cocharacter $\lambda_{q-1}$ factors through $P_{(\lambda_1, \lambda_2, \ldots, \lambda_{q-1})}$, we apply the Rees construction to the $P_{(\lambda_1, \lambda_2, \ldots, \lambda_{q-1})}$-bundle $\widetilde{\mathcal{G}}^{1}$ on $\mathbb{A}^1_{K} \times_{K} Y$ using the cocharacter $\lambda_{q-1}$. This way we obtain the $\mathbb{G}_m^2$-equivariant $P_{(\lambda_1, \lambda_2, \ldots, \lambda_{q-1})}$-bundle $\widetilde{\mathcal{G}}^2 \vcentcolon= \Rees(\widetilde{\mathcal{G}}^1, \lambda_{q-1})$ on $\mathbb{A}^2_{K} \times_{K} Y$. We iterate this procedure: at the $i^{th}$ step we have a $P_{(\lambda_1, \lambda_2, \ldots, \lambda_{q-i})}$-bundle $\widetilde{\mathcal{G}}^{i}$ on $\mathbb{A}^{i}_{K} \times_{K} Y$ and we apply the Rees construction using the cocharacter $\lambda_{q-i}$ to get the $\mathbb{G}_m^{i+1}$-equivariant $P_{(\lambda_1, \lambda_2, \ldots, \lambda_{q-i-1})}$-bundle $\widetilde{\mathcal{G}}^{i+1} \vcentcolon= \Rees(\widetilde{\mathcal{G}}^i, \lambda_{q-i})$ on $\mathbb{A}^{i+1}_{K} \times_{K} Y$.
\begin{defn}
We denote by $\Rees(\mathcal{G}_{\vec{\lambda}}, \vec{\lambda})$ the $P_{\lambda_1}$-bundle on $\Theta^q_{K} \times_{K} Y$ obtained at the $q^{th}$ step of the procedure described above.
\end{defn}
By construction, the restriction $\Rees(\mathcal{G}_{\vec{\lambda}}, \vec{\lambda})|_{\vec{1}_{K} \times_{K} Y}$ is canonically isomorphic to the $P_{\lambda_1}$-bundle obtained via $\mathcal{G}_{\vec{\lambda}}$ using the inclusion $P_{\vec{\lambda}} \subset P_{\lambda_1}$. We can extend the structure group from $P_{\lambda_1}$ to $G_{K}$ in order to obtain a $G_{K}$-bundle $\Rees(\mathcal{G}_{\vec{\lambda}}, \vec{\lambda}) \times^{P_{\lambda_1}} G_{K}$ on $\Theta^q_{K} \times_{K} Y$. The restriction to $\vec{1}_{K} \times_{K} Y$ is canonically isomorphic to the original $G_{K}$-bundle $\mathcal{G}$.
\begin{example}
If the group is a general linear group $\GL(V)$, this process can be described more concretely as follows. Let $\mathcal{E}$ be a $\GL(V)_{K}$-bundle on $Y$. We can think of $\mathcal{E}$ as a vector bundle of constant rank. Let $(\vec{\lambda}, \mathcal{E}_{\vec{\lambda}})$ be a $q$-weighted parabolic reduction of $\mathcal{E}$. For each index $i$, we set $\mathcal{E}_{(\lambda_1, \lambda_2, \ldots, \lambda_i)}$ to be the $P_{(\lambda_1,\lambda_2, \ldots, \lambda_{i})}$-parabolic reduction of $\mathcal{E}$ obtained from $\mathcal{E}_{\vec{\lambda}}$ by extending the structure group $P_{\vec{\lambda}} \subset P_{(\lambda_1, \lambda_2, \ldots, \lambda_i)}$. The pair $\left(\lambda_i, \mathcal{E}_{(\lambda_1, \lambda_2, \ldots, \lambda_i)}\right)$ amounts to $\mathbb{Z}$-weighted filtration $(\mathcal{E}^{(i)}_m)_{m \in \mathbb{Z}}$ of $\mathcal{E}$. We use this to define a lexicographic $\mathbb{Z}^q$-filtration $f = (\mathcal{E}_{\vec{m}})_{\vec{m} \in \mathbb{Z}^q}$ of $\mathcal{E}$ by vector subbundles where for a tuple $\vec{m} = (m_i)_{i=1}^q$, we set $\mathcal{E}_{\vec{m}} = \bigcap_{i=1}^q \mathcal{E}^{(i)}_{m_i}$. Then we apply the multi-graded Rees construction to $f$ in order to obtain a vector bundle $\widetilde{\mathcal{E}}_{f}$ on $\Theta^q_{K} \times_{K} Y$. The $\GL(V)_{K}$-bundle $\widetilde{\mathcal{E}}_{f}$ is precisely $\Rees(\mathcal{G}_{\vec{\lambda}}, \vec{\lambda}) \times^{P_{\lambda_1}} \GL(V)_{K}$. A similar description holds for $\GL(V^{\bullet})$-bundles, by applying this construction to each component of a tuple $\mathcal{E}^{\bullet}$.
\end{example}

Now let $(\mathcal{F}^{\bullet}, \sigma)$ be a $\rho$-sheaf over $X_{K}$ and let $f=(\mathcal{F}^{\bullet}_{\vec{m}})_{\vec{m} \in \mathbb{Z}^q}$ be a lexicographic $\mathbb{Z}^q$-filtration of $\mathcal{F}^{\bullet}$. Let $U$ be a big open subset of $X_{K}$ where the tuple $\text{gr}(f)$ is locally-free. Then the restriction $(\widetilde{\mathcal{F}^{\bullet}}_{f})|_{\Theta^q_{K} \times_{K} U}$ is locally-free, so we can view it as a $\GL(V^{\bullet})$-bundle. Since $(\mathcal{F}^{\bullet}_{\vec{m}})_{\vec{m} \in \mathbb{Z}^q}$ is a filtration of the $\rho$-sheaf, this $\GL(V^{\bullet})$-bundle comes equipped with a canonical reduction of structure group $\widetilde{\mathcal{G}}$ to $G_{K}$. Let $\mathcal{G}$ be the $G_{K}$-bundle on $U$ corresponding to the $\rho$-sheaf $(\mathcal{F}^{\bullet}, \sigma)$ (equivalently $\mathcal{G} = \widetilde{\mathcal{G}}|_{\vec{1}_{K} \times_{K}U}$). 

\begin{defn}
We say that the lexicographic filtration $f=(\mathcal{F}^{\bullet}_{\vec{m}})_{\vec{m} \in \mathbb{Z}^q}$ comes from a $q$-weighted parabolic reduction $(\vec{\lambda}, \mathcal{G}_{\vec{\lambda}})$ of $\mathcal{G}$ if $\widetilde{\mathcal{G}} = \Rees(\mathcal{G}_{\vec{\lambda}}, \vec{\lambda}) \times^{P_{\lambda_1}} G_{K}$.
\end{defn}

More explicitly, if $j: U \hookrightarrow X_{K}$ is a big open subset where the tuple $\mathcal{F}^{\bullet}$ is locally-free, a $q$-weighted parabolic reduction $(\vec{\lambda}, \mathcal{G}_{\vec{\lambda}})$ induces a $q$-weighted parabolic reduction $(\rho_{*}(\vec{\lambda}), \rho_{*}\mathcal{G}_{\vec{\lambda}})$ of the $\GL(V^{\bullet})_{K}$-bundle $\mathcal{F}^{\bullet}|_{U}=:\mathcal{E}^{\bullet}$. We define a lexicographic $\mathbb{Z}^q$-filtration $(\mathcal{E}^{\bullet}_{\vec{m}})_{\vec{m} \in \mathbb{Z}^q}$ of $\mathcal{E}^{\bullet}$ associated to the $q$-weighted parabolic reduction $(\rho_{*}(\vec{\lambda}), \rho_{*}\mathcal{G}_{\vec{\lambda}})$ and, for each $\vec{m} \in \mathbb{Z}^q$, set $\widetilde{\mathcal{E}}^{\bullet}_{\vec{m}} \vcentcolon = j_{*} (\mathcal{E}^{\bullet}_{\vec{m}}) \cap \mathcal{F}^{\bullet}$ (the intersection being taken for each coordinate of the tuple). This defines a lexicographic $\mathbb{Z}^q$-filtration of $\mathcal{F}^{\bullet}$, and $(\mathcal{F}^{\bullet}_{\vec{m}})_{\vec{m}\in \mathbb{Z}^q} = (\widetilde{\mathcal{E}}^{\bullet}_{\vec{m}})_{\vec{m} \in \mathbb{Z}^q}$ exactly when $(\mathcal{F}^{\bullet}_{\vec{m}})_{\vec{m}\in \mathbb{Z}^q}$ comes from the $q$-weighted parabolic reduction $(\vec{\lambda}, \mathcal{G}_{\vec{\lambda}})$.
\end{subsection}

\begin{subsection}{The GHN filtration for $\rho$-sheaves} \label{subsection: pointwise Gieseker filtration}
Fix a tuple $V^{\bullet} = (V^i)_{i=1}^b$ of vector spaces over $k$ and let $\rho: G \rightarrow \GL(V^{\bullet})$ be a faithful representation. In this subsection, we recursively apply the leading term HN filtration to the associated graded Levi sheaves in order to obtain a sequence of multi-weighted parabolic reductions of $(\mathcal{F}^{\bullet}, \sigma)$.

Let $K \supset k$ be a field extension such that $G_{K}$ is split and let $(\mathcal{F}^{\bullet}, \sigma)$ be an unstable $\rho$-sheaf on $X_{K}$. Let us proceed to describe the recursive construction of the GHN filtration of $(\mathcal{F}^{\bullet}, \sigma)$.

 \underline{First Step}: Let $f_{can, 1} :=(\mathcal{F}^{\bullet}_m)_{m \in \mathbb{Z}}$ be the leading term HN filtration of $(\mathcal{F}^{\bullet}, \sigma)$ and let $(\lambda_1, \mathcal{G}_{\lambda_1})$ be the corresponding weighted parabolic reduction. The cocharacter $\rho \circ \lambda_1$ induces a grading on each vector space $V^i$. This induces a nontrivial direct sum decomposition of the tuple $V^{\bullet}$. Indeed, if the decomposition was trivial then $\rho \circ \lambda_1$ would be a cocharacter of the center of $\GL(V^{\bullet})$ and using the same reasoning as in Proposition \ref{prop: lambda to derived} we would get $\nu(f_{can,1}) =0$, contradicting the fact that $f_{can,1}$ is destabilizing.

We can think of this direct sum decomposition as a finer tuple $V_{\lambda_1}^{\bullet}$ of direct summands, where we have increased the number of indexes. The group $\GL(V^{\bullet}_{\lambda_1})$ is the Levi subgroup of $\GL(V^{\bullet})$ determined by the cocharacter $\rho \circ \lambda_1$ which, by definition, contains the image of the Levi $L_{\lambda_1} \subset G$ under $\rho$. We denote by $\rho_{\lambda_1} : L_{\lambda_1} \rightarrow \GL(V^{\bullet}_{\lambda_1})$ the restriction to $L_{\lambda_1}$ of the homomorphism $\rho$. 

The description of the associated graded $f_{can,1}|_{0} = \left(\bigoplus_{m \in \mathbb{Z}} \mathcal{F}^{\bullet}_m / \mathcal{F}^{\bullet}_{m+1}, \text{gr}(\sigma)\right)$ given at the end of Subsection \ref{subsection: filtrations} shows that we can canonically view $f_{can,1}|_{0}$ as a $\rho_{\lambda_1}$-sheaf for the Levi $L_{\lambda_{1}}$. We call this the $1^{st}$ associated graded Levi sheaf and denote it by $\text{gr}_1(\mathcal{F}^{\bullet}, \sigma) := (\text{gr}_1(\mathcal{F}^{\bullet}), \text{gr}_1(\sigma))$. 

\underline{Second Step}: It might happen that the $\rho_{\lambda_1}$-sheaf $\text{gr}_1(\mathcal{F}^{\bullet}, \sigma)$ is unstable (see \Cref{example: torsion free sheaves}). If that is the case, we can consider the leading term HN filtration $f_{can,2}=(\text{gr}_1(\mathcal{F}^{\bullet})_m)_{m \in \mathbb{Z}}$ of the $\rho_{\lambda_1}$-sheaf $\text{gr}_1(\mathcal{F}^{\bullet}, \sigma)$. This is by definition a $\mathbb{Z}$-filtration of $\text{gr}_1(\mathcal{F}^{\bullet})$. Note that $(L_{\lambda_1})_{K}$ must also be split, and therefore there exists a weighted parabolic reduction $(\lambda_2, \mathcal{G}_{\lambda_2})$ of $\text{gr}_1(\mathcal{F}^{\bullet}, \sigma)$ that induces the leading term HN filtration $(\text{gr}_1(\mathcal{F}^{\bullet})_m)_{m \in \mathbb{Z}}$.

By construction we can think of $(\lambda_1, \lambda_2)$ as a nested pair. We set $L_{(\lambda_1, \lambda_2)}$ to be the Levi subgroup of $L_{\lambda_1}$ centralizing $\lambda_2$ and denote by $P_{(\lambda_1, \lambda_2)}$ the preimage under $q_{\lambda_1}: P_{\lambda_1} \rightarrow L_{\lambda_1}$ of $P_{\lambda_{2}}$. By \cite[Lem. 2.1]{biswas-holla-hnreduction}, the $P_{\lambda_2}$-parabolic reduction $\mathcal{G}_{\lambda_2}$ lifts to a unique $P_{(\lambda_1, \lambda_2)}$-parabolic reduction $\mathcal{G}_{(\lambda_1, \lambda_2)}$ that recovers the $P_{\lambda_1}$-reduction $\mathcal{G}_{\lambda_1}$ under the extension of groups $P_{(\lambda_1, \lambda_2)} \subset P_{\lambda_1}$.

Each piece in the $\mathbb{Z}$-filtration $(\text{gr}_1(\mathcal{F}^{\bullet}_m))_{m \in \mathbb{Z}}$ of the associated graded $\text{gr}_1(\mathcal{F}^{\bullet}) = \bigoplus_{m \in \mathbb{Z}} \mathcal{F}^{\bullet}_m / \mathcal{F}_{m+1}^{\bullet}$ can be lifted to the original $\mathcal{F}^{\bullet}$ by taking the preimages in each subquotient. This way we obtain a lexicographic $\mathbb{Z}^2$-filtration $(\mathcal{F}^{\bullet}_{(m_1, m_2)})_{(m_1, m_2) \in \mathbb{Z}^2}$ of the original tuple $\mathcal{F}^{\bullet}$. This is the lexicographic filtration coming from the 2-weighted parabolic reduction $\left((\lambda_1, \lambda_2), \mathcal{G}_{(\lambda_1, \lambda_2)}\right)$. In particular $(\mathcal{F}^{\bullet}_{(m_1, m_2)})_{(m_1, m_2) \in \mathbb{Z}^2}$ is actually a lexicographic $\mathbb{Z}^{2}$-filtration of $(\mathcal{F}^{\bullet}, \sigma)$.

The cocharacter $\lambda_2$ induces a nontrivial direct sum decomposition of the tuple $V^{\bullet}_{\lambda_1}$. We can think of this decomposition as a (strictly) finer tuple $V^{\bullet}_{(\lambda_1, \lambda_2)}$ of direct summands. The restriction of $\rho$ to the Levi subgroup $L_{(\lambda_1, \lambda_2)}$ factors through $\GL(V^{\bullet}_{(\lambda_1, \lambda_2)})$, which is a smaller Levi subgroup of $\GL(V^{\bullet})$ contained in the Levi $\GL(V^{\bullet}_{\lambda_1})$. We denote this restriction by $\rho_{(\lambda_1, \lambda_2)}: L_{(\lambda_1, \lambda_2)} \rightarrow \GL(V^{\bullet}_{(\lambda_1, \lambda_2)})$. The associated graded $f_{can,2}|_{0}$ acquires canonically the structure of a $\rho_{(\lambda_1, \lambda_2)}$-sheaf, denoted by $\text{gr}_2(\mathcal{F}^{\bullet}, \sigma) := (\text{gr}_2(\mathcal{F}^{\bullet}), \text{gr}_2(\sigma))$. We call this the $2^{nd}$ associated graded Levi sheaf. We have that $\text{gr}_2(\mathcal{F}^{\bullet})$ is the associated graded for the lexicographic $\mathbb{Z}^2$-filtration $(\mathcal{F}^{\bullet}_{(m_1, m_2)})_{(m_1, m_2) \in \mathbb{Z}^2}$ that we have defined. Moreover, its corresponding $L_{(\lambda_1, \lambda_2)}$-rational bundle is the same as the one obtained from the $P_{(\lambda_1, \lambda_2)}$-parabolic reduction $\mathcal{G}_{(\lambda_1, \lambda_2)}$ via the Levi quotient map $P_{(\lambda_1, \lambda_2)} \rightarrow L_{(\lambda_1, \lambda_2)}$. So we can view $\text{gr}_2(\mathcal{F}^{\bullet})$ as the associated graded of the 2-weighted parabolic reduction $\left((\lambda_1, \lambda_2), \mathcal{G}_{(\lambda_1, \lambda_2)}\right)$ \footnote{We entrust the reader with the precise formulation of what this means.}.

\underline{Recursion}: If $\text{gr}_2(\mathcal{F}^{\bullet}, \sigma)$ is not Gieseker semistable, then we can iterate this construction and consider the leading term HN filtration $f_{can, 3}$, and so on. At the $j^{th}$ step of this iteration we have a $j$-weighted parabolic reduction $((\lambda_1, \ldots, \lambda_j), \mathcal{G}_{(\lambda_1, \ldots, \lambda_j)})$ of $(\mathcal{F}^{\bullet}, \sigma)$. The associated underlying lexicographic $\mathbb{Z}^j$-filtration $(\mathcal{F}^{\bullet}_{\vec{m}})_{\vec{m} \in \mathbb{Z}^j}$ coming from $((\lambda_1, \ldots, \lambda_j), \mathcal{G}_{(\lambda_1, \ldots, \lambda_j}))$ is uniquely determined up to scaling the weights. Recall that $(\lambda_1, \ldots, \lambda_j)$ determines a Levi subgroup $L_{(\lambda_1, \ldots, \lambda_j)}$ that centralizes all the cocharacters in the tuple. We have an induced direct sum decomposition $V^{\bullet}_{(\lambda_1, \ldots, \lambda_j)}$, so that the representation $\rho$ restricted to $L_{(\lambda_1, \ldots, \lambda_j)}$ factors as $\rho_{(\lambda_1, \ldots, \lambda_j)} : L_{(\lambda_1, \ldots, \lambda_j)} \rightarrow \GL(V^{\bullet}_{(\lambda_1, \ldots, \lambda_j)})$. We also have a $j^{th}$ associated graded Levi $\rho_{(\lambda_1, \ldots, \lambda_j)}$-sheaf $(\text{gr}_j(\mathcal{F}^{\bullet}), \text{gr}_j(\sigma))$.

\underline{Termination}: At each step of the process, the nontrivial direct sum decomposition of $V^{\bullet}$ is finer than in the previous step. Since the vector spaces $V^i$ are finite dimensional, this process must terminate eventually. Suppose that it terminates at the $q^{th}$ step to yield a $q$-weighted parabolic reduction $\left(\vec{\lambda}, \mathcal{G}_{\vec{\lambda}}\right)$. We call this a Gieseker-Harder-Narasimhan (multi-weighted) parabolic reduction of $(\mathcal{F}^{\bullet}, \sigma)$. The corresponding $q^{th}$ associated graded Levi $\rho_{\vec{\lambda}}$-sheaf $(\text{gr}_q(\mathcal{F}^{\bullet}), \text{gr}_q(\sigma))$ is called the Gieseker-Harder-Narasimhan associated Levi of $(\mathcal{F}^{\bullet}, \sigma)$. By construction, $(\text{gr}_q(\mathcal{F}^{\bullet}), \text{gr}_q(\sigma))$ is Gieseker semistable. We call the lexicographic $\mathbb{Z}^q$-filtration $(\mathcal{F}^{\bullet}_{\vec{m}})_{\vec{m} \in \mathbb{Z}^q}$ the (multi-weighted) Gieseker-Harder-Narasimhan filtration of the $\rho$-sheaf $(\mathcal{F}^{\bullet}, \sigma)$. This filtration is uniquely determined up to scaling the weights by a rational tuple $\vec{p}$, because of the uniqueness of the leading term HN filtration up to scaling at each step of the process.

This construction is recorded in the following definition.
\begin{defn} \label{defn: GHN filtration}
Suppose that the field extension $K \supset k$ is algebraically closed and let $(\mathcal{F}^{\bullet}, \sigma)$ be a $\rho$-sheaf on $X_{K}$. The Gieseker-Harder-Narasimhan filtration of $(\mathcal{F}^{\bullet}, \sigma)$ is the unique (up to scaling of the weights) lexicographic $\mathbb{Z}^q$-filtration $(\mathcal{F}^{\bullet}_{\vec{m}})_{\vec{m} \in \mathbb{Z}^q}$ of $(\mathcal{F}^{\bullet}, \sigma)$ constructed above. Its Gieseker-Harder-Narasimhan associated Levi $\rho_{\vec{\lambda}}$-sheaf, as described above, is Gieseker semistable.

If $K \supset k$ is not algebraically closed, then by uniqueness we can descend the GHN filtration of the base change to any algebraic closure of $K$. Therefore we obtain a unique (up to scaling the weights) lexicographic $\mathbb{Z}^q$-filtration $(\mathcal{F}^{\bullet}_{\vec{m}})_{\vec{m} \in \mathbb{Z}^q}$, which we call the GHN filtration of $(\mathcal{F}^{\bullet}, \sigma)$.
\end{defn}

\begin{remark}
In particular, when the $\rho$-sheaf $(\mathcal{F}^{\bullet}, \sigma)$ is Gieseker semistable, the GHN filtration constructed above produces a trivial filtration as in Definition \ref{defn: leading term filtration} .
\end{remark}

\begin{remark}If $G_{K}$ is split then, by construction, the GHN filtration $(\mathcal{F}^{\bullet}_{\vec{m}})_{\vec{m} \in \mathbb{Z}^q}$ comes from a $q$-weighted parabolic reduction $(\vec{\lambda}, \mathcal{G}_{\vec{\lambda}})$. 
\end{remark}

We end this subsection by explicitly describing the procedure for the construction of the GHN filtration in this case of torsion-free sheaves (as in \Cref{example: torsion free sheaves}). Let $0 = \mathcal{F}^{\text{l-term}}_{0} \subset \mathcal{F}^{\text{l-term}}_{1} \subset \mathcal{F}^{\text{l-term}}_{2} \subset \cdots \subset \mathcal{F}^{\text{l-term}}_{j_1} = \mathcal{F}$
be the unweighted leading term filtration, with leading term index $i_1$. By assigning weights as above, this yields the leading term HN filtration $(\, ^{(1)}\mathcal{F}_{m})_{m \in \mathbb{Z}}$ of $\mathcal{F}$, whose associated cocharacter is $\lambda_1$. The $1^{st}$ associated graded Levi $\rho_{\lambda_1}$-sheaf consists of the tuple
\[ \text{gr}_1(\mathcal{F}) = \left( \mathcal{F}^{\text{l-term}}_{i}/ \mathcal{F}^{\text{l-term}}_{i-1}\right)_{i=1}^{j_1} \]
The leading term HN filtration of $\text{gr}_1(\mathcal{F})$ is obtained by taking the leading term filtration of each element $\mathcal{F}^{\text{l-term}}_{i}/ \mathcal{F}^{\text{l-term}}_{i-1}$ in the tuple. Each corresponding leading term index will be strictly smaller than $i_1$. The sheaves appearing in the leading term filtration of  $\mathcal{F}^{\text{l-term}}_{i}/ \mathcal{F}^{\text{l-term}}_{i-1}$ are of the form $\mathcal{F}^{\text{HN}}_{h}/\mathcal{F}^{\text{l-term}}_{i-1}$ for some index $1\leq h\leq n$. We denote the corresponding lexicographic $\mathbb{Z}^2$-filtration of $\mathcal{F}$ by $(\, ^{(2)}\mathcal{F}_{\vec{m}})_{\vec{m} \in \mathbb{Z}^2}$. Let $\vec{m}_1 > \vec{m}_2 > \cdots > \vec{m}_{j_2}$ denote the jumping points of $(\, ^{(2)}\mathcal{F}_{\vec{m}})_{\vec{m} \in \mathbb{Z}^2}$. The unweighted filtration
\[ 0 \subset \, ^{(2)}\mathcal{F}_{\vec{m}_1} \subset \,  ^{(2)}\mathcal{F}_{\vec{m}_2} \subset \cdots \subset \, ^{(2)}\mathcal{F}_{\vec{m}_{j_2}} = \mathcal{F}\]
is coarser than the Gieseker-Harder-Narasimhan filtration, but finer than the leading term filtration. The reduced Hilbert polynomials of the graded pieces are strictly increasing, and each graded piece $^{(2)}\mathcal{F}_{\vec{m}_i} / \, ^{(2)}\mathcal{F}_{\vec{m}_{i-1}}$ is semistable up to a degree $i_2$ in the Hilbert polynomial that is strictly smaller than the degree $i_1$ in the previous step. The $2^{nd}$ associated graded Levi sheaf is $\text{gr}_2(\mathcal{F})=\left(^{(2)}\mathcal{F}_{\vec{m}_i}/ \, ^{(2)}\mathcal{F}_{\vec{m}_{i-1}}\right)_{i=1}^{j_2}$. We iterate this procedure, at each step we take the leading term filtration of the graded pieces. 

Suppose that this process terminates at the $q^{th}$ step. Let $(\, ^{(q)}\mathcal{F}_{\vec{m}})_{\vec{m}\in \mathbb{Z}^q}$ be the GHN filtration obtained at this $q^{th}$ step. Denote by $\vec{m}_1> \vec{m}_2 > \cdots > \vec{m}_{j_q}$ the jumping points of the GHN filtration. Consider the unweighted filtration
\begin{equation}
\label{eqn: unweighted GHN torsion-free}
0 \subset \, ^{(q)}\mathcal{F}_{\vec{m}_1} \subset \,  ^{(q)}\mathcal{F}_{\vec{m}_2} \subset \cdots \subset \, ^{(q)}\mathcal{F}_{\vec{m}_{j_q}} =\mathcal{F} .\end{equation}
The reduced Hilbert polynomials of the graded pieces of this filtration are strictly increasing by construction. Moreover, each graded piece must be Gieseker semistable, because the $q^{th}$ step is the last one of the procedure. Hence, this filtration is the classical Gieseker-Harder-Narasimhan filtration of $\mathcal{F}$. The associated $q$-weights in the GHN filtration encode information about how high the degree of the difference of reduced Hilbert polynomials is between two consecutive graded pieces.

We have therefore shown the following.
\begin{prop}
\label{prop:GHNcomparison}
Let $\mathcal{F}$ be a torsion-free sheaf. Let $(\, ^{(q)}\mathcal{F}_{\vec{m}})_{\vec{m}\in \mathbb{Z}^q}$ be the GHN filtration. The corresponding unweighted filtration \eqref{eqn: unweighted GHN torsion-free} coincides with the classical Gieseker-Harder-Narasimhan filtration of $\mathcal{F}$. \qed
\end{prop}
\end{subsection}

\begin{subsection}{Example of Gieseker-Harder-Narasimhan filtration} \label{subsection: example HN filtration}
We compare the different filtrations that we have
considered in an example. Let $G=\SO(7)$ and let $\rho:\SO(7)\to \GL(7)$ be
the standard representation. Set $X=\mathbb{P}^3$, 
and let $L\cong \mathbb{P}^1\subset \mathbb{P}^3$ be a fixed line. Choose 
two different points $ Z=\{p,p'\}\subset \mathbb{P}^3$. We denote by $\mathcal{I}_L$ the ideal
sheaf of $L$, $\mathcal{I}_p$ the ideal sheaf of $p$, and
$\mathcal{I}_Z$ the ideal sheaf of $Z$. Set
\[
\mathcal{F}=
\mathcal{I}_L \oplus \mathcal{I}_Z \oplus \mathcal{I}_p 
\oplus 
\mathcal{O} \oplus \mathcal{O} \oplus \mathcal{O} \oplus \mathcal{O}
.
\]
Let $\varphi:\mathcal{F}\otimes \mathcal{F} \to \mathcal{O}$ be the symmetric
bilinear form defined by the matrix with $1$'s in the antidiagonal and zeros everywhere else.
Set $U=\mathbb{P}^3\setminus (L\cup Z)$. The sheaf $\mathcal{F}|_{U}$ is locally free, and in fact trivial.
Let $\varphi'$ be the canonical trivialization of $\det\mathcal{F}$. Clearly $\det\varphi=\varphi'{}^2$,
hence the pair $(\varphi,\varphi')$ induces a reduction of the trivial $\GL(7)$-bundle
to $\SO(7)$.
The rational $\SO(7)$-bundle
$\mathcal{G}\to U$
which it defines is trivial. Let $\psi$ denote the canonical isomorphism between
$\rho_*(\mathcal{G})$ and $\mathcal{F}|_U$. 

The principal $\rho$-sheaf
$(\mathcal{G},\mathcal{F},\psi)$ is slope semistable 
(Definition \ref{defn: slope stable}). The rational principal
$\SO(7)$-bundle $\mathcal{G}$ is trivial, hence also slope semistable 
(Definition \ref{defn: ramanathan rational}). Hence the
Harder-Narasimhan canonical reduction defined in 
\cite{anchouche-hassan-biswas} is trivial.

The Hilbert polynomials are
  \begin{align*}
P_{\mathcal{O}}(n)=\frac{1}{6}(n^3+6n^2+10n+3)
&\quad\quad& P_{\mathcal{I}_p}(n)=\frac{1}{6}(n^3+6n^2+10n-3)\\
P_{\mathcal{I}_Z}(n)=\frac{1}{6}(n^3+6n^2+10n-9)&\quad\quad& 
P_{\mathcal{I}_L}(n)=\frac{1}{6}(n^3+6n^2  -\; 6n-3)
  \end{align*}
Note that the three first polynomials only differ in the last term, and
the last polynomial differs in the last two terms.

The principal $\rho$-sheaf
$(\mathcal{G},\mathcal{F},\psi)$ is Gieseker unstable
(Definition \ref{defn: Gieseker stable}). 
We will now calculate the leading term HN filtration. We work in the standard basis where the bilinear form is given by the matrix above. We will use cocharacters
of $\SO(7)$ of the form
$$
\operatorname{diag}\operatorname(t^a,t^b,t^c,1,t^{-c},t^{-b},t^{-a})
$$
with $a\geq b\geq c\geq 0$. The associated
filtration $\mathcal{F}_m$ satisfies
$\mathcal{F}_m^\perp=\mathcal{F}_{-m+1}^{}$, where we recall that 
$\mathcal{F}_i{}^\perp :=\ker  \big(\mathcal{F}\stackrel{\varphi}{\longrightarrow}\mathcal{F}^\vee \longrightarrow \mathcal{F}_i{}^\vee\big ).$ The leading term HN
filtration is
\[
\overbrace{\mathcal{F}}^{\mathcal{F}_{-1}}  \quad\supset\quad
\overbrace
{\mathcal{I}_Z \oplus \mathcal{I}_p \oplus 
\mathcal{O} \oplus \mathcal{O} \oplus \mathcal{O} \oplus \mathcal{O}}
^{\mathcal{F}_0}
\quad\supset\quad
\overbrace{\mathcal{O}}^{\mathcal{F}_{1}}
\]
where $\mathcal{F}_1$ is the last summand in $\mathcal{F}$,
and $\mathcal{F}_1^\perp=\mathcal{F}_0^{}$.
By definition of leading term HN filtration, this is the unique
filtration, with $\mathcal{F}_m^\perp=\mathcal{F}_{-m+1}^{}$, which
gives a maximum of the numerical function \eqref{eqn_1}.

This filtration has been obtained by the cocharacter
$\lambda_1=\operatorname{diag} ( t,1,1,1,1,1,t^{-1})$. The associated
Levi is $L_{\lambda_1}=\mathbb{G}_m\times \SO(5)$,
the representation $\rho_{\lambda_1}$ is
$$
\rho_{\lambda_1}:\mathbb{G}_m\times \SO(5) \longrightarrow \GL(1)\times \GL(5) \times \GL(1), \; \; (t,g) \mapsto  (t,g,t^{-1})
$$
and the associated graded
sheaf is
$$
\overbrace{\mathcal{I}_L}^{\mathcal{F}^{-1}:= \mathcal{F}_{-1}/\mathcal{F}_{0}}
\quad\oplus\quad 
\overbrace
{\mathcal{I}_Z \oplus \mathcal{I}_p\oplus\mathcal{O} \oplus \mathcal{O} \oplus \mathcal{O} }
^{\mathcal{F}^0 := \mathcal{F}_{0}/\mathcal{F}_{1}}
\quad\oplus\quad 
\overbrace{\mathcal{O}}^{\mathcal{F}^{1} := \mathcal{F}_{1}}.
$$
The $\SO(5)$-sheaf defined by $\mathcal{F}^0$ is not Gieseker stable.
Its leading term HN filtration is given by the cocharacter
$\lambda_2=\operatorname{diag}(t^2,t,1,t^{-1},t^{-2})$ and it is
\begin{gather*}
\overbrace{\mathcal{F}^0}^{\mathcal{F}^0_{-2}} 
\quad\supset\quad
\overbrace{
\mathcal{I}_p\oplus \mathcal{O} \oplus \mathcal{O} \oplus \mathcal{O}}
^{\mathcal{F}^0_{-1}}
\quad\supset\quad
\overbrace{
\mathcal{O} \oplus \mathcal{O} \oplus \mathcal{O}}
^{\mathcal{F}^0_0}
\quad\supset\quad
\overbrace{\mathcal{O} \oplus \mathcal{O}}^{\mathcal{F}^0_1}
\quad\supset\quad
\overbrace{\mathcal{O}}^{\mathcal{F}^0_2}
\end{gather*}

Finally, the (multi-weighted) Gieseker-Harder-Narasimhan filtration is
\begin{gather*}
\overbrace{\mathcal{F}}^{\mathcal{F}_{(-1,0)}}
\quad\supset\quad
\overbrace{\mathcal{I}_Z\oplus \mathcal{I}_p\oplus\mathcal{O}^{\oplus 4}}
^{\mathcal{F}_{(0,-2)}}
\quad\supset\quad
\overbrace{\mathcal{I}_p\oplus\mathcal{O}^{\oplus 4}}
^{\mathcal{F}_{(0,-1)}}
\quad\supset\quad
\overbrace{\mathcal{O}^{\oplus 4}}
^{\mathcal{F}_{(0,0)}}
\quad\supset\quad
\overbrace{\mathcal{O}^{\oplus 3}}
^{\mathcal{F}_{(0,1)}}
\quad\supset\quad
\overbrace{\mathcal{O}^{\oplus 2}}
^{\mathcal{F}_{(0,2)}}
\quad\supset\quad
\overbrace{\mathcal{O}}
^{\mathcal{F}_{(1,0)}}
\end{gather*}
The process stops here because the associated Levi principal sheaf is Gieseker semistable. Indeed, the Levi is 
$L_{\lambda_2}=\mathbb{G}_m\times (\mathbb{G}_m\times \mathbb{G}_m)\subset \mathbb{G}_m\times \SO(5)$, the representation is
$$
\rho_{\lambda_2}: \mathbb{G}_m \times (\mathbb{G}_m\times
\mathbb{G}_m) \longrightarrow \GL(1)^{\times 7},  \; \; \;
 (a,b,c)  \mapsto  (a,b,c,1,c^{-1},b^{-1},a^{-1})
$$
and the associated graded is isomorphic to $\mathcal{F}$.

\end{subsection}

\begin{subsection}{The GHN filtration in families}
In this subsection we define a notion of relative GHN filtrations. We prove that relative GHN filtrations induce a stratification of $\Bun_{\rho}(X)$ by locally closed substacks.

Let $S$ be a $k$-scheme. Let $(\mathcal{F}^{\bullet}, \sigma)$ be a $\rho$-sheaf on $X_{S}$. 
\begin{defn}
A relative lexicographic $\mathbb{Z}^q$-filtration of the underlying tuple of sheaves $\mathcal{F}^{\bullet} = (\mathcal{F}_i)_{i=1}^b$ is a sequence of tuples of subsheaves $(\mathcal{F}^{\bullet}_{\vec{m}})_{\vec{m} \in \mathbb{Z}^q}$ satisfying the following conditions
\begin{enumerate}[(1)]
    \item $\mathcal{F}^{\bullet}_{\vec{m}} \subset \mathcal{F}^{\bullet}_{\vec{n}}$ whenever $\vec{m} > \vec{n}$.
    \item There exists some $M_{max} \in \mathbb{Z}^q$ such that $\mathcal{F}^{\bullet}_{\vec{m}} = 0$ for all $\vec{m} > M_{max}$. Similarly, there exists some $M_{min} \in \mathbb{Z}^q$ such that $\mathcal{F}^{\bullet}_{\vec{m}} = \mathcal{F}^{\bullet}$ for all $\vec{m} < M_{min}$.
    \item For all $\vec{m}$, we have that $\mathcal{F}^{\bullet}_{\vec{m}}/\mathcal{F}^{\bullet}_{\vec{m}+1}$ is a tuple of families of torsion-free sheaves on $X_{S}$.
\end{enumerate}
\end{defn}
For any such filtration $f= (\mathcal{F}^{\bullet}_{\vec{m}})_{\vec{m} \in \mathbb{Z}^q}$, we can apply the multi-graded Rees construction in order to obtain a tuple $\widetilde{\mathcal{F}^{\bullet}}_{f}$ of families of torsion-free sheaves on $\Theta^q \times X_{S}$. Let $U$ be a big open subset of $X_{S}$ where the associated graded $\text{gr}(f)$ is locally-free. One can use $\mathbb{G}_m^q$-equivariance to conclude that the restriction $(\widetilde{\mathcal{F}^{\bullet}}_{f})|_{\Theta^q \times U}$ is locally-free. In particular, it can be viewed as a $\GL(V^{\bullet})$-bundle. The section $\sigma$ gives a $G$-reduction of structure group of $(\widetilde{\mathcal{F}^{\bullet}}_{f})|_{\vec{1} \times U}$. 

\begin{defn}
We say that $f= (\mathcal{F}^{\bullet}_{\vec{m}})_{\vec{m} \in \mathbb{Z}^q}$ is a relative lexicographic $\mathbb{Z}^q$-filtration of the $\rho$-sheaf $(\mathcal{F}^{\bullet}, \sigma)$ if the $G$-reduction induced by $\sigma$ extends (uniquely) to $\Theta_{k}^{q} \times U$.
\end{defn}

    We can also define relative $q$-weighted reductions for $(\mathcal{F}^{\bullet}, \sigma)$ in a similar way as in the case when $S = \Spec(K)$. One can say as well when a relative lexicographic $\mathbb{Z}^q$-filtration of $(\mathcal{F}^{\bullet}, \sigma)$ comes from a $q$-weighted parabolic reduction $(\vec{\lambda}, \mathcal{G}_{\vec{\lambda}})$, we omit the details.

\begin{defn}
Let $S$ and $(\mathcal{F}^{\bullet}, \sigma)$ be as above and let $(\mathcal{F}^{\bullet}_{\vec{m}})_{\vec{m} \in \mathbb{Z}^q}$ be a lexicographic $\mathbb{Z}^q$-filtration of $(\mathcal{F}^{\bullet}, \sigma)$, where we might allow a different choice of $q$ over each connected component of $S$. We say that $(\mathcal{F}^{\bullet}_{\vec{m}})_{\vec{m} \in \mathbb{Z}^q}$ is a GHN filtration for $(\mathcal{F}^{\bullet}, \sigma)$ if for all points $s \in S$ the restriction $(\mathcal{F}^{\bullet}_{\vec{m}}|_{X_{s}})_{\vec{m} \in \mathbb{Z}^q}$ is the GHN filtration of $(\mathcal{F}^{\bullet}|_{X_{s}}, \sigma|_{X_{s}})$.
\end{defn}

\begin{defn}
Let $H\rightarrow S$ be a locally closed stratification of $S$. We say that the $S$-scheme $H$ is universal for GHN filtrations of $(\mathcal{F}^{\bullet}, \sigma)$ if it represents the functor that classifies pairs $\left(T, \left[(\mathcal{F}^{\bullet}_{\vec{m}})_{\vec{m} \in \mathbb{Z}^q}]\right]\right)$, where $T$ is a $S$-scheme and $\left[(\mathcal{F}^{\bullet}_{\vec{m}})_{\vec{m} \in \mathbb{Z}^q}\right]$ is an equivalence class of relative GHN filtrations of $(\mathcal{F}^{\bullet}|_{X_{T}}, \sigma|_{X_{T}})$, up to scaling the weights.
\end{defn}

\begin{thm} \label{thm: relative Gieseker filtrations}
The relative GHN filtrations induce a stratification of $\Bun_{\rho}(X)$ by locally closed substacks. Equivalently, for any $k$-scheme $S$ and every $\rho$-sheaf $(\mathcal{F}^{\bullet}, \sigma)$ there exists a unique locally closed stratification $H$ of $S$ that is universal for GHN filtrations of $(\mathcal{F}^{\bullet}, \sigma)$. 

If, in addition, $S$ is locally of finite type over a field extension $K \supset k$ such that $G_{K}$ is split, then the restriction of the universal relative GHN filtration\footnote{
More precisely, we choose a representative up to scaling. 
} to each connected component of the stratification comes from a relative multi-weighted parabolic reduction.
\end{thm}
\begin{proof}
Let $S$ be a $k$-scheme and let $(\mathcal{F}^{\bullet}, \sigma)$ be a $\rho$-sheaf on $X_{S}$. We will denote by $\RelGies(\mathcal{F}^{\bullet}, \sigma)$ the functor that parametrizes pairs $\left(T, [(\mathcal{F}^{\bullet}_{\vec{m}})_{\vec{m} \in \mathbb{Z}^q}]\right)$, where $T$ is a $S$-scheme and $[(\mathcal{F}^{\bullet}_{\vec{m}})_{\vec{m} \in \mathbb{Z}^q}]$ is an equivalence class of relative GHN filtrations of $(\mathcal{F}^{\bullet}|_{X_{T}}, \sigma|_{X_{T}})$ up to scaling the weights. We would like to show that $\RelGies(\mathcal{F}^{\bullet}, \sigma)$ is represented by a locally closed stratification of $S$. Since locally closed immersions satisfy \'etale descent, we can check this after base-changing to a field extension $K\supset k$ such that $G_{K}$ is split. After replacing $k$ with $K$, we can assume without loss of generality that $G$ is split. Moreover, since $\Bun_{\rho}(X)$ is locally of finite type over $k$, we can reduce to the case when $S$ is of finite type over $k$. We will construct $H$ recursively.

\underline{First Step}: We denote by $H_1$ the locally closed stratification of $S$ that is universal for leading term HN filtrations of $(\mathcal{F}^{\bullet}, \sigma)$ (Proposition \ref{prop: relative mixed Gieseker filtrations}). There is a natural forgetful functor $\RelGies(\mathcal{F}^{\bullet}, \sigma) \rightarrow H_1$ given by sending a relative GHN filtration $(\mathcal{F}^{\bullet}_{\vec{m}})_{\vec{m}\in \mathbb{Z}^q}$ to the relative leading term HN filtration $(\mathcal{F}^{\bullet}_m)_{m \in \mathbb{Z}}$ defined by
\[ \mathcal{F}^{\bullet}_m \vcentcolon = \bigcup_{(m_2, m_3, \ldots, m_q) \in \mathbb{Z}^{q-1}} \mathcal{F}_{(m,m_2, \ldots, m_q)} \]
The fact that this is a relative leading term HN filtration follows from the pointwise construction of the GHN filtration and the fact that pulling back to a fiber commutes with taking unions.

Choose a representative $f_1$ for the universal leading term HN filtration of $(\mathcal{F}^{\bullet}|_{X_{H_1}}, \sigma|_{X_{H_1}})$. By the last statement in Proposition \ref{prop: relative mixed Gieseker filtrations}, the restriction of $f_1$ to each component of $H_1$ comes from a relative weighted parabolic reduction $(\lambda_1, \mathcal{G}_{\lambda_1})$. We can use this to define, on each component of $H_1$, a relative $1^{st}$ associated graded Levi $\rho_{\lambda_1}$-sheaf, in exactly the same way as in the construction of the GHN filtration in Subsection \ref{subsection: pointwise Gieseker filtration}. Putting all of these together for all connected components, we obtain what we can call a relative $1^{st}$ associated graded Levi $\rho_{\lambda_1}$-sheaf $\text{gr}_1 \vcentcolon = \text{gr}_1(\mathcal{F}^{\bullet}|_{X_{H_1}}, \sigma|_{X_{H_1}})$ defined on $H_1$, where we allow a different choice of the cocharacter $\lambda_1$ over each connected component of $H_1$. 

\underline{Second Step}: We can apply Proposition \ref{prop: relative mixed Gieseker filtrations} to the restriction of $\text{gr}_1$ on each connected component in order to obtain a refinement $H_2$ of the stratification $H_1$. The scheme $H_2$ is universal for the leading term HN filtration of $\text{gr}_1$, hence we can choose a representative $f_2$ for the universal filtration. By the last statement of Proposition \ref{prop: relative mixed Gieseker filtrations}, the restriction of $f_2$ to each connected component of $H_2$ comes from a weighted reduction $(\lambda_2, \mathcal{G}_{\lambda_2})$ of $\text{gr}_1$. The procedure to construct the GHN filtration can be applied verbatim to this relative situation. In particular, we can construct a relative lexicographic $\mathbb{Z}^2$-filtration of $(\mathcal{F}^{\bullet}|_{X_{H_2}}, \sigma|_{X_{H_2}})$. This lexicographic filtration comes from a 2-weighted parabolic reduction $((\lambda_1, \lambda_2), \mathcal{G}_{(\lambda_1, \lambda_2)})$. By construction, $H_2$ is universal for lexicographic $\mathbb{Z}^2$-filtrations obtained in the second step of the construction of the GHN filtration. We can define a forgetful morphism $\RelGies(\mathcal{F}^{\bullet}, \sigma) \rightarrow H_2$ in the same way as in the case of the first step $H_1$. In this case we remember the first two coordinates of the $\mathbb{Z}^q$-filtration and take the union over the other coordinates. 

\underline{Recursion}: Now we can iterate as in the construction of the (pointwise) GHN filtration. This process must terminate by the same reasoning as in Subsection \ref{subsection: pointwise Gieseker filtration}, namely, the dimensions of the vector spaces in the tuple $V^{\bullet}$ are finite. If it terminates at the $q^{th}$ step, then the corresponding locally closed stratification $H_q$ is universal for GHN filtrations of $(\mathcal{F}^{\bullet}, \sigma)$. Furthermore, on each connected component of $H_q$ the restriction of (a representative up to scaling of) the universal GHN filtration comes from a relative multi-weighted parabolic reduction.
\end{proof}

\end{subsection}
\end{section}

\appendix

\begin{section}{Gieseker-Harder-Narasimhan filtrations in positive characteristic} \label{appendix: positive char}
Many of the arguments in this article do not use the hypothesis that the characteristic of $k$ is $0$. In this appendix we explain what needs to be modified in order to define Gieseker-Harder-Narasimhan filtrations in positive characteristic. We can show that $\nu$ defines a weak $\Theta$-stratification on $\text{Bun}_{\rho}(X)$ (as in \cite[Defn. 2.23]{torsion-freepaper}). This allows us to define the GHN filtration of a $\rho$-sheaf after maybe passing to a purely inseparable extension of the ground field.

Fix an arbitrary field $k$. We use the same setup and notation as in the characteristic $0$ case ($X$ is a polarized smooth projective geometrically irreducible $k$-variety, $G$ is a connected reductive group over $k$, and $\rho$ is a faithful representation of $G$). Note that everything in Subsection \ref{subsection: big open and torsion-free} applies to general $k$.

The construction of $\text{Red}_{G}(\mathcal{F})$ in Subsection \ref{subsection: scheme of G-reductions} can be done in a similar way in arbitrary characteristic. The proofs of all the results in Subsection \ref{subsection: scheme of G-reductions} remain valid, with the only exception that in Lemma \ref{lemma: reduction scheme and base-change} it is no longer clear that the natural base-change morphism $\text{Red}_{G}(\mathcal{F}) \times_{X_S} X_T \to \text{Red}_{G}(\mathcal{F}|_{X_T})$ is necessarily an isomorphism. This is because it is no longer true that formation of $G$-invariants commutes with arbitrary base-change. We note, however, that for any big open $U \subset X_{S}$ such that $\mathcal{F}$ is locally free, the restriction $\text{Red}_{G}(\mathcal{F}) \times_{X_S} U_{T} \to \text{Red}_{G}(\mathcal{F}|_{X_T})_{U_{T}}$ is an isomorphism. This is because both sides are naturally identified with the quotient $\mathcal{P}|_{U_{T}}/G$ by Proposition \ref{prop: scheme of reductions as fiber bundle}. This is used to show the following lemma.
\begin{lemma} \label{lemma: key lemma positive char}
Let $S$ be a $k$-scheme, and let $\mathcal{F}$ be a relative torsion-free sheaf of rank $r$ on $X_{S}$. For any $T \to S$, there is a natural bijection between the following two sets
\begin{enumerate}[(A)]
    \item Sections of the structure morphism $\text{Red}_{G}(\mathcal{F}|_{X_T}) \to X_{T}$.
    \item Sections of the structure morphism $\text{Red}_{G}(\mathcal{F}) \times_{X_S} X_T \to X_{T}$.
\end{enumerate}
\end{lemma}
\begin{proof}
Let $U \subset X_S$ denote a big open subset such that $\mathcal{F}|_{U}$ is locally-free. By Lemma \ref{lemma: lemma on very big open subsets} (c), sections in the set (A) are in correspondence with sections $U_{T} \to \text{Red}_{G}(\mathcal{F}|_{X_T})_{U_{T}}$ defined on the big open $U_{T} \subset X_{T}$. Since the base-change morphism $\text{Red}_{G}(\mathcal{F}) \times_{X_S} X_{T} \to \text{Red}_{G}(\mathcal{F}|_{X_T})$ restricts to an isomorphism over the open set $U_{T}$, these sections are in bijection with sections $U_{T} \to \text{Red}_{G}(\mathcal{F}) \times_{X_S} U_T$. Another application of Lemma \ref{lemma: lemma on very big open subsets} (c) shows that all such sections extend uniquely to a morphism $X_{T} \to \text{Red}_{G}(\mathcal{F}) \times_{X_S} X_T$ as in (B).
\end{proof}
In other words, even though the schemes $\text{Red}_{G}(\mathcal{F}) \times_{X_S} X_T$ and $ \text{Red}_{G}(\mathcal{F}|_{X_T})$ are not necessarily isomorphic, they have the same sections over $X_{T}$ because they are isomorphic over a big open subset. 

\begin{defn} A $\rho$-sheaf is a pair $(\mathcal{F}, \sigma)$ where
\begin{enumerate}[(1)]
    \item $\mathcal{F}$ is a family of torsion-free sheaves of rank $r$ on $X_S$.
    \item $\sigma$ is a section of the structure morphism $\Red_{G}(\mathcal{F}) \rightarrow X_S$.
\end{enumerate}
The stack $\Bun_{\rho}(X)$ is the pseudofunctor from $\left(\Aff_{k}\right)^{op}$ into groupoids that sends an affine scheme $T \in \Aff_{k}$ to $\Bun_{\rho}(X)\, (T) = \left[ \; \; \text{groupoid of $\rho$-sheaves $(\mathcal{F}, \sigma)$ on $X_S$} \; \; \; \right]$.
\end{defn}
Once we know Lemma \ref{lemma: key lemma positive char}, everything in Section \ref{section: rho-sheaves} applies verbatim. Similarly, the definitions and arguments in Subsections \ref{subsection: numerical invariant}, \ref{subsection: affine grassmannians}, \ref{subsection: monotonicity properties of the numerical invariant} and \ref{subsection: hn boundedness} apply without a change. Hence we have the following result in arbitrary characteristic
\begin{prop} \label{prop: monotonicity and hn boundedness positive char}
The polynomial numerical invariant $\nu$ on the stack $\text{Bun}_{\rho}(X)$ is strictly $\Theta$-monotone, strictly $S$-monotone, and satisfies the HN boundedness condition. \qed
\end{prop}

Proposition \ref{prop: monotonicity and hn boundedness positive char} yields the necessary hypotheses to apply \cite[Thm. 2.26]{torsion-freepaper} (or \cite[Thm. B]{halpernleistner2021structure}). This yields the following theorem (see \cite[Defn. 2.23]{torsion-freepaper} for a definition of weak $\Theta$-stratification).
\begin{thm} \label{thm: weak theta stratification positive characteristic}
The polynomial numerical invariant $\nu$ defines a weak $\Theta$-stratification on the stack $\text{Bun}_{\rho}(X)$. \qed
\end{thm}

Note that in this case we only get a weak $\Theta$-stratification. In particular, the corresponding notion of leading term HN filtration of a $\rho$-sheaf $(\mathcal{F}, \sigma)$ is not necessarily defined over the field $K$ of definition of $(\mathcal{F},\sigma)$; one might need to allow to pass to a finite purely inseparable extension of $K$. This is to be expected for small characteristics, because of the failure of Behrend's conjecture even in the case when $X$ is a curve \cite{heinloth-behrends-conjecture}. We note that it is still the case that the leading HN filtration is unique up to scaling, even though it might not be defined over the ground field.

The analogous definition of relative leading HN filtrations (Definition \ref{defn: leading term filtration}) applies in arbitrary characteristic. However, since we are working with a weak $\Theta$-stratification, the universal stratification $H \to S$ is no longer a locally closed immersion. Instead, we get that the universal morphism $H \to S$ parametrizing relative leading term HN filtrations is surjective, radicial and of finite type. Each connected component of $H$ is finite over its image.

As long as the characteristic of $k$ is not too small, the proofs in Section \ref{section: properness} can still be applied to show Theorem \ref{thm: valuative criterion for properness rho sheaves}. More precisely, we just need that condition (2) in Proposition \ref{prop: extension after ramified cover} holds.
\begin{thm}
Choose an algebraic closure $\overline{k}$ of $k$. Let $Z \subset G_{\overline{k}}$ denote the maximal central torus of the base-change $G_{\overline{k}}$. Suppose that the characteristic of $k$ does not divide the order of the kernel of the simply-connected cover $\widetilde{G_{\overline{k}}/Z} \to G_{\overline{k}}/Z$ of the semisimple group $G_{\overline{k}}/Z$. Then the morphism $\text{Bun}_{\rho}(X) \to \Spec(k)$ satisfies the existence part of the valuative criterion for properness \cite[\href{https://stacks.math.columbia.edu/tag/0CLK}{Tag 0CLK}]{stacks-project} in the case of complete discrete valuation rings (as in Theorem \ref{thm: valuative criterion for properness rho sheaves}). \qed
\end{thm}

The definitions and constructions in Sections \ref{section: generalization} and \ref{section: GHN filtrations} apply similarly. The only change is that the GHN filtration of a $\rho$-sheaf $(\mathcal{F}, \sigma)$ need not be defined over the field $K$ of definition, but again one might need to pass to a finite purely inseparable extension. The analogue of Theorem \ref{thm: relative Gieseker filtrations} still holds; the only thing that needs to be modified is to replace the locally closed stratification $H \to S$ with a surjective radicial morphism of finite type.

Finally, we note that the proof of Proposition \ref{prop: comparison mixed gieseker filtration vs slope filtration} applies to arbitrary characteristic. This means that if the representation $\rho$ is central, then the stratification of $\text{Bun}_{\rho}(X)$ by radicial morphisms induced by our Gieseker-Harder-Narasimhan filtrations refines the theory developed in \cite{gurjar2020hardernarasimhan} in arbitrary characteristic.

\begin{remark}
Using the same techniques, one could generalize the results in this appendix to the case when $X \to S$ is a smooth projective morphism with geometrically connected fibers over an arbitrary base-scheme $S$, $G \to S$ is a reductive group scheme with connected fibers, and $\rho$ is a faithful representation $\rho: G \hookrightarrow \prod_{i} GL(V^i)$ for some vector bundles $V^i$ on $S$. We have chosen not to pursue this level of generality in this paper.
\end{remark}
\end{section}

\footnotesize
\bibliography{moduli_singular_bundles.bib}
\bibliographystyle{alpha}

\end{document}